\documentclass[11pt]{article}

\usepackage{mathrsfs}
\usepackage{kbordermatrix}
\usepackage{enumerate}
\usepackage[section]{algorithm}
\usepackage{algorithmic}
\usepackage{multirow}
\usepackage{xcolor}

\usepackage{graphicx}
\usepackage{amsmath,amsfonts,amsthm,amsbsy,amssymb}
\usepackage{fixmath}
\usepackage{amssymb,latexsym}

\usepackage{hyperref}

\usepackage{cleveref}

\def\by{\pmb{y}}

\def\bbC{\mathbb{C}}
\def\bbI{\mathbb{I}}

\def\bbR{\mathbb{R}}

\def\scrB{\mathscr{B}}

\def\scrG{\mathscr{G}}
\def\scrH{\mathscr{H}}

\def\scrR{\mathscr{R}}

\def\cR{{\cal R}}

\def\td{{\sc td}}

\def\ptbd{{\sc ptbd}}
\def\ptsvd{{\sc ptsvd}}

\newcommand\STM[2]{{\rm St}(#1,#2)}

\def\wtd{\widetilde}
\def\what{\widehat}

\usepackage{accents}

\DeclareMathOperator{\BDiag}{BDiag}

\DeclareMathOperator{\diag}{diag}
\DeclareMathOperator{\dist}{dist}

\DeclareMathOperator*{\opt}{opt}
\DeclareMathOperator{\rank}{rank}
\DeclareMathOperator{\sym}{sym}
\DeclareMathOperator{\tr}{tr}
\DeclareMathOperator{\ufd}{ufd}

\DeclareMathOperator{\F}{F}
\DeclareMathOperator{\HH}{H}
\DeclareMathOperator{\T}{T}

\DeclareMathOperator{\KKT}{KKT}

\def\scrR{\mathscr{R}}

\def\Red{\textcolor{red}}

\newtheorem{theorem}{Theorem}[section]
\newtheorem{lemma}{Lemma}[section]

\theoremstyle{definition}

\newtheorem{remark}{Remark}[section]

\allowdisplaybreaks

\numberwithin{equation}{section}
\numberwithin{figure}{section}
\numberwithin{table}{section}

\def\sss{\scriptstyle}

\title{An NPDo Approach for Tensor Block-Diagonalization}

\author{Ren-Cang Li%
\thanks{Department of Mathematics, University of Texas at Arlington, Arlington, TX 76019-0408, USA.
        Supported in part by NSF DMS-2407692.
        Email: {\tt rcli@uta.edu}.}
\and
Li Wang%
\thanks{Department of Mathematics, University of Texas at Arlington, Arlington, TX 76019-0408, USA.
        Supported in part by NSF DMS-2407692.
        Email: {\tt li.wang@uta.edu}.}
\and
Mei Yang%
\thanks{Department of Mathematics, University of Texas at Arlington, Arlington, TX 76019-0408, USA.
        Email: {\tt mei.yang@uta.edu}.}
}

\date{
      May 9, 2026
}

\begin{document}

\maketitle

\begin{abstract}
This paper is concerned with {\em Partial Tensor Block-Diagonalization\/} of a multiway tensor
by orthonormal matrices so that the extracted block-diagonal part optimally represents  the tensor.
The basic idea is to maximize the block-diagonal part via the tensor's mode-multiplications
by orthonormal matrices.
For that reason, it will be referred to {\em Principal Tensor Block-Diagonalization} (\ptbd), which contains
the Tucker decomposition (\td) of a tensor as a special case with just one block. Also as a special case is
the approximate dominant  tensor SVD in which
each block-size is 1-by-1.
An NPDo approach is proposed to optimize the block-diagonal part for computing \ptbd.
It is shown the NPDo approach combined with Gauss-Seidel-type updating is globally convergent to a stationary point
while the objective increases monotonically.
Numerical experiments are presented to illustrate the efficiency of the NPDo approach.


\end{abstract}

\clearpage
\tableofcontents

\clearpage
\section{Introduction}\label{sec:intro}
A matrix is sometimes referred as a 2-mode or 2-way array and in general a tensor is an array of more than 2 modes or 2 ways \cite{bako:2025,dldv:2000a,elde:2007,koba:2009}. In various applications, data is naturally collected as tensors that need to be
analyzed to gain knowledge. Most tensor decompositions of 3 modes or more are extensions from corresponding ones in
matrix analysis and can only be done approximately. In this paper, we will  present an NPDo approach for computing dominant tensor block-diagonalization.

Let $B\equiv[b_{i_1i_2\cdots i_m}]\in\bbC^{n_1\times n_2\times\cdots\times n_m}$ be an $m$-mode tensor. We will adopt a few operations on tensors: the Frobenius norm, unfolding, and mode-multiplication, as introduced in \cite{dldv:2000a,elde:2007}.  The tensor Frobenius norm of $B$ is
$$
\|B\|_{\F}=\Big(\sum_{i_1,i_2,\ldots,i_m}|b_{i_1i_2\cdots i_m}|^2\Big)^{1/2}.
$$
Let $N=n_1n_2\cdots n_m$.
There are $m$ different unfolding matrices from the $m$-mode tensor $B$. Following the informal definition
\cite[Definition~I.8]{bako:2025} and denoting by $B_{\ufd,\ell}\in\bbC^{n_{\ell}\times N/n_{\ell}}$ the $\ell$th unfolding matrix,
we define the $(i_1\cdots i_{\ell-1}\,i_{\ell}\cdots i_m)$th column of $B_{\ufd,\ell}$ as
\begin{equation}\label{eq:ufd,j}
\begin{bmatrix}
  b_{i_1\cdots i_{\ell-1}\,1\,i_{\ell}\cdots i_m} \\
  b_{i_1\cdots i_{\ell-1}\,2\,i_{\ell}\cdots i_m} \\
  \vdots \\
  b_{i_1\cdots i_{\ell-1}\,n_{\ell}\,i_{\ell}\cdots i_m},
\end{bmatrix}
\end{equation}
i.e., the $\ell$th subscript runs from $1$ to $n_{\ell}$ while all others stay fixed. As each subscript $i_j$ $(j\ne\ell$)
varies from $1$ to $n_j$, $B_{\ufd,\ell}$ has a total of $N/n_{\ell}$ columns. There is a question how these columns
should be ordered to form $B_{\ufd,\ell}$. For the interest of this paper,
the order of arranging these columns does not matter so long it is done in the same way every time
the $\ell$th unfolding matrix is formed.
The $\ell$th mode-multiplication by $B$ is defined as
\begin{equation}\label{eq:j-times}
(B\times_{\ell} X)_{(i_1i_2\cdots i_m)}=\sum_{j=1}^{n_{\ell}}x_{i_{\ell}\,j}b_{i_1\cdots i_{\ell-1}\,j\,i_{\ell+1}\cdots i_m}
\quad\mbox{for $X\equiv[x_{ij}]\in\bbC^{k\times n_{\ell}}$},
\end{equation}
to yield $B\times_{\ell} X\in\bbC^{n_1\times\cdots\times n_{\ell-1}\times k\times n_{\ell+1}\times\cdots\times n_m}$,
where $(\,\cdot\,)_{(i_1i_2\cdots i_m)}$ refers to the $(i_1,i_2,\ldots, i_m)$th entry of a tensor.
Denote by
$$
\STM{k}{n}=\{P\in\bbC^{n\times k}\,:\,P^{\HH}P=I_k\}\subset\bbC^{n\times k},
$$
the complex Stiefel manifold, where $1\le k\le n$.

By {\em dominant tensor block-diagonalization}, we mean to seek $m$ orthonormal matrices
$P_{\ell}\in\STM{k_{\ell}}{n_{\ell}}$ for $1\le \ell\le m$
so that
\begin{equation}\label{eq:C}
T:=B\times_1 P_1^{\HH}\times_2 P_2^{\HH}\cdots\times_m P_m^{\HH}\in\bbC^{k_1\times k_2\times\cdots\times k_m}
\end{equation}
nontrivially concentrates on its block-diagonal components and it optimally assumes  part of the total mass of the tensor.
Necessarily, $1\le k_{\ell}\le n_{\ell}$ for $1\le\ell\le m$.
We now make it concrete what we meant by that $T$ concentrates on its block-diagonal components
and at the same time and assumes part of the total mass of the tensor optimally.
Let
\begin{equation}\label{eq:tau(ki)-n}
\tau_{k_{\ell}}=(k_{\ell 1},\dots,k_{\ell t}),\,\,
\mbox{integer $k_{\ell j}\ge 1$}, \,\,
k_{\ell}:=\sum_{j=1}^t k_{\ell j}\le n_{\ell},\quad
\mbox{for $1\le\ell\le m$}
\end{equation}
which we will call a {\em partition\/} of $k_{\ell}$, where integer $t\ge 1$.
We partition $T$ according to $\tau_{k_{\ell}}$ along each of the $m$ ways and view
$T$ as a $t\times \cdots\times t$ $m$-mode block tensor.
The {\em $(\tau_{k_1},\ldots,\tau_{k_m})$-block diagonal part\/} of $T$ is defined as and denoted by
\begin{equation}\label{eq:BDiag-dfn}
\BDiag_{(\tau_{k_1},\ldots,\tau_{k_m})}(T)=\diag(T_{111}, T_{222},\ldots, T_{ttt}),
\end{equation}
where
$$
T_{111}=T_{(1:k_{11},1:k_{21},\ldots,1:k_{m1})}, \quad
T_{222}=T_{(k_{11}+1:k_{11}+k_{12},k_{21}+1:k_{21}+k_{22},\ldots,k_{m1}+1:k_{m1}+k_{m2})},
$$
and similarly for other $T_{iii}$
of size $k_{1i}\times\cdots\times t_{mi}$.
The tensor $T$ is referred to as {\em a $(\tau_{k_1},\ldots,\tau_{k_m})$-block diagonal tensor\/} if all the blocks in $T$, except diagonal blocks $T_{iii}$,
are exactly zeros.
With notation $\BDiag_{(\tau_{k_1},\ldots,\tau_{k_m})}$ just introduced,
the dominant tensor block-diagonalization is to simply maximize
$\|T\|_{\F}$ subject to $P_{\ell}\in\STM{k_{\ell}}{n_{\ell}}\,\,\forall\ell$.

Our goal in this paper is to seek a principal tensor block diagonalization (\ptbd) in the sense of finding
an optimal tuple $(P_1,\ldots,P_m)$ with $P_{\ell}\in\STM{k_{\ell}}{n_{\ell}}\,\,\forall\ell$
such that
$T$ given by \eqref{eq:C}
is optimally close to being $(\tau_{k_1},\ldots,\tau_{k_m})$-block diagonal
and at the same time, the $(\tau_{k_1},\ldots,\tau_{k_m})$-block diagonal part
$\BDiag_{(\tau_{k_1},\ldots,\tau_{k_m})}(T)$ are maximized so that
\begin{equation}\label{eq:BTD}
B\approx \BDiag_{(\tau_{k_1},\ldots,\tau_{k_m})}(T)
        \times_1 P_1\times_2 P_2\cdots\times_m P_m
\end{equation}
optimally. There are three special cases: 1) For the case $\tau_{k_{\ell}}=(1,1,\ldots,1)$ for $1\le\ell\le m$,
\eqref{eq:BTD} can be regard as a type of tensor SVD \cite{chsa:2009}, a natural extension of matrix SVD
\cite{demm:1997,govl:2013}, although exact tensor SVD does not exist generically; 2) For the case $k_{\ell}=1$ for $1\le\ell\le m$,
it is the so-called rank-one approximation \cite{zhgo:2001}; 3) For the case $t=1$ and thus $\tau_{k_{\ell}}=(k_{\ell})$ for $1\le\ell\le m$,
it is the Tucker decomposition (TD) \cite{tuck:1966}.

With $T$ partitioned according to $\tau_{k_{\ell}}$ as mentioned moments ago, approximation \eqref{eq:BTD} can be written
concisely as
\begin{equation}\label{eq:BTD'}
B\approx \sum_{i=1}^t T_{iii}
        \times_1 P_{1i}\times_2 P_{2i}\cdots\times_m P_{mi},
\end{equation}
where $P_{\ell i}$ for $1\le i\le t$ are obtained from partitioning $P_{\ell}$ according to $\tau_{k_{\ell}}$ for $1\le\ell\le m$:
$$
P_{\ell}=[P_{\ell 1},\ldots,P_{\ell t}]
\quad\mbox{with}\quad
P_{\ell i}\in\bbC^{n_{\ell}\times k_{\ell i}}.
$$
When all $T_{iii}$ is 1-by-1-by-1, it looks like a canonical polyadic decomposition \cite{bako:2025,koba:2009} but not quite the same because $P_{\ell}=[P_{\ell 1},\ldots,P_{\ell t}]\,\,\forall\ell$  are orthonormal.
In general, without orthonormality assumption
on $P_{\ell}$,  \eqref{eq:BTD'} had appeared in the literature \cite{dela:2008a,dela:2008b,dlni:2008},
\cite[section 5.7]{koba:2009} (and more references therein).
However, in this paper, we are looking at \eqref{eq:BTD'} through the perspective of partially and optimally block-diagonalizing a tensor along the line of matrix SVD, and for that reason we will impose the orthonormality assumption
on each $P_{\ell}$.

The optimality in approximation \eqref{eq:BTD} is determined by
\begin{subequations}\label{eq:opt-PBTD}
\begin{equation}\label{eq:opt-PBTD'}
\max_{P_{\ell}\in\STM{k_{\ell}}{n_{\ell}}\,\,\forall\ell} \, f(P_1,\ldots,P_m),
\end{equation}
where
\begin{equation}\label{eq:opt-PBTD-obj}
f(P_1,\ldots,P_m)=\left\|\BDiag_{(\tau_{k_1},\ldots,\tau_{k_m})}
     \big(B\times_1 P_1^{\HH}\times_2 P_2^{\HH}\cdots\times_m P_m^{\HH}\big)\right\|_{\F}^2.
\end{equation}
\end{subequations}
It includes the TSVD problem \cite{chsa:2009} as a special case.
It can be seen that \eqref{eq:opt-PBTD} is equivalent to
$$
\min_{P_{\ell}\in\STM{k_{\ell}}{n_{\ell}}\,\,\forall\ell} \,
   \|B-\BDiag_{(\tau_{k_1},\ldots,\tau_{k_m})}(T)\times_1 P_1\times_2 P_2\cdots\times_m P_m\|_{\F}^2.
$$
As a result, it also encompasses previous optimization models for TD \cite[section~6.4]{bako:2025} and
the rank-one approximation \cite{zhgo:2001}. If there is any $k_{\ell i}>1$ in \eqref{eq:tau(ki)-n}, optimizer
tuple $(P_1,\ldots,P_m)$ is non-unique in nature.


{\bf Contributions.}
We will show that the KKT condition for  \eqref{eq:opt-PBTD},  at optimality, leads to a system of
$m$ coupled polar decompositions of the matrix-valued functions that depend on their orthonormal polar factors.
We will develop an alternating NPDo approach for its numerical solution, where the term
NPDo stands for nonlinear polar decomposition with
orthonormal\footnote {The word ``orthogonal'' was used in \cite{li:2024} due to that the presentation there
  was in terms of the real number field. It is switched  to ``orthonormal'' in \cite{lilw:2026} which is generic
  regarding whether the real or complex complex fields.} polar factor dependency.
It is shown the NPDo approach combined with Gauss-Seidel-type updating is globally convergent to a stationary point
while the objective $f$ increases monotonically.
The NPDo approach when specialized to TD (i.e., $t=1$) becomes a HOOI variant \cite{dldv:2000b}.
Hence our global convergence result applies to HOOI straightforwardly.
Much of the developments bears similarity to our recent work on principal joint SVD-type block diagonalization \cite{liwy:2026jsvd:arXiv}.

The rest of this paper is organized as follows. In \cref{sec:NPDo}, we review the NPDo approach in \cite{wazl:2022a}
for maximizing  sum of coupled traces. It will serve as the workhorse for the alternating NPDo approach for
\ptbd\ in \cref{sec:TBD}. In \cref{sec:cvg-NPDo4PTBD} we establish our convergence results for
the alternating NPDo approach. Numerical experiments are presented in \cref{sec:egs} to demonstrate the effectiveness of
our approach. Finally, conclusions are drawn in \cref{sec:concl}.

{\bf Notation}.
We follow the following notation convention throughout this paper.
  $\bbR^{m\times n}$  is the set of $m\times n$ real matrices,  $\bbR^n=\bbR^{n\times 1}$, and $\bbR=\bbR^1$,
        and similarly $\bbC^{m\times n}$,  $\bbC^n$, and $\bbC$ except for the complex numbers.
$I_n\in\bbR^{n\times n}$ is the identity matrix or simply $I$ if its size is clear from the context.
For a matrix/vector $X$,
  $X^{\T}$ and $X^{\HH}$ stand for its transpose and complex conjugate transpose, respectively.
For $X\in\bbC^{m\times n}$, $\|X\|_2$, $\|X\|_{\F}$, and $\|X\|_{\tr}$
are its spectral norm, Frobenius norm, and trace norm (also known as the nuclear norm), respectively, and
$\sigma_{\min}(X)$ is its smallest singular
value\footnote {It is understood that $X\in\bbC^{m\times n}$ has $\min\{m,n\}$
     singular values.}.
$\cR(X)$ is the column space of a matrix $X$, spanned by its columns. For $X\in\bbC^{n\times n}$, $\tr(X)$ is its trace.
  A matrix $A\succ 0\, (\succeq 0)$ means that it is Hermitian and positive definite (semi-definite), and
        accordingly
        $A\prec 0\, (\preceq 0)$ if $-A\succ 0\, (\succeq 0)$.

\section{
         Sum of Coupled Traces}\label{sec:NPDo}
Our working engine for efficiently solving \eqref{eq:opt-PBTD} is the NPDo approach for maximizing the sum of coupled traces
that was originally investigated in \cite{bomt:1998,wazl:2022a} and has since been much extended (see \cite[Example 5.3]{li:2024}).
We consider
\begin{subequations}\label{eq:OptOnSTM-master0}
\begin{equation}\label{eq:OptOnSTM-master0a}
\max_{X\in\STM{k}{n}} \Big\{ f(X):=\sum_{i=1}^M\tr\big(X_i^{\HH}A_iX_i\big)\Big\},
\end{equation}
where
\begin{equation}\label{eq:OptOnSTM-master0b}
      \framebox{
      \parbox{12.0cm}{
      $A_i\succeq 0$ for $1\le i\le M$, $X_i\in\bbC^{n\times k_i}$ for $1\le i\le M$ are
submatrices consisting of a few or all columns of $X$ (in particular, sharing
common columns of $X$ by different $X_i$ is allowed).
      }}
\end{equation}
\end{subequations}
But we point out that the individual assumptions in \cite{bomt:1998,wazl:2022a}, respectively, are stronger:
we are not assumed here:
$M=k$ and $X_i$ is the $i$th column of $X$ in \cite{bomt:1998}, while in \cite{wazl:2022a}
$X=[X_1,X_2,\ldots,X_M]$, i.e., submatrices $X_i$ are from partitioning $X$ column-wise. From
the perspective of the general NPDo theory \cite{li:2024}, $f(X)$ is a special convex composition
of atomic functions $\tr\big(X_i^{\HH}A_iX_i\big)$, and also a special case of
the problem \cite[(2.1)]{lilw:2026}.

It can be seen that each $X_i$ can be expressed as
\begin{equation}\label{eq:Xi=XJi}
X_i=XJ_i\quad\mbox{for $1\le i\le M$},
\end{equation}
where
each $J_i\in\bbR^{k\times k_i}$ is a submatrix of $I_k$, consisting of those columns of $I_k$
with the same column indices as $X_i$ to $X$.
The objective $f$ is a real-valued function on the complex matrix variable $X\in\bbC^{n\times k}$.
Its Euclidean gradient is defined to be the unique matrix $\nabla f(P) \in \bbC^{n\times k}$
such that for $E\in \bbC^{n\times k}$ with sufficiently small $\|E\|_2$ \cite[section~2]{lilw:2026}
\begin{equation}\label{eq:f(P+E)}
f(P+E)=f(P)+\langle \nabla f(P),E  \rangle
          +o(\|E\|_2),
\end{equation}
where $\langle \cdot,\cdot  \rangle$ denotes the standard inner product on $\bbC^{n\times k}$:
$\langle X,Y  \rangle = \Re\big(\tr(Y^{\HH}X)\big)$ and
$\Re(\cdot)$ takes the real part of a complex number.
 We have
\begin{equation}\label{eq:scrH-NPDo}
\scrH(X):= 
    \nabla f(X)
    =\sum_{i=1}^M\big[\nabla \tr\big(X_i^{\HH}A_iX_i\big)\big]\,J_i^{\T}
    =\sum_{i=1}^M2A_iX_iJ_i^{\T}\in\bbC^{n\times k}.
\end{equation}
With \eqref{eq:scrH-NPDo}, the KKT condition for \eqref{eq:OptOnSTM-master0} can be stated as \cite[section~2]{li:2024}
\begin{equation}\label{eq:KKT-master0}
\scrH(X)=X\Lambda
\quad \mbox{with} \quad
\Lambda^{\HH}=\Lambda\in\bbC^{k\times k},\quad X\in\STM{k}{n}.
\end{equation}
A critical property for the objective function $f(X)$ of \eqref{eq:OptOnSTM-master0} is
\begin{equation}\label{eq:NPDo-satisfied}
\framebox{
\parbox{13cm}{\em
{\rm \cite{wazl:2022a}}
Given $X,\,\what X\in\bbC^{n\times k}$,
if
$\Re(\tr(\what X^{\HH}\scrH(X)))\ge\tr(X^{\HH}\scrH(X))+\eta$
for some $\eta\in\bbR$,
then $f(\what X)\ge f(X)+\eta$, where $\Re(\cdot)$ takes the real part of a complex number.
}}
\end{equation}
In the NPDo theory \cite{li:2024}, the statement in \eqref{eq:NPDo-satisfied} says that  the objective function $f(X)$
satisfies {\bf the NPDo Ansatz} postulated there \cite[p.188]{li:2024}. Two important implications of
\eqref{eq:NPDo-satisfied} are: 1) at any maximizer $X$, \eqref{eq:KKT-master0} is satisfied
and at the same time $\Lambda\succeq 0$, i.e., it is the polar decomposition of $\scrH(\cdot)$
evaluated at $X$, and 2) the following SCF (self-consistent-field) iteration
\begin{equation}\label{eq:SCF-form:NPDo:intro}
      \framebox{
      \parbox{12.0cm}{
      SCF for NPDo \eqref{eq:KKT-master0}: given $X^{(0)}\in\STM{k}{n}$, iteratively
      compute polar decomposition
      $\scrH(X^{(j-1)})=X^{(j)}\Lambda_j$ of $\scrH(X^{(j-1)})$ for $X^{(j)}\in\STM{k}{n}$.
      }}
\end{equation}
is convergent \cite[Theorems~3.2 and 3.3]{li:2024}.
The first implication leads to the name {\em nonlinear polar decomposition with orthonormal polar factor dependency\/} (NPDo).
In general $k\ll n$ when $n$ is large, and the polar decomposition in SCF \eqref{eq:SCF-form:NPDo:intro}
should be computed by the thin SVD: $\scrH(X^{(j-1)})=Q_j\Sigma_jS_j^{\HH}$ where $Q_j\in\STM{k}{n}$ and
$S_j\in\STM{k}{n}$, and then $X^{(j)}=Q_jS_j^{\HH}$.

\begin{remark}\label{rk:polar-nonuniqueness}
It is important to note that
the orthonormal polar factor $X^{(j)}$ in \eqref{eq:SCF-form:NPDo:intro} is unique if and only if
$\rank(\scrH(X^{(j-1)}))=k$ \cite{li:1993b,luli:2024};
otherwise there is some arbitrariness in $X^{(j)}$. However, the uncertainty
caused by $\rank(\scrH(X^{(j-1)}))<k$ does not prevent the objective from moving up. In fact, we have
\cite[Lemma 4.2]{li:2026}
$$
f(X^{(j)})\ge f(X^{(j-1)})+\underbrace{\big[\big\|\scrH(X^{(j-1)})\big\|_{\tr}-\Re\tr\big(\big[X^{(j-1)}\big]^{\HH}\scrH(X^{(j-1)})\big)\big]}_{=:\eta}
$$
where $\eta\ge 0$ and $\eta=0$ if and only if $\scrH(X^{(j-1)})=X^{(j-1)}\big(\big[X^{(j-1)}\big]^{\HH}\scrH(X^{(j-1)})\big)$
is already a polar decomposition.
\end{remark}


\section{Principal Tensor Block Diagonalization}\label{sec:TBD}
Consider the $m$-mode tensor $B\equiv[b_{i_1i_2\cdots i_m}]\in\bbC^{n_1\times n_2\times\cdots\times n_m}$.
As we explained in \cref{sec:intro}, we seek $P_{\ell}\in\STM{k_{\ell}}{n_{\ell}}\,\,\forall\ell$
so that
$T=B\times_1 P_1^{\HH}\times_2 P_2^{\HH}\cdots\times_m P_m^{\HH}$ is  approximately a $(\tau_{k_1},\ldots,\tau_{k_m})$-block diagonal tensor and, at the same time, assumes part of the total mass of the tensor optimally. To that end, we propose to solve the optimization problem \eqref{eq:opt-PBTD}.
We start by partitioning $P_{\ell}$ columnwise, according to \eqref{eq:tau(ki)-n}, as
\begin{equation}\label{eq:UVW-part'n}
P_{\ell}=\kbordermatrix{ &\sss k_{\ell 1} &\sss k_{\ell 2} &\sss \cdots &\sss k_{\ell t} \\
                  & P_{\ell 1} & P_{\ell 2} & \cdots & P_{\ell t}}\quad\forall\ell,
\end{equation}
and then expand the objective $f(P_1,\ldots,P_m)$ of \eqref{eq:opt-PBTD} to
\begin{equation}\label{eq:f(UVW)-sum}
f(P_1,\ldots,P_m)=\sum_{s=1}^t\big\|B\times_1 P_{1s}^{\HH}\times_2 P_{2s}^{\HH}\cdots\times_m P_{ms}^{\HH}\big\|_{\F}^2.
\end{equation}
Once an optimal tuple $(P_1,P_2,\ldots,P_m)$ of \eqref{eq:opt-PBTD}
is computed, the corresponding smaller tensor $T$ can be recovered by
$T=B\times_1 P_1^{\HH}\times_2 P_2^{\HH}\cdots\times_m P_m^{\HH}$ to yield
a principal block diagonalization \eqref{eq:BTD} of tensor $B$ upon extracting its block-diagonal part
$\BDiag_{(\tau_{k_1},\ldots,\tau_{k_m})}(T)$.

It can be seen that, for any $Q_{\ell i}\in\STM{k_{\ell i}}{k_{\ell i}}$ for $1\le\ell\le m$ and $1\le i\le t$,
\begin{multline*}
\big\|\BDiag_{(\tau_{k_1},\ldots,\tau_{k_m})}\big(B\times_1 (P_1Q_1)^{\HH}\times_2 (P_2Q_2)^{\HH}
              \cdots\times_m (P_mQ_m)^{\HH}\big)\big\|_{\F}^2 \\
\equiv\big\|\BDiag_{(\tau_{k_1},\ldots,\tau_{k_m})}\big(B\times_1 P_1^{\HH}\times_2 P_2^{\HH}
              \cdots\times_m P_m^{\HH}\big)\big\|_{\F}^2,
\end{multline*}
where $Q_{\ell}=\diag(Q_{\ell 1},Q_{\ell 2},\ldots,Q_{\ell t})\in\STM{k_{\ell}}{k_{\ell}}$.
Hence the maximizer tuple $(P_1,\ldots,P_m)$ of \eqref{eq:opt-PBTD} is inherently non-unique if any one of $k_{\ell i}>1$.
For that reason, it is not so appropriate to seek approximate maximizer tuples from an iterative method
to converge entrywise but rather to increasingly satisfy its KKT condition as a whole as the iteration processes.
This observation is important when it comes to perform convergence analysis of our method later
in \cref{sec:cvg-NPDo4PTBD}.

\subsection{The  NPDo Approach -- Alternating SCF}\label{ssec:NPDo-alt:PBTD}
In the rest of this paper, for integer tuple $(\ell,i_1,i_2,\ldots,i_{m-1})$ appears in the same paragraph, it is understood that it is a permutation of $(1,2,\ldots,m)$ such that $i_1<i_2<\cdots<i_{m-1}$, i.e.,
$$
(\ell,i_1,i_2,\ldots,i_{m-1})=(\ell,1,\ldots,\ell-1,\ell+1,\ldots,m).
$$
There are only $m$ such tuples, and in the case of $m=3$, there are
$(1,2,3)$, $(2,1,3)$, and $(3,1,2)$.

Define, for any $(\ell,i_1,i_2,\ldots,i_{m-1})$ as specified moments ago and $1\le s\le t$,
\begin{equation}\label{eq:Cs}
C_{\ell s}(P_{i_1s},\ldots,P_{i_{m-1}s})
    :=\big[B\times_{i_1} P_{i_1s}^{\HH}\cdots\times_{i_{m-1}} P_{i_{m-1}s}^{\HH}\big]_{\ufd,\ell}
    \in\bbC^{n_{\ell}\times (k_{i_1s}\cdots k_{i_{m-1\,s}})}.
\end{equation}
Alternatively, the objective $f(P_1,\ldots,P_m)$ of \eqref{eq:opt-PBTD} can be expressed as
\begin{align}
f(P_1,\ldots,P_m)
    &=\sum_{s=1}^t\big\|P_{\ell s}^{\HH}
           C_{\ell s}(P_{i_1s},\ldots,P_{i_{m-1}s})\big\|_{\F}^2\nonumber \\
    &=\sum_{=1}^t\tr\big(P_{\ell s}^{\HH}H_{\ell s}(P_{i_1s},\ldots,P_{i_{m-1}s})P_{\ell s}\big),
          \label{eq:opt-PBTD:obj-tr}
\end{align}
where $H_{\ell s}(P_{i_1s},\ldots,P_{i_{m-1}s})\succeq 0$ given by
\begin{equation}\label{eq:Hi}
H_{\ell s}(P_{i_1s},\ldots,P_{i_{m-1}s})
  :=[C_{\ell s}(P_{i_1s},\ldots,P_{i_{m-1}s})][C_{\ell s}(P_{i_1s},\ldots,P_{i_{m-1}s})]^{\HH}\in\bbC^{n_{\ell}\times n_{\ell}}.
\end{equation}
In \eqref{eq:opt-PBTD:obj-tr}, there are $m$ different reformulations of the objective $f(P_1,\ldots,P_m)$
in terms of matrix traces, dependent on the unfolding $[\cdot]_{\ufd,\ell}$, and all $m$ take the form as
the sum of coupled traces investigated in \cite{wazl:2022a}.
Let $\scrH_{\ell}(P_1,\ldots,P_m)\in\bbC^{n_{\ell}\times k_{\ell}}$ for $1\le\ell\le m$ be
the partial Euclidean gradient of $f(P_1,\ldots,P_m)$
with respective to $P_{\ell}$:
\begin{equation}\label{eq:opt-PBTD-partdiff}
\scrH_{\ell}(P_1,\ldots,P_m)
    =\kbordermatrix{ &\sss k_{\ell 1} &\sss \cdots &\sss k_{\ell t} \\
                  & H_{\ell 1}(P_{i_11},\ldots,P_{i_{m-1}1}))P_{\ell 1}
           & \cdots & H_{\ell t}(P_{i_1t},\ldots,P_{i_{m-1}1}))P_{\ell t}}.
\end{equation}
The KKT condition of \eqref{eq:opt-PBTD} consists of
\begin{equation}\label{eq:PBTD:KKT}
\scrH_{\ell}(P_1,\ldots,P_m) = P_{\ell}\Lambda_{\ell}, \quad
P_{\ell}\in\STM{k_{\ell}}{n_{\ell}}, \quad
\Lambda_{\ell}=\Lambda_{\ell}^{\HH}\in\bbC^{k_{\ell}\times k_{\ell}}
\quad\forall\ell.
\end{equation}
As a result of \eqref{eq:NPDo-satisfied} \cite{wazl:2022a}, we have

\begin{theorem}\label{thm:f(UVW)-Ansatz}
Denote by $\scrH_{\ell}$ the partial Euclidean gradient of $f(P_1,\ldots,P_m)$ with respect to $P_{\ell}$, as in \eqref{eq:opt-PBTD-partdiff},
and let $(P_1,\ldots,P_m)\in\bbC^{n_1\times k_1}\times\cdots\times\bbC^{n_m\times k_m}$.
For any $\what P_{\ell}\in\bbC^{n_{\ell}\times k_{\ell}}$, if
        \begin{equation}\label{eq:f(P1toPm)-Ansatz-1}
        \tr(\what P_{\ell}^{\HH}\scrH_{\ell}(P_1,\ldots,P_m))
           \ge\tr(P_{\ell}^{\HH}\scrH_{\ell}(P_1,\ldots,P_m))+\eta_{\ell}\quad\mbox{for some $\eta_{\ell}\in\bbR$},
        \end{equation}
then $f(P_1,\ldots,P_{\ell-1},\what P_{\ell},P_{\ell+1},\ldots,P_m)
                   \ge f(P_1,\ldots,P_{\ell-1},P_{\ell},P_{\ell+1},\ldots,P_m)+\eta_{\ell}$.
\end{theorem}

\begin{algorithm}[t]
\caption{NPDoPTBD: the alternating NPDo approach for solving \eqref{eq:opt-PBTD}.}
\label{alg:NPDo:PBTD}
\begin{algorithmic}[1]
\REQUIRE $m$-mode tensor $B\in\bbC^{n_1\times n_2\times\cdots\times n_m}$
         (and, accordingly, matrix-valued functions
         $\scrH_{\ell}(\cdots)$ in \eqref{eq:opt-PBTD-partdiff}),
         and initial $(P_1^{(0)},\ldots,P_m^{(0)})$ with $P_{\ell}^{(0)}\in\STM{k_{\ell}}{n_{\ell}}\,\,\forall\ell$;
\ENSURE  an approximate maximizer tuple of \eqref{eq:opt-PBTD}.
\FOR{$j=0,1,\ldots$ until convergence}
    \FOR{$\ell=1,2,\ldots, m$}
       \STATE  compute $P_{\ell}^{(j+1)}$ as the orthonormal polar factor of
               $\scrH_{\ell}(P_1^{(j+1)},\ldots,P_{\ell-1}^{(j+1)},P_{\ell}^{(j)},\ldots,P_m^{(j)})$;
    \ENDFOR
\ENDFOR
\RETURN the last $(P_1^{(j)},\ldots,P_m^{(j)})$.
\end{algorithmic}
\end{algorithm}

\Cref{thm:f(UVW)-Ansatz} is broader than what we need because the theorem does not require that all $P_j$ and $\what P_{\ell}$ to
have orthonormal columns.
Along the line of the proof of \cite[Theorem~3.1]{li:2024}, we can also get

\begin{theorem}\label{thm:maximizer-BTSVD}
Let $(P_{*1},\ldots,P_{*m})$ with $P_{*\ell}\in\STM{k_{\ell}}{n_{\ell}}\,\,\forall\ell$ be a global maximizer tuple of \eqref{eq:opt-PBTD}. Then the KKT condition
\eqref{eq:PBTD:KKT} holds with
$P_{\ell}=P_{*\ell}\,\,\forall\ell$, i.e.,
\begin{equation}\label{eq:PBTD:KKTatMAX}
\scrH_{\ell}(P_{*1},\ldots,P_{*m})=P_{*\ell}\Lambda_{\ell*}, \quad
   \Lambda_{\ell}=P_{*\ell}^{\HH}\scrH_{\ell}(P_{*1},\ldots,P_{*m})\succeq 0\quad\forall\ell.
\end{equation}
\end{theorem}

We outline the alternating NPDo approach for solving \eqref{eq:opt-PBTD} in \Cref{alg:NPDo:PBTD},
which is a  Gauss-Seidel-type updating scheme, as it
mimics the one commonly used in linear system solving \cite{demm:1997}.
A few of comments about implementing \Cref{alg:NPDo:PBTD} are in order.
\begin{enumerate}[(i)]
  \item Carefully examining the formula of $\scrH_{\ell}$ in \eqref{eq:opt-PBTD-partdiff} via
        \eqref{eq:Cs} and \eqref{eq:Hi}, we find that there is no need to explicitly form each
        $H_{\ell s}(P_{i_1s},\ldots,P_{i_{m-1}s})$ and then compute $H_{\ell s}(P_{i_1s},\ldots,P_{i_{m-1}s})P_{\ell s}$, but
        instead compute the latter according to
        $$
        [C_{\ell s}(P_{i_1s},\ldots,P_{i_{m-1}s})]\,\big([C_{\ell s}(P_{i_1s},\ldots,P_{i_{m-1}s})]^{\HH} P_{\ell s}\big).
        $$
  \item As is in \Cref{alg:NPDo:PBTD}, each loop on $j$ is just for one $j$. Consecutive $C_{\ell}^{(j)}$ with respective to $\ell$ and then
        to $j$ share common computations that should be taken advantage of for the sake of saving work. For example, for $m=3$,
        it is beneficial to do two consecutive $j$ at a time, as the following table suggests:
        $$
        \setlength{\tabcolsep}{4pt}
        \begin{tabular}{|llcl|}
          \hline
        $C_{1s}(P_{2s}^{(j)},P_{3s}^{(j)})$, & $C_{2s}(P_{1s}^{(j+1)},P_{3s}^{(j)})$
                   & share  & $B\times_3P_3^{(j)}$, \\
        $C_{3s}(P_{1s}^{(j+1)},P_{2s}^{(j+1)})$, & $C_{1s}(P_{2s}^{(j+1)},P_{3s}^{(j+1)})$
                   & share  & $B\times_2P_2^{(j+1)}$, \\
        $C_{2s}(P_{1s}^{(j+2)},P_{3s}^{(j+1)})$, & $C_{3s}(P_{1s}^{(j+2)},P_{2s}^{(j+2)})$
                   & share  & $B\times_1P_1^{(j+2)}$. \\
          \hline
        \end{tabular}
        $$
  \item At Line 3, the thin SVD provides an efficient way to compute the orthonormal polar factors.
        It is important to note that $P_{\ell}^{(j+1)}$, as an orthonormal polar factor may not be
        uniquely determined, similarly to what we commented in \Cref{rk:polar-nonuniqueness}.
        If that happens, any particular orthonormal polar factor is just as good as any others, while
        the objective will still go up as dictated by \Cref{thm:f(UVW)-Ansatz}.
  \item A natural and cheap stopping criterion is through checking if
        \begin{equation}\label{eq:obj-stop}
        |f_j-f_{j-1}|/f_j\le\epsilon_1,
        \end{equation}
        where $\epsilon_1$ is a preselected tolerance. However, it has a potential pitfall, namely
        false convergence, if there is a period of iterations where increases in objective value are extremely small. For safe guard,
        we may check the residual for the KKT condition \eqref{eq:PBTD:KKT}:
        \begin{equation}\label{eq:stop-BTSVD}
       \epsilon_{\KKT,j}:=
            \sum_{\ell=1}^m\frac {\big\|\scrH_{\ell}(P_1^{(j)},\ldots,P_m^{(j)})-P_{\ell}^{(j)}\Lambda_{\ell}(P_1^{(j)},\ldots,P_m^{(j)})\big\|_{\F}}
                {\|B\|_{\F}\|B_{\ufd,\ell}\|_2}\le\epsilon_2,
        \end{equation}
        where $\Lambda_{\ell}(P_1^{(j)},\ldots,P_m^{(j)})=\sym\big(\big[P_{\ell}^{(j)}\big]^{\HH}\scrH_{\ell}(P_1^{(j)},\ldots,P_m^{(j)})\big)$,
        $\sym(X):=(X+X^{\HH})/2$ for a square matrix $X$, and $\epsilon_2$ is another preselected tolerance.
        But computing 
        the left-hand side of \eqref{eq:stop-BTSVD} entails extra work since not all
        $\scrH_{\ell}(P_1^{(j)},\ldots,P_m^{(j)})$ are computed.
        As a comprise, at Line 1, we replace \eqref{eq:stop-BTSVD} with
        \begin{equation}\label{eq:stop-BTSVD'}
        \tilde\epsilon_{\KKT,j}:=\sum_{\ell=1}^m
        \frac {\|\scrH_{\ell}(P_1^{(j+1)},\ldots,P_{\ell-1}^{(j+1)},P_{\ell}^{(j)},\ldots,P_m^{(j)})-P_{\ell}^{(j)}\Omega_{\ell}^{(j)}\|_{\F}}
                 {\|B\|_{\F}\|B_{\ufd,\ell}\|_2}\le\epsilon_2,
        \end{equation}
        where $\Omega_{\ell}^{(j)}=\sym\big(\big[P_{\ell}^{(j)}\big]^{\HH}\scrH_{\ell}(P_1^{(j+1)},\ldots,P_{\ell-1}^{(j+1)},P_{\ell}^{(j)},\ldots,P_m^{(j)})\big)$.

\end{enumerate}
        Computing $\|B_{\ufd,\ell}\|_2$ needed in \eqref{eq:stop-BTSVD} and \eqref{eq:stop-BTSVD'}  can be nontrivial for large $n_{\ell}$, but fortunately for  normalization purpose some rough estimate
        is good enough, e.g., replacing $\|B_{\ufd,\ell}\|_2$ with $\sqrt{\|B_{\ufd,\ell}\|_1\|B_{\ufd,\ell}\|_{\infty}}$.
        Another possibility is to use the Golub-Kahan-Lanczos bidiagonalization \cite{bddrv:2000} and often
        running a few bidiagonalization steps  can produce a very good estimate of $\|B_{\ufd,\ell}\|_2$, for the same reason as in \cite{zhli:2011}.

\subsection{Acceleration with LOCG}\label{ssec:accNPDo}
We will create a LOCG-accelerated version of \Cref{alg:NPDo:PBTD}.
Without loss of generality, let $(P_1^{(-1)},\ldots,P_m^{(-1)})$ be the tuple of approximate maximizers of \eqref{eq:opt-PBTD}
from the very previous iterative step, and $(P_1,\ldots,P_m)$ the  tuple of current approximate maximizers.
We are now looking for the next approximate maximizer tuple
$(P_1^{(1)},\ldots,P_m^{(1)})$, along the line of LOCG (locally optimal conjugate gradient), according to
\begin{subequations}\label{eq:BTSVD-LOCG}
\begin{align}
&(P_1^{(1)},\ldots,P_m^{(1)})=\arg\max_{X_{\ell}\in\STM{k_{\ell}}{n_{\ell}}\,\,\forall\ell}
                 f(X_1,\ldots,X_m) \label{eq:BTSVD-LOCG-1}\\
&\mbox{s.t.}\,\,\cR(X_{\ell})\subseteq\cR([P_{\ell},\scrR_{\ell}(P_1,\ldots,P_m),P_{\ell}^{(-1)}])\,\,\forall\ell,
             \label{eq:BTSVD-LOCG-2}
\end{align}
\end{subequations}
where $\cR(\cdot)$ denotes the range of a matrix, i.e., the subspace spanned by its columns, and, for $1\le\ell\le m$,
\begin{equation}\label{eq:Rs(P1toPm)}
\scrR_{\ell}(P_1,\ldots,P_m)=\scrH_{\ell}(P_1,\ldots,P_m)-P_{\ell}\sym\big(P_{\ell}^{\HH}\scrH_{\ell}(P_1,\ldots,P_m)\big).
\end{equation}
Initially for the first iteration, we don't have $(P_1^{(-1)},\ldots,P_m^{(-1)})$ and
it is understood that $P_{\ell}^{(-1)}$ is absent from \eqref{eq:BTSVD-LOCG-2}, i.e.,
simply
$\cR(X_{\ell})\subseteq\cR([P_{\ell},\scrR_{\ell}(P_1,\ldots,P_m)])$.

We still have to numerically solve \eqref{eq:BTSVD-LOCG}. For that purpose, let $S_{\ell}\in\STM{\hat n_{\ell}}{n_{\ell}}$ be an orthonormal basis matrix of subspace
$\cR\big(\big[P_{\ell},\scrR_{\ell}(P_1,\ldots,P_m),P_{\ell}^{(-1)}\big]\big)$. Generically, $\hat n_{\ell}=3k_{\ell}$ but $\hat n_{\ell}<3k_i$ can happen.
It can be implemented by the Gram-Schmidt orthogonalization process. For example, for $\cR\big(\big[P_{\ell},\scrR_{\ell}(P_1,\ldots,P_m),P_{\ell}^{(-1)}\big]\big)$,
we notice that $P_{\ell}\in\STM{k_{\ell}}{n_{\ell}}$ already. In MATLAB, to fully take advantage of its optimized functions, we simply set, using $S_1$ as an example,
$S_1=[\scrR_1(P_1,\ldots,P_m),P_1^{(-1)}]$ (or $S_1=\scrR_1(P_1,\ldots,P_m)$ for the first iteration) and then  do
\begin{equation}\label{eq:W-compute}
\framebox{
\begin{minipage}{6.5cm}
\tt  S1=S1-P1*(P1'*S1); S1=orth(S1); \\
\tt S1=S1-P1*(P1'*S1); S1=orth(S1);\\
\tt S1=[P1,S1];
\end{minipage}
}
\end{equation}
where the first two lines  perform the classical Gram-Schmidt orthogonalization twice to almost ensure that
the resulting  columns of $S_1$ are fully orthogonal to the columns of $P_1$ at the end of the second line,
and {\tt orth} is a MATLAB function for
orthogonalization\footnote{Another option is to use MATLAB's thin {\tt qr}:
    {\tt [S1,$\sim$]=qr(S1,0)}.
    }.
It is important to note that the first $k_1$ columns of
the final $S_1$ are the same as those of $P_1$. Similarly, we compute other $S_{\ell}\in\STM{\hat n_{\ell}}{n_{\ell}}$ such that
the first $k_{\ell}$ columns of $S_{\ell}$ are $P_{\ell}$. So we get
$$
\cR\big(\big[P_{\ell},\scrR_{\ell}(P_1,\ldots,P_m),P_{\ell}^{(-1)}\big]\big)=\cR(S_{\ell})\,\,\forall\ell.
$$
Hence \eqref{eq:BTSVD-LOCG-2} can be equivalently stated as
\begin{subequations}\label{eq:LOCGsub}
\begin{equation}\label{eq:LOCGsub:Y}
X_{\ell}=S_{\ell}Y_{\ell}\quad\mbox{for}\,\,\, 1\le\ell\le m.
\end{equation}
Problem \eqref{eq:BTSVD-LOCG} becomes
\begin{equation}\label{eq:LOCGsub-1}
(Y_{1;\opt},\ldots,Y_{m;\opt})=\arg\max_{Y_{\ell}\in\STM{k_{\ell}}{\hat n_{\ell}}\,\,\forall\ell}
             \wtd f(Y_1,Y_2,\ldots,Y_m),
\end{equation}
where, upon setting $\wtd B=B\times_1 S_1^{\HH}\times_2 S_2^{\HH}\cdots\times_m S_m^{\HH}
              \in\bbC^{\hat n_1\times \hat n_2\times\cdots\times \hat n_m}$,
\begin{align}
\wtd f(Y_1,Y_2,\ldots,Y_m)&=f(S_1Y_1,S_2Y_2,\ldots,S_mY_m) \nonumber \\
   &=\sum_{\ell=1}^N
     \big\|\BDiag_{(\tau_{k_1},\ldots,\tau_{k_m})}(\wtd B\times_1 Y_1^{\HH}\times_2 Y_2^{\HH}\cdots\times_m Y_m^{\HH})\big\|_{\F}^2,
       \label{eq:obj-eq:LOCGsub}
\end{align}
in the same form of \eqref{eq:opt-PBTD} but with much smaller $\wtd B$.
\end{subequations}
Solving the reduced problem \eqref{eq:LOCGsub} by \Cref{alg:NPDo:PBTD}
is much faster.
\Cref{alg:accNPDo:PBTD} outlines what we have discussed so far.

\begin{algorithm}[t]
\caption{The LOCG-accelerated NPDo for solving \eqref{eq:opt-PBTD}}
\label{alg:accNPDo:PBTD}
\begin{algorithmic}[1]
\REQUIRE $m$-mode tensor $B\in\bbC^{n_1\times n_2\times\cdots\times n_m}$
         (and, accordingly, matrix-valued functions
         $\scrH_{\ell}(\cdots)$ in \eqref{eq:opt-PBTD-partdiff}),
         and initial $(P_1^{(0)},\ldots,P_m^{(0)})$ with $P_{\ell}^{(0)}\in\STM{k_{\ell}}{n_{\ell}}\,\,\forall\ell$;
\ENSURE  an approximate maximizer tuple of \eqref{eq:opt-PBTD}.
\STATE $P_{\ell}^{(-1)}=[\,]$ (null matrix) for $1\le\ell\le m$;
\FOR{$j=0,1,\ldots$ until convergence}
    \STATE compute $R_{\ell|j}=\scrR_{\ell}(P_1^{(j)},\ldots,P_m^{(j)})$ for $1\le\ell\le m$, where
           $\scrR_{\ell}(P_1^{(j)},\ldots,P_m^{(j)})$ is calculated according to \eqref{eq:Rs(P1toPm)};
    \STATE compute $S_{\ell|j}\in\STM{\hat n_{\ell}}{n_{\ell}}$ such that
           $\cR(S_{\ell|j})=\cR\big(\big[P_{\ell}^{(j)},R_{\ell|j},P_{\ell}^{(j-1)}\big]\big)$, similarly to \eqref{eq:W-compute}
           for $S_1$;
    \STATE solve \eqref{eq:LOCGsub-1} for $(Y_{1;\opt},\ldots,Y_{m;\opt})$ by \Cref{alg:NPDo:PBTD}
          with initially $Y_{\ell}^{(0)}$
           being the first $k_{\ell}$ columns of $I_{\hat n_{\ell}}$;
    \STATE $P_{\ell}^{(j+1)}=S_{\ell|j}Y_{\ell;\opt}$ for $1\le\ell\le m$;
\ENDFOR
\RETURN last $(P_1^{(j)},\ldots,P_m^{(j)})$.
\end{algorithmic}
\end{algorithm}

\begin{remark}\label{rk:SCF4npd+LOCG}
There are a few  comments in order, regarding \Cref{alg:accNPDo:PBTD}.
\begin{enumerate}[(i)]
  \item The stopping criterion \eqref{eq:stop-BTSVD} can be used at Line 2;
  \item It is important to compute $S_{\ell}$ at Line~4 in such a way, as explained moments ago, that its first $k_{\ell}$
        columns are exactly the same as those of $P_{\ell}^{(j)}$.
        This is because as $P_{\ell}^{(j)}$ converges, $P_{\ell}^{(j+1)}$ changes little from $P_{\ell}^{(j)}$ and hence
        $Y_{1;\opt}$ is increasingly close to the first $k_{\ell}$ columns of $I_{\hat n_{\ell}}$.
        The same can be said about $S_2$. This explains the choice of $Y_{\ell}^{(0)}$
        at Line~5.
  \item At Line 4, some saving can be achieved by reusing qualities that are already computed. For example, we may use
        \begin{align*}
        B\times_1 \big(P_1^{(j+1)}\big)^{\HH}&=\big(B\times_1 S_1^{\HH}\big)\times_1 Y_{1;\opt}^{\HH}, \\
        B\times_1 \big(P_1^{(j+1)}\big)^{\HH}\times_2 \big(P_2^{(j+1)}\big)^{\HH}
           &=\big(B\times_1 S_1^{\HH}\times_2 S_2^{\HH}\big)\times_1 Y_{1;\opt}^{\HH}\times_2 Y_{2;\opt}^{\HH},
        \end{align*}
        and, similarly for others,
        to compute the left-hand sides, since $B\times_1 S_1^{\HH}, B\times_1 S_1^{\HH}\times_2 S_2^{\HH},\ldots$
        have been already computed earlier.
  \item An area of improvement is to solve \eqref{eq:LOCGsub-1} with an accuracy, comparable but fractionally better than the
        current $(P_1^{(j)},\ldots,P_m^{(j)})$ as an approximate solution of \eqref{eq:opt-PBTD}.
        Specifically, if we use
        \eqref{eq:stop-BTSVD} at Line~2 here to stop the for-loop: Lines 2--7, with tolerance $\epsilon$, then instead of using the same
        $\epsilon$ for \Cref{alg:NPDo:PBTD} at its Line 1 when the algorithm is called here at Line 5,
        we can use a fraction, say $1/8$, of the total normalized KKT residual, the left-hand side of
        \eqref{eq:stop-BTSVD}, at the current approximation $(P_1^{(j)},\ldots,P_m^{(j)})$ as stopping tolerance within the call to \Cref{alg:NPDo:PBTD}.
\end{enumerate}
\end{remark}

\section{Convergence Analysis}\label{sec:cvg-NPDo4PTBD}
In this section, we will perform a  convergence analysis for \Cref{alg:NPDo:PBTD,alg:accNPDo:PBTD}.
For the sake of presentation, recalling \eqref{eq:Cs} -- \eqref{eq:opt-PBTD-partdiff}, we introduce
\begin{subequations}\label{eq:scrH4proof}
\begin{align}
&\scrH_{\ell}^{(j)}=\scrH_{\ell}(P_1^{(j)},\ldots,P_m^{(j)})\,\,\mbox{for $1\le\ell\le m$, and},  \label{eq:scrH4proof-1}\\
&\what\scrH_{\ell}^{(j)}=\scrH_{\ell}(P_1^{(j+1)},\ldots,P_{\ell-1}^{(j+1)},P_{\ell}^{(j)},\ldots,P_m^{(j)})
    \,\,\mbox{for $1\le\ell\le m$}. \label{eq:scrH4proof-2}
\end{align}
\end{subequations}
By convention, $\what\scrH_1^{(j)}=\scrH_1^{(j)}$.
Also
$\Theta(\cdot,\cdot)$ is the diagonal matrix of the canonical angles between two subspaces of an equal dimension
(see, e.g., \cite{li:2026,stsu:1990}).



\subsection{Convergence Analysis for \Cref{alg:NPDo:PBTD}}
The following lemma is an equivalent restatement of \cite[Lemma 4.10]{moso:1983}
(see also \cite[Proposition 7]{kaqi:1999}) in the context of a metric space.
We will need it to prove one of the conclusions in our main theorem in this subsection, \Cref{thm:cvg4SCF4NPDo-GS:PBTD} below.

\begin{lemma}[{\cite[Lemma 4.10]{moso:1983}}]\label{lm:isolatedconvg}
Let $\scrG$ be a metric space with metric $\dist(\cdot,\cdot)$, and let
$\{\by_i\}_{i=0}^{\infty}$ be a sequence in $\scrG$. If
$\by_*\in \scrG$ is an isolated accumulation point 	
of the sequence such that, for every subsequence $\{\by_i\}_{i\in\bbI}$
converging to $\by_*$, there is an infinite subset $\widehat{\bbI}\subseteq \bbI$ satisfying
$\dist(\by_i,\by_{i+1})\to 0$ as $\what\bbI\ni i\to\infty$,
then the entire sequence $\{\by_i\}_{i=0}^{\infty}$ converges to $\by_*$.
\end{lemma}

\begin{theorem}\label{thm:cvg4SCF4NPDo-GS:PBTD}
Let the sequence $\{(P_1^{(j)},P_2^{(j)},\ldots,P_m^{(j)})\}_{j=0}^{\infty}$ be generated by \Cref{alg:NPDo:PBTD},
$(P_{*1},P_{*2},\ldots,P_{*m})$ an accumulation point of $\{(P_1^{(j)},P_2^{(j)},\ldots,P_m^{(j)})\}_{j=0}^{\infty}$ whose
subsequence $\{(P_1^{(j)},P_2^{(j)},\ldots,P_m^{(j)})\}_{j\in\bbI}$ converges to $(P_{*1},P_{*2},\ldots,P_{*m})$, and
$(\what P_{*1},\what P_{*2},\ldots,\what P_{*\,m-1})$ an accumulation point of $\{(P_1^{(j+1)},P_2^{(j+1)},\ldots,P_{m-1}^{(j+1)})\}_{j\in\bbI}$.
Denote by\footnote{Notice that $\what \scrH_{*1}=\scrH_1(P_{*2},\ldots,P_{*m})=\scrH_{*1}$.},
for $1\le\ell\le m$,
\begin{subequations}\label{eq:Hell-star-all}
\begin{align}
\scrH_{*\ell}&:=\scrH_{\ell}(P_{*1},\ldots,P_{*\,\ell-1},P_{*\ell},\ldots,P_{*m}), \label{eq:scrHell-star} \\
\what \scrH_{*\ell}&:=\scrH_{\ell}(\what P_{*1},\ldots,\what P_{*\ell-1},P_{*\,\ell},\ldots,P_{*m}). \label{eq:hatscrHell-star}
\end{align}
\end{subequations}
The following statements hold.
\begin{enumerate}[{\rm (a)}]
  \item The sequence $\{f(P_1^{(j)},\ldots,P_m^{(j)})\}_{j=0}^{\infty}$ is monotonically increasing and convergent;
  \item We have, for $1\le\ell\le m$,
        \begin{equation}\label{eq:BTSVD:KKTatMAX'-GS}
        \what \scrH_{*\ell}\,P_{*\ell}=P_{*\ell}\Lambda_{*\ell}, \quad
        \Lambda_{*\ell}=P_{*\ell}^{\HH}\,\what \scrH_{*\ell}\succeq 0.
        \end{equation}
%
        As a result, \eqref{eq:PBTD:KKTatMAX} holds if also $\what P_{*\ell}=P_{*\ell}$ for $1\le\ell\le m-1$.
  \item In item~{\rm (b)}, if for $1\le\ell\le m-1$
        \begin{equation}\label{eq:full-rank2}
        \rank(\scrH_{\ell}(P_{*1},\ldots,P_{*m}))=k_{\ell},
        \end{equation}
        then $\what P_{*\ell}=P_{*\ell}$ for $1\le\ell\le m-1$
        and hence $(P_{*1},\ldots,P_{*m})$ satisfies the KKT condition~\eqref{eq:PBTD:KKT},
        i.e., \eqref{eq:PBTD:KKTatMAX} holds.
  \item If $(P_{*1},\ldots,P_{*m})$ is an isolated accumulation point of $\{(P_1^{(j)},\ldots,P_m^{(j)})\}_{j=0}^{\infty}$
        and if
        \begin{equation}\label{eq:full-rank3}
        \rank(\scrH_{\ell}(P_{*1},\ldots,P_{*m}))=k_{\ell}\quad
         \mbox{for $1\le\ell\le m$},
        \end{equation}
        then the entire sequence $\{(P_1^{(j)},\ldots,P_m^{(j)})\}_{j=0}^{\infty}$ converges to $(P_{*1},\ldots,P_{*m})$.
  \item We have $2m$ convergent series
        \begin{subequations}\label{eq:cvg4SCF4NPDo-GS:BTSVD:series}
        \begin{align}
        \sum_{j=0}^{\infty}\sigma_{\min}(\what\scrH_{\ell}^{(j)})\,
                         \big\|\sin\Theta\big(\cR(P_{\ell}^{(j+1)}),\cR(P_{\ell}^{(j)})\big)\big\|_{\F}^2
                      &<\infty,    \label{eq:cvg4SCF4NPDo-GS:BTSVD:series-1a} \\
        \sum_{j=0}^{\infty}\sigma_{\min}(\what\scrH_{\ell}^{(j)})\,
                  \frac {\big\|\what\scrH_{\ell}^{(j)}-P_{\ell}^{(j)}\big([P_{\ell}^{(j)}]^{\HH}\what\scrH_{\ell}^{(j)}\big)\big\|_{\F}^2}
                        {\big\|\what\scrH_{\ell}^{(j)}\big\|_{\F}^2}
                      &<\infty,                 \label{eq:cvg4SCF4NPDo-GS:BTSVD:series-2a}
        \end{align}
        \end{subequations}
        for $1\le\ell\le m$.
\end{enumerate}
\end{theorem}

\begin{proof}
The flow of computations goes as follows:
for $j=0,1,2,\ldots$
\begin{align}
    (P_1^{(j)},P_2^{(j)},\ldots,P_m^{(j)})
&\to (P_1^{(j+1)},P_2^{(j)},\ldots,P_m^{(j)}) \nonumber\\
&\to (P_1^{(j+1)},P_2^{(j+1)},P_3^{(j)},\ldots,P_m^{(j)}) \nonumber\\
&\to \cdots \nonumber\\
&\to (P_1^{(j+1)},P_2^{(j+1)},\ldots,P_m^{(j+1)}). \label{eq:flow-GS}
\end{align}
Recall \eqref{eq:scrH4proof-1} and let, for $1\le\ell\le m$,
\begin{subequations}\label{eq:cvg4HOOI:TD:pf-1}
\begin{align}
\eta_{j+\ell/m}&=\tr\big([P_{\ell}^{(j+1)}]^{\HH}\scrH_{\ell}^{(j)}\big)-\tr\big([P_{\ell}^{(j)}]^{\HH}\scrH_{\ell}^{(j)}\big)
            \label{eq:cvg4SCF4NPDo-GS:BTSVD:pf-1aa}\\
          &=\|\scrH_{\ell}^{(j)}\|_{\tr}-\tr\big([P_{\ell}^{(j)}]^{\HH}\scrH_{\ell}^{(j)}\big),
                 \label{eq:cvg4SCF4NPDo-GS:BTSVD:pf-1ab}
\end{align}
\end{subequations}
due to how
$P_{\ell}^{(j+1)}$ is defined in \Cref{alg:NPDo:PBTD}. All
$\eta_{j+\ell/m}$ for $1\le\ell\le m$ are nonnegative by \cite[Lemma~4.2]{li:2026}. Now use
\Cref{thm:f(UVW)-Ansatz} to conclude that, for $1\le\ell\le m$,
\begin{multline}\label{eq:cvg4SCF4NPDo-GS:BTSVD:pf-2'}
f(P_1^{(j+1)},\ldots,P_{\ell-1}^{(j+1)},P_{\ell}^{(j+1)},P_{\ell+1}^{(j)}\,\ldots,P_m^{(j)}) \\
   \ge f(P_1^{(j+1)},\ldots,P_{\ell-1}^{(j+1)},P_{\ell}^{(j)},P_{\ell+1}^{(j)}\,\ldots,P_m^{(j)})+ \eta_{j+\ell/m},
\end{multline}
yielding
\begin{align*}
f(P_1^{(j+1)},P_2^{(j+1)},\ldots,P_m^{(j+1)})&\ge f(P_1^{(j+1)},\ldots,P_{m-1}^{(j+1)},P_m^{(j)})\\
    &\ge \cdots \\
    &\ge f(P_1^{(j)},P_2^{(j)},\ldots,P_m^{(j)}).
\end{align*}
This proves item (a).
Along the way, we also showed
\begin{equation}\label{eq:cvg4SCF4NPDo-GS:BTSVD:pf-3}
f(P_1^{(j+1)},\ldots,P_m^{(j+1)})\ge f(P_1^{(j)},\ldots,P_m^{(j)})
   +\sum_{\ell=1}^m\Big[\|\what\scrH_{\ell}^{(j)}\|_{\tr}-\tr\big([P_{\ell}^{(j)}]^{\HH}\what\scrH_{\ell}^{(j)}\big)\Big].
\end{equation}

In what follows, without loss of generality, we may also assume that
$$
\big\{\big(P_1^{(j+1)},P_2^{(j+1)},\ldots,,P_{m-1}^{(j+1)}\big)\big\}_{j\in\bbI}
$$
converges to
$(\what P_{*1},\what P_{*2},\ldots,\what P_{*\,m-1})$; otherwise,
we can pick up a convergent subsequence of $\{(P_1^{(j+1)},P_2^{(j+1)},\ldots,,P_{m-1}^{(j+1)})\}_{j\in\bbI}$
and reassign $\bbI$ accordingly.
We have
\begin{subequations}\label{eq:cvg4SCF4NPDo-GS:BTSVD:pf-5}
\begin{alignat}{2}
\lim_{\bbI\ni j\to \infty}P_{\ell}^{(j)}&=P_{*\ell} &\quad&\mbox{for $1\le\ell\le m$}, \\
\lim_{\bbI\ni j\to \infty}P_{\ell}^{(j+1)}&=\what P_{*\ell} &\quad&\mbox{for $1\le\ell\le m-1$}.
\end{alignat}
\end{subequations}
%

We now prove item~(b). Notice that the contrary to \eqref{eq:BTSVD:KKTatMAX'-GS} is
\begin{equation}\label{eq:cvg4SCF4NPDo-GS:BTSVD:pf-6}
\mbox{either $\cR(\what \scrH_{*\ell})\not\subseteq\cR(P_{*\ell})$ or
$P_{*\ell}^{\HH}\,\what \scrH_{*\ell}\not\succeq 0$}.
\end{equation}
Hence the contrary to \eqref{eq:BTSVD:KKTatMAX'-GS} for $1\le\ell\le m$ is that at least one of the $2m$ relations in \eqref{eq:cvg4SCF4NPDo-GS:BTSVD:pf-6} is true, which, by \cite[Lemma~4.2]{li:2026}, implies
$\delta:=\delta_1+\cdots+\delta_m>0$ where
$$
\delta_{\ell}:=\|\what \scrH_{*\ell})\|_{\tr}-\tr(P_{*\ell}^{\HH}\what \scrH_{*\ell})\ge 0.
$$
Since $\|\scrH_{\ell}(P_1,\ldots,P_m)\|_{\tr}$ and $\tr(P_{\ell}^{\HH}\scrH_{\ell}(P_1,\ldots,P_m))$ for $1\le\ell\le m$, and $f(P_1,\ldots,P_m)$ are continuous in
$(P_1,\ldots,P_m)\in\bbC^{n_1\times k_1}\times\cdots\times\bbC^{n_m\times k_m}$,
there is an $j_0\in\bbI$ such that
\begin{subequations}\label{eq:cvg4SCF4NPDo-GS:BTSVD:pf-7}
\begin{gather}
\Big|\|\what\scrH_{*\ell}\|_{\tr}-\|\what\scrH_{\ell}^{(j_0)}\|_{\tr}\Big| <\frac {\delta}{2m+1}, \label{eq:thm:cvg4SCF4NPDo-GS:pf-7a}\\
\Big|\tr\big((P_{\ell}^{(j_0)})^{\HH}\what\scrH_{\ell}^{(j_0)}\big)-\tr\big(P_{*\ell}^{\HH}\what\scrH_{*\ell}\big)\Big|
         <\frac {\delta}{2m+1},     \label{eq:thm:cvg4SCF4NPDo-GS:pf-7b}\\
f(P_{*1},\ldots,P_{*m})-\frac {\delta}{2(2m+1)}<f(P_1^{(j_0)},\ldots,P_m^{(j_0)})\le f(P_{*1},\ldots,P_{*m}).
      \label{eq:thm:cvg4SCF4NPDo-GS-subs-1:pf-7e}
\end{gather}
\end{subequations}
By \eqref{eq:cvg4SCF4NPDo-GS:BTSVD:pf-3}, we have
\begin{align*}
f(P_1^{(j_0+1)},\ldots,P_m^{(j_0+1)})&\ge f(P_1^{(j_0)},\ldots,P_m^{(j_0)}) \\
  & \quad    +\sum_{\ell=1}^m\Big[\|\what\scrH_{\ell}^{(j_0)}\|_{\tr}-\tr\big([P_{\ell}^{(j_0)}]^{\HH}\what\scrH_{\ell}^{(j_0)}\big)\Big]  \\
  &>f(P_{*1},\ldots,P_{*m})-\frac {\delta}{2(2m+1)} \\
  & \quad     +\sum_{\ell=1}^m\Big[\|\what\scrH_{*\ell}\|_{\tr}-\frac {\delta}{2m+1}-\tr(P_{*\ell}^{\HH}\what\scrH_{*\ell})
                -\frac {\delta}{2m+1}\Big] \\
  &=f(P_{*1},\ldots,P_{*m})+\frac {\delta}{2(2m+1)}>f(P_{*1},\ldots,P_{*m}),
\end{align*}
contradicting $f(P_1^{(i)},\ldots,P_m^{(i)})\le\lim_{j\to\infty}f(P_1^{(j)},\ldots,P_m^{(j)})= f(P_{*1},\ldots,P_{*m})$
for all $1\le i<\infty$.
This completes the proof of item~(b).

We will now show, by induction, that $\what P_{*\ell}=P_{*\ell}$ under \eqref{eq:full-rank2} for $1\le\ell\le m-1$.
Since we always have
$$
\scrH_1(P_1^{(j)},\ldots,P_m^{(j)})=P_1^{(j+1)}\Lambda_1^{(j)},
$$
a polar decomposition. Letting $\bbI\ni j\to\infty$
yields $\scrH_{*1}=\scrH_1(P_{*1},\ldots,P_{*m})=\what P_{*1}\what \Lambda_{1*}$, a polar decomposition also.
If $\rank(\scrH_{*1})=k_1$, then the polar decomposition of $\scrH_{*1}$ is
unique \cite{high:2008,li:1993b,li:1995,li:2014HLA} and hence $\what P_{*1}=P_{*1}$ by the fact that
\eqref{eq:BTSVD:KKTatMAX'-GS} for $\ell=1$ is another polar decomposition of $\scrH_{*1}\equiv \what\scrH_{*1}$.
Suppose now that $\what P_{*\ell}=P_{*\ell}$
for $1\le\ell\le s< m-1$.
Since we have
$$
\scrH_{s+1}(P_1^{(j+1)},\ldots,P_s^{(j+1)},P_{s+2}^{(j)},\ldots,P_m^{(j)})=P_{s+1}^{(j+1)}\Lambda_{s+1}^{(j)},
$$
a polar decomposition.
Letting $\bbI\ni j\to\infty$
yields
\begin{equation}\label{eq:thm:cvg4SCF4NPDo-GS-subs-1:pf-8}
\what\scrH_{s+1}(\what P_{*1},\ldots,\what P_{s*},P_{*\,s+1},\ldots,P_{*m})=\what P_{*\,s+1}\what \Lambda_{*\,s+1},
\end{equation}
implying $\scrH_{*\,s+1}\,\what P_{*\,s+1}=\what P_{*\,s+1}\what \Lambda_{*\,s+1}$ by the induction hypothesis that
$\what P_{*\ell}=P_{*\ell}$ for $1\le\ell\le s< m-1$.
Equation~\eqref{eq:thm:cvg4SCF4NPDo-GS-subs-1:pf-8} is
a polar decomposition also.
This, together with \eqref{eq:full-rank2} for $\ell=s+1$, imply that $\what P_{*\,s+1}=P_{*\, s+1}$, completing the induction proof,
and, at the same time, the proof of item~(c).

We now prove item~(d) with the help of \Cref{lm:isolatedconvg}. Let
$\{(P_1^{(j)},\ldots,P_m^{(j)})\}_{j\in\bbI_1}$ be a convergent subsequence that converges to $(P_{*1},\ldots,P_{*m})$.
Since $\{(P_1^{(j+1)},\ldots,P_m^{(j+1)})\}_{j\in\bbI_1}$ is a bounded sequence, it has a convergent subsequence
$\{(P_1^{(j+1)},\ldots,P_m^{(j+1)})\}_{j\in\bbI_2}$ that converges to $(\what P_{*1},\ldots,\what P_{*m})$,
where $\bbI_2\subseteq\bbI_1$.
Use the argument for the proof of item~(c) to conclude that $\what P_{*\ell}=P_{*\ell}$ for $1\le\ell\le m-1$.
It remains to show $\what P_{*m}=P_{*m}$.
Always $\scrH_m(P_1^{(j+1)},\ldots,P_{m-1}^{(j+1)},P_m^{(j)})=P_m^{(j+1)}\Lambda_m^{(j)}$, a polar decomposition.
Letting $\bbI_1\supseteq\bbI_2\ni j\to\infty$ yields
$\scrH_m(P_{*1},\ldots,P_{*m})=\what P_{*m}\Lambda_{*m}$ due to $\what P_{*\ell}=P_{*\ell}$ for $1\le\ell\le m-1$.
This is yet another polar decomposition of
$\scrH_m(P_{*1},\ldots,P_{*m})=\scrH_{*m}$, besides
\eqref{eq:BTSVD:KKTatMAX'-GS} for $\ell=m$. It follows from \eqref{eq:full-rank3} for $\ell=m$ that
$\scrH_{*m}$ has a unique polar decomposition, implying $\what P_{*m}=P_{*m}$, as expected.
Hence as $\bbI_2\ni j\to\infty$
$$
\|P_{\ell}^{(j)}-P_{\ell}^{(j+1)}\|_{\F}\le\|P_{\ell}^{(j)}-P_{*\ell}\|_{\F}+\|P_{*\ell}-P_{\ell}^{(j+1)}\|_{\F}\to 0
$$
for $1\le\ell\le m$.
Now use Lemma~\ref{lm:isolatedconvg} to conclude that the entire sequence
$\{(P_1^{(j)},\ldots,P_m^{(j)})\}_{i=0}^{\infty}$ converges to $(P_{*1},\ldots,P_{*m})$, as needed.

Finally for item (e), we return to \eqref{eq:cvg4SCF4NPDo-GS:BTSVD:pf-3}.
By \cite[Theorem~4.1]{li:2026}, we have for $1\le\ell\le m$
\begin{equation}\label{eq:thm:cvg4SCF4NPDo-GS-subs-1:pf-9}
\frac {\big\|\what\scrH_{\ell}^{(j)}-P_{\ell}^{(j)}\big([P_{\ell}^{(j)}]^{\HH}\what\scrH_{\ell}^{(j)}\big)\big\|_{\F}^2}
                        {\big\|\what\scrH_{\ell}^{(j)}\big\|_{\F}^2}
    \le\big\|\sin\Theta\big(\cR(P_{\ell}^{(j+1)}),\cR(P_{\ell}^{(j)})\big)\big\|_{\F}^2
    \le\frac {2\eta_{j+\ell/m}}{\sigma_{\min}(\what\scrH_{\ell}^{(j)})}.
\end{equation}
Hence \eqref{eq:cvg4SCF4NPDo-GS:BTSVD:series-2a} is a consequence of \eqref{eq:cvg4SCF4NPDo-GS:BTSVD:series-1a}.
To see \eqref{eq:cvg4SCF4NPDo-GS:BTSVD:series-1a}, we combine
$$
2\sum_{j=0}^{\infty}\sum_{\ell=1}^m\eta_{j+\ell/m}
     \le\lim_{j\to\infty} f(P_1^{(j+1)},\ldots,P_m^{(j+1)})-f(P_1^{(0)},\ldots,P_m^{(0)})
     <\infty,
$$
with the second inequality in \eqref{eq:thm:cvg4SCF4NPDo-GS-subs-1:pf-9}.
The proof is completed.
\end{proof}

\subsection*{Discussion}
As a corollary of Theorem~\ref{thm:cvg4SCF4NPDo-GS:PBTD}(e),
if $\sigma_{\min}(\scrH_1^{(j)})$ is eventually bounded below away from $0$
uniformly\footnote {By which we mean that there exist a constant $c>0$ and an integer $K$ such that
  $\sigma_{\min}(\scrH_1^{(j)})\ge c$ for all $j\ge K$.},
then
\begin{subequations}\label{eq:NPDo-always}
\begin{equation}\label{eq:NPDo-always-1}
\lim_{j\to\infty}\frac {\big\|\scrH_1^{(j)}-P_1^{(j)}\big([P_1^{(j)}]^{\HH}\scrH_1^{(j)}\big)\big\|_{\F}}
                        {\big\|\scrH_1^{(j)}\big\|_{\F}} =0,
\end{equation}
namely, increasingly $\scrH_1^{(j)}\approx P_1^{(j)}\big([P_1^{(j)}]^{\HH}\scrH_1^{(j)}\big)$ towards
a polar decomposition of $\scrH_1(P_1^{(j)},\ldots,P_m^{(j)})$, which means that $(P_1^{(j)},\ldots,P_m^{(j)})$ becomes a more and more accurate approximate solution
to NPDo \eqref{eq:PBTD:KKT} for $\ell=1$, even in the absence of knowing whether the entire sequence $\{(P_1^{(j)},\ldots,P_m^{(j)})\}_{j=0}^{\infty}$ converges or not.
For the same reason, if $\sigma_{\min}(\what\scrH_{\ell}^{(j)})$ is eventually bounded below away from $0$
uniformly, then
\begin{equation}\label{eq:NPDo-always-2}
\lim_{j\to\infty}\frac {\big\|\what\scrH_{\ell}^{(j)}-P_{\ell}^{(j)}\big([P_{\ell}^{(j)}]^{\HH}\what\scrH_{\ell}^{(j)}\big)\big\|_{\F}}
                        {\big\|\what\scrH_{\ell}^{(j)}\big\|_{\F}},
\end{equation}
\end{subequations}
namely, increasingly $\what\scrH_{\ell}^{(j)}\approx P_{\ell}^{(j)}\big([P_{\ell}^{(j)}]^{\HH}\what\scrH_{\ell}^{(j)}\big)$,
regardless whether the entire sequence $\{(P_1^{(j+1)},\ldots,P_{\ell-1}^{(j+1)},P_{\ell}^{(j)},\ldots,P_m^{(j)})\}_{j=0}^{\infty}$ converges or not.

Implications in \eqref{eq:NPDo-always} by \Cref{thm:cvg4SCF4NPDo-GS:PBTD}(e) are perhaps more significant
than the conclusions on an accumulation point $(P_{*1},\ldots,P_{*m})$ in \Cref{thm:cvg4SCF4NPDo-GS:PBTD}(b,c) due to our previously
comments that maximizer tuples $(P_1,\ldots,P_m)$ of \eqref{eq:opt-PBTD} are inherently non-unique. Rather than claiming
entrywise convergence of computed approximations, these implications shows that
computed maximizer tuples progress to satisfy the KKT condition while, at the same time,
the objective moves up as guaranteed by \Cref{thm:cvg4SCF4NPDo-GS:PBTD}(a).

Next, we comment on our development so far in comparison with what in the literature for the Tucker decomposition (TD)
\cite[p.478]{koba:2009},
the tensor SVD (TSVD)~\cite{chsa:2009}, and rank-1 approximation of a tensor~\cite{zhgo:2001}.
We start by mentioning that the studies in the literature is pre-dominantly, if not exclusively, for real tensors, i.e.,
with real number entries, whereas ours is stated for complex tensors and automatically applies to real tensors without
modifications.

The Tucker decomposition (TD) has $t=1$ and all $\tau_{k_{\ell}}=(k_{\ell})$ for
$1\le\ell\le m$, and hence problem \eqref{eq:opt-PBTD} degenerates to
\begin{equation}\label{eq:opt-TD}
\max_{P_{\ell}\in\STM{k_{\ell}}{n_{\ell}}\,\,\forall\ell} \,
   \Big\{f(P_1,\ldots,P_m):=\big\|B\times_1 P_1^{\HH}\times_2 P_2^{\HH}\cdots\times_m P_m^{\HH}\big\|_{\F}^2\Big\}.
\end{equation}
What differentiates the objective here from the one in \eqref{eq:opt-PBTD:obj-tr} is that $f$ in \eqref{eq:opt-PBTD:obj-tr}
is the sum of $t$ coupled terms whereas here it has just one term which makes maximizing $f$ alternatingly
over $P_{\ell}$ and corresponding analysis much simpler.

The tensor SVD (TSVD)~\cite{chsa:2009} has $\tau_{k_{\ell}}=(1,1,\ldots,1)$ for $1\le\ell\le m$.
Chen and Saad conducted a thorough investigation on the case
with the same objective as ours and they imposed orthogonality constraints to deal with the Stiefel manifolds.
They obtained essentially the same KKT condition
in terms of $m$ equations
as ours in \eqref{eq:PBTD:KKT};
they recognized each equation becomes a polar decomposition at optimality; their eventual ALS iteration,
\cite[Algorithm~1]{chsa:2009}, is the same as our \Cref{alg:NPDo:PBTD};
They obtained \Cref{thm:cvg4SCF4NPDo-GS:PBTD}(a) and \Cref{thm:cvg4SCF4NPDo-GS:PBTD}(b,c) somewhat, but not
\Cref{thm:cvg4SCF4NPDo-GS:PBTD}(d,e), especially \Cref{thm:cvg4SCF4NPDo-GS:PBTD}(e)  that quantitatively sheds light on the speed of convergence
and implies \eqref{eq:NPDo-always}.

The rank-1 approximation of a tensor~\cite{zhgo:2001} has all $k_{\ell}=1$ (and hence $t=1$, too, as for TD).
Problem \eqref{eq:opt-PBTD} is now even simpler than TD \eqref{eq:opt-TD}. Among others,
Zhang and Golub~\cite{zhgo:2001} created \Cref{alg:NPDo:PBTD} with both Gauss-Seidel-type updating and Jacobi-type updating
without any need to involve polar decomposition. Local convergence for Gauss-Seidel-type updating is established.
More specially, Kofidis and Regalia~\cite{kore:2002} focused on the same problem for supersymmetric tensors.

\subsection*{Possible Extension}
It is interesting to observe that in proving \Cref{thm:cvg4SCF4NPDo-GS:PBTD} the actual forms of
all $\scrH_{\ell}(\cdots)$ play no roles at all. What have been used in the proofs are:
\begin{enumerate}[i)]
  \item $\scrH_{\ell}(\cdots)$ for $1\le\ell\le m$ have the properties stated in \Cref{thm:f(UVW)-Ansatz}(a,b), and
  \item $\scrH_{\ell}(\cdots)$ for $1\le\ell\le m$ and $f(\cdots)$ are continuous.
\end{enumerate}
Hence \Cref{thm:cvg4SCF4NPDo-GS:PBTD} remains valid for solving a general
\begin{equation}\label{eq:opt-on-STM2}
\max_{P_{\ell}\in\STM{k_{\ell}}{n_{\ell}}\,\,\forall\ell}\,  f(P_1,\ldots,P_m),
\end{equation}
which possesses the two properties listed above, where
$f$ is a real-valued function and
$
\scrH_{\ell}(P_1,\ldots,P_m):=\nabla_{P_{\ell}} f(P_1,\ldots,P_m)
$
is the partial Euclidean gradient of $f$ with respect to $P_{\ell}$.

\subsection{Convergence Analysis for \Cref{alg:accNPDo:PBTD}}
Next we will analyze the convergence of \Cref{alg:accNPDo:PBTD},
the LOCG-accelerated NPDo for solving \eqref{eq:opt-PBTD}. In the algorithm, it calls
\Cref{alg:NPDo:PBTD} to solve the reduced problem \eqref{eq:LOCGsub-1}. It is unrealistic to
ask for the reduced problem to be solved exactly. The question is how much accurately it should be solved
to achieve reasonable convergence behavior in the end. Making the sequence of objective values $\{f(P_1^{(j)},\ldots,P_m^{(j)})\}_{j=0}^{\infty}$ monotonically increasing is easy, but that is not enough.
At the minimum, we should
expect that any accumulation point of $\{(P_1^{(j)},P_2^{(j)},\ldots,P_m^{(j)})\}_{j=0}^{\infty}$
to satisfy the KKT condition \eqref{eq:PBTD:KKT}.
To that end, we will assume that in \Cref{alg:accNPDo:PBTD} the computed
$(P_1^{(j+1)},P_2^{(j+1)},\ldots,P_m^{(j+1)})$ at the end of $j$th loop satisfies
\begin{align}
\eta_j:=f(P_1^{(j+1)},P_2^{(j+1)},\ldots,P_m^{(j+1)})
   &- f(P_1^{(j)},P_2^{(j)},\ldots,P_m^{(j)}) \nonumber\\
   &\ge    c\;\underbrace{\max_{1\le\ell\le m}
         \Big[\|\scrH_{\ell}^{(j)}\|_{\tr}-\tr\big([P_{\ell}^{(j)}]^{\T}\scrH_{\ell}^{(j)}\big)\Big]}_{=:\eta_{*j}},
            \label{eq:accNPDo:PBTD:assume}
\end{align}
where $\scrH_{\ell}^{(j)}$ is as defined in \eqref{eq:scrH4proof-1} and $c>0$ is a constant, independent of
$j$. We now explain this is a reasonable assumption. Considering the $j$th iteration before
calling \Cref{alg:NPDo:PBTD}, we have
the reduced problem \eqref{eq:LOCGsub-1} with $S_{\ell}\,\forall\ell$ replaced with
$S_{\ell|j}\,\forall\ell$ at Line 4 of \Cref{alg:accNPDo:PBTD}. Correspondingly, for the reduced problem
we have for each $\ell$
$$
\wtd\scrH_{\ell}^{(0)}:=\wtd\scrH_{\ell}(Y_1,\ldots,Y_m)\Big|_{Y_{\ell}=Y_{\ell}^{(0)}\,\forall\ell}
  =S_{\ell|j}^{\T}\scrH_{\ell}(P_1^{(j)},P_2^{(j)},\ldots,P_m^{(j)})
  =S_{\ell|j}^{\T}\scrH_{\ell}^{(j)},
$$
and hence, if done independently on each $\ell$, one step of NPDo iteration with $\wtd\scrH_{\ell}^{(0)}$ will increase the objective
$\wtd f$, and thereby the original $f$, by
$$
\|\wtd\scrH_{\ell}^{(0)}\|_{\tr}-\tr\big([Y_{\ell}^{(0)}]^{\T}\wtd\scrH_{\ell}^{(0)}\big)
  =\|\scrH_{\ell}^{(j)}\|_{\tr}-\tr\big([P_{\ell}^{(j)}]^{\T}\scrH_{\ell}^{(j)}\big).
$$
With this in mind, the assumption \eqref{eq:accNPDo:PBTD:assume} merely requires that
calling \Cref{alg:NPDo:PBTD} each time within \Cref{alg:accNPDo:PBTD} computes
the next approximation accurately up to the point that the  objective increases
at least by a constant fraction of what one independent NPDo step on the mode achieves.

\begin{theorem}\label{thm:cvg4SCF4accNPDo-GS:PBTD}
Let the sequence $\{(P_1^{(j)},P_2^{(j)},\ldots,P_m^{(j)})\}_{j=0}^{\infty}$ be generated by \Cref{alg:accNPDo:PBTD},
and $(P_{*1},P_{*2},\ldots,P_{*m})$ an accumulation point of $\{(P_1^{(j)},P_2^{(j)},\ldots,P_m^{(j)})\}_{j=0}^{\infty}$.
Suppose that at the $j$th loop, the computed
$(P_1^{(j+1)},P_2^{(j+1)},\ldots,P_m^{(j+1)})$ with the help of \Cref{alg:NPDo:PBTD} satisfies
\eqref{eq:accNPDo:PBTD:assume}.
The following statements hold.
\begin{enumerate}[{\rm (a)}]
  \item The sequence $\{f(P_1^{(j)},\ldots,P_m^{(j)})\}_{j=0}^{\infty}$ is monotonically increasing and convergent;
  \item $(P_{*1},\ldots,P_{*m})$ satisfies the KKT condition~\eqref{eq:PBTD:KKT},
        i.e., \eqref{eq:PBTD:KKTatMAX} holds.
  \item We have $2m$ convergent series:
        for $1\le\ell\le m$
        \begin{subequations}\label{eq:cvg4SCF4accNPDo-GS:BTSVD:series}
        \begin{align}
        \sum_{j=0}^{\infty}\sigma_{\min}(\scrH_{\ell}^{(j)})\,
                         \big\|\sin\Theta\big(\cR(\scrH_{\ell}^{(j)}),\cR(P_{\ell}^{(j)})\big)\big\|_{\F}^2
                      &<\infty,    \label{eq:cvg4SCF4accNPDo-GS:BTSVD:series-1a} \\
        \sum_{j=0}^{\infty}\sigma_{\min}(\scrH_{\ell}^{(j)})\,
                  \frac {\big\|\scrH_{\ell}^{(j)}-P_{\ell}^{(j)}\big([P_{\ell}^{(j)}]^{\HH}\scrH_{\ell}^{(j)}\big)\big\|_{\F}^2}
                        {\big\|\scrH_{\ell}^{(j)}\big\|_{\F}^2}
                      &<\infty,                 \label{eq:cvg4SCF4accNPDo-GS:BTSVD:series-2a}
        \end{align}
        \end{subequations}
        where $\scrH_{\ell}^{(j)}$ is as defined in \eqref{eq:scrH4proof-1}.
\end{enumerate}
\end{theorem}

\begin{proof}
Item (a) immediately follows from \eqref{eq:accNPDo:PBTD:assume}.
Essentially the proof of \Cref{thm:cvg4SCF4NPDo-GS:PBTD}(b) works, too, for item (b) here.
Let
$$
\delta=\max_{1\le\ell\le m}
         \big[\|\scrH_{*\ell}\|_{\tr}-\tr\big(P_{*\ell}^{\HH}\scrH_{*\ell}\big)\big],
$$
where $\scrH_{*\ell}$ is as defined in \eqref{eq:scrHell-star}.
If at least of the equations in \eqref{eq:PBTD:KKTatMAX} does not hold, then $\delta>0$. By continuity,
there is an $j_0\in\bbI$ such that
\begin{subequations}\label{eq:cvg4SCF4accNPDo-GS:BTSVD:pf-7}
\begin{gather}
\Big|\|\scrH_{*\ell}\|_{\tr}-\|\scrH_{\ell}^{(j_0)}\|_{\tr}\Big| <\frac {\delta}3, \label{eq:thm:cvg4SCF4accNPDo-GS:pf-7a}\\
\Big|\tr\big((P_{\ell}^{(j_0)})^{\HH}\scrH_{\ell}^{(j_0)}\big)-\tr\big(P_{*\ell}^{\HH}\scrH_{*\ell}\big)\Big|
         <\frac {\delta}3,     \label{eq:thm:cvg4SCF4accNPDo-GS:pf-7b}\\
f(P_{*1},\ldots,P_{*m})-\frac {c\delta}{6}<f(P_1^{(j_0)},\ldots,P_m^{(j_0)})\le f(P_{*1},\ldots,P_{*m}).
      \label{eq:thm:cvg4SCF4accNPDo-GS-subs-1:pf-7e}
\end{gather}
\end{subequations}
By \eqref{eq:accNPDo:PBTD:assume}, we have
\begin{align*}
f(P_1^{(j_0+1)},\ldots,P_m^{(j_0+1)})&\ge f(P_1^{(j_0)},\ldots,P_m^{(j_0)}) \\
  & \quad    +c\max_{1\le\ell\le m}
       \Big[\|\scrH_{\ell}^{(j_0)}\|_{\tr}-\tr\big([P_{\ell}^{(j_0)}]^{\HH}\scrH_{\ell}^{(j_0)}\big)\Big]  \\
  &>f(P_{*1},\ldots,P_{*m})-\frac {c\delta}{6} \\
  & \quad     +c\max_{1\le\ell\le m}\Big[\|\scrH_{*\ell}\|_{\tr}-\frac {\delta}3-\tr(P_{*\ell}^{\HH}\scrH_{*\ell})
                -\frac {\delta}3\Big] \\
  &=f(P_{*1},\ldots,P_{*m})+\frac {c\delta}6>f(P_{*1},\ldots,P_{*m}),
\end{align*}
contradicting $f(P_1^{(i)},\ldots,P_m^{(i)})\le\lim_{j\to\infty}f(P_1^{(j)},\ldots,P_m^{(j)})= f(P_{*1},\ldots,P_{*m})$
for all $1\le i<\infty$.
This completes the proof of item~(b).

Finally, instead of \eqref{eq:thm:cvg4SCF4NPDo-GS-subs-1:pf-9},
by \cite[Theorem~4.1]{li:2026}, we have for $1\le\ell\le m$
\begin{subequations}\label{eq:thm:cvg4SCF4accNPDo-GS-subs-1:pf-9}
\begin{align}
\frac {\big\|\scrH_{\ell}^{(j)}-P_{\ell}^{(j)}\big([P_{\ell}^{(j)}]^{\HH}\scrH_{\ell}^{(j)}\big)\big\|_{\F}^2}
                        {\big\|\scrH_{\ell}^{(j)}\big\|_{\F}^2}
    &\le\big\|\sin\Theta\big(\cR(\scrH_{\ell}^{(j)}),\cR(P_{\ell}^{(j)})\big)\big\|_{\F}^2
        \label{eq:thm:cvg4SCF4accNPDo-GS-subs-1:pf-9a} \\
    &\le\frac {2\Big[\|\scrH_{\ell}^{(j)}\|_{\tr}-\tr\big([P_{\ell}^{(j)}]^{\T}\scrH_{\ell}^{(j)}\big)\Big]}
              {\sigma_{\min}(\scrH_{\ell}^{(j)})} \nonumber \\
    &\le\frac {2\eta_{*j}}{\sigma_{\min}(\scrH_{\ell}^{(j)})}. \label{eq:thm:cvg4SCF4accNPDo-GS-subs-1:pf-9b}
\end{align}
\end{subequations}
Hence \eqref{eq:cvg4SCF4accNPDo-GS:BTSVD:series-2a} is a consequence of \eqref{eq:cvg4SCF4accNPDo-GS:BTSVD:series-1a}.
To see \eqref{eq:cvg4SCF4accNPDo-GS:BTSVD:series-1a}, we combine
$$
c\sum_{j=0}^{\infty}\eta_{*j}
  \le\sum_{j=0}^{\infty}\eta_j
     \le\lim_{j\to\infty} f(P_1^{(j+1)},\ldots,P_m^{(j+1)})-f(P_1^{(0)},\ldots,P_m^{(0)})
     <\infty,
$$
with \eqref{eq:thm:cvg4SCF4accNPDo-GS-subs-1:pf-9}.
The proof is completed.
\end{proof}

It is noted the difference between \eqref{eq:cvg4SCF4NPDo-GS:BTSVD:series-1a} and \eqref{eq:cvg4SCF4accNPDo-GS:BTSVD:series-1a}
lies in the appearance of $\cR(P_{\ell}^{(j+1)})$ in the former but $\cR(\scrH_{\ell}^{(j)})$ in the latter. The reason is
no longer $P_{\ell}^{(j+1)}$ is taken as the orthonormal polar factor of $\scrH_{\ell}^{(j)}$
by \Cref{alg:accNPDo:PBTD}.

\section{Numerical Experiments}\label{sec:egs}
In this section, we will report some numerical tests on randomly generated tensors $B$
that are approximately block-diagonal, in the sense of \eqref{eq:BTD}. As $m>3$ introduces no essentially difference in theory and experiments, except perhaps more computational work, we will use $3$-mode tensor in our experiments, i.e., $m=3$.
Specifically, given $n_{\ell}$ and $k_{\ell}$ and its partition $\tau_{k_{\ell}}$ in \eqref{eq:tau(ki)-n}
for $1\le\ell\le m=3$, we first generate, in MATLAB,
dependent on real or complex tensors,
\begin{alignat*}{2}
T&={\tt zeros}([n_1,n_2,n_3]), && \\
T_{(1:k_1,1:k_2,1:k_3)}&=\diag(T_{111},\ldots, T_{ttt}), && \\
E&={\tt randn}([n_1,n_2,n_3])\quad
             &&\mbox{or}\quad {\tt randn}([n_1,n_2,n_3])+{\tt 1i*randn}([n_1,n_2,n_3]),
\end{alignat*}
where, for $1\le i\le t$,
$$
T_{iii}={\tt randn}([k_{1i},k_{2i},k_{3i}])\quad
             \mbox{or}\quad{\tt randn}([k_{1i},k_{2i},k_{3i}])+{\tt 1i*randn}([k_{1i},k_{2i},k_{3i}]), \\
$$
and, for $1\le\ell\le m=3$,
$$
Q_{\ell}={\tt orth}({\tt randn}(n_{\ell})),
\quad\mbox{or}\quad
  {\tt orth}({\tt randn}(n_{\ell})+{\tt 1i*randn}(n_{\ell})),
$$
and finally
\begin{equation}\label{eq:random-B}
B=(T+\eta\,E)\times_1 Q_1\times_2 Q_2\times_3 Q_3,
\end{equation}
where parameter $\eta$ controls how close the generated $B$ to be exactly block-diagonalizable. In particular, for $\eta=0$,
$B=T\times_1 P_1\times_2 P_2\times_3 P_3$ exactly for some block-diagonal $T\in\bbC^{k_1\times k_2\times k_3}$ and
$P_{\ell}\in\STM{k_{\ell}}{n_{\ell}}$ for $1\le\ell\le m=3$.

All experiments are carried out within the MATLAB environment (MATLAB R2022) on a Dell Precision 3660 desktop with an
Intel i9 processor (3200 Mhz), 32 GB memory, running Microsoft Windows 11 Enterprise.

\subsection{Convergence Behavior}
For each of the four cases, real or complex and principal tensor SVD or block-diagonalization, we generate
one set $\{T, E, Q_1, Q_2, Q_3\}$ and construct different tensors according to \eqref{eq:random-B} as $\eta$
varies from $2^{-8}$ down to $2^{-3}$.
Their dimensions are
$$
\mbox{real:}\quad [n_1,n_2,n_3]=[600,   550,   500],
\quad\mbox{and}\,\,\mbox{complex:}\quad [n_1,n_2,n_3]=[480,   440,   400].
$$
These size parameters almost reach the limit of MATLAB's {\tt mat} file data format for saving.
Without saving a tensor after its generation, we can go for larger sizes as we will do in the next subsection.
In the case of making randomly generated $B$ approximately diagonalizable, we use $k=10$,
while for the case of making $B$ approximately block-diagonalizable, we let $t=4$ and
$[k_{1j},k_{2j},k_{3j}]=[2,3,2]$ in \eqref{eq:tau(ki)-n}.

\begin{figure}[t]
{\centering
\begin{tabular}{ccc}
    \resizebox*{0.31\textwidth}{0.17\textheight}{\includegraphics{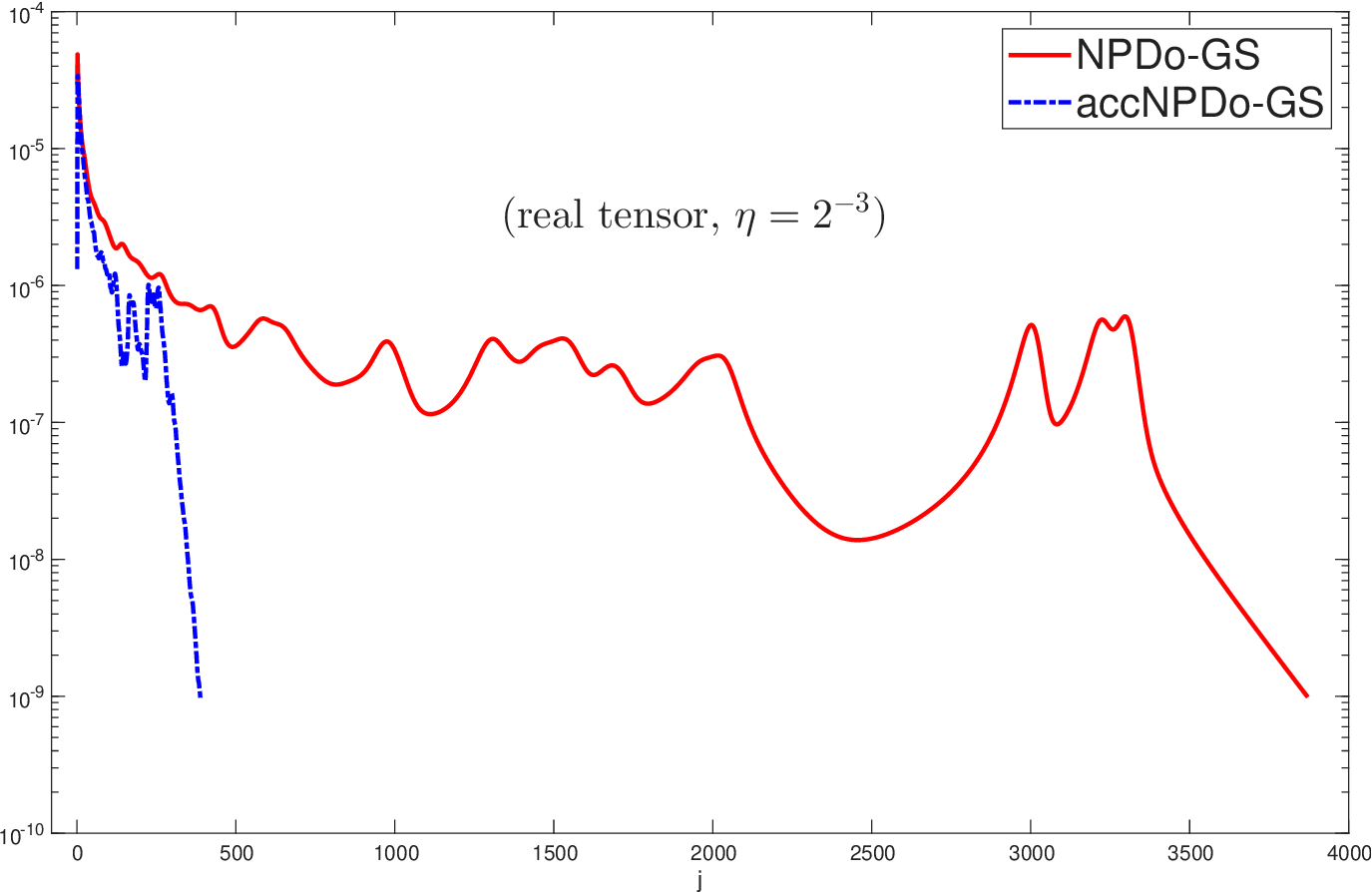}}
  & \resizebox*{0.31\textwidth}{0.17\textheight}{\includegraphics{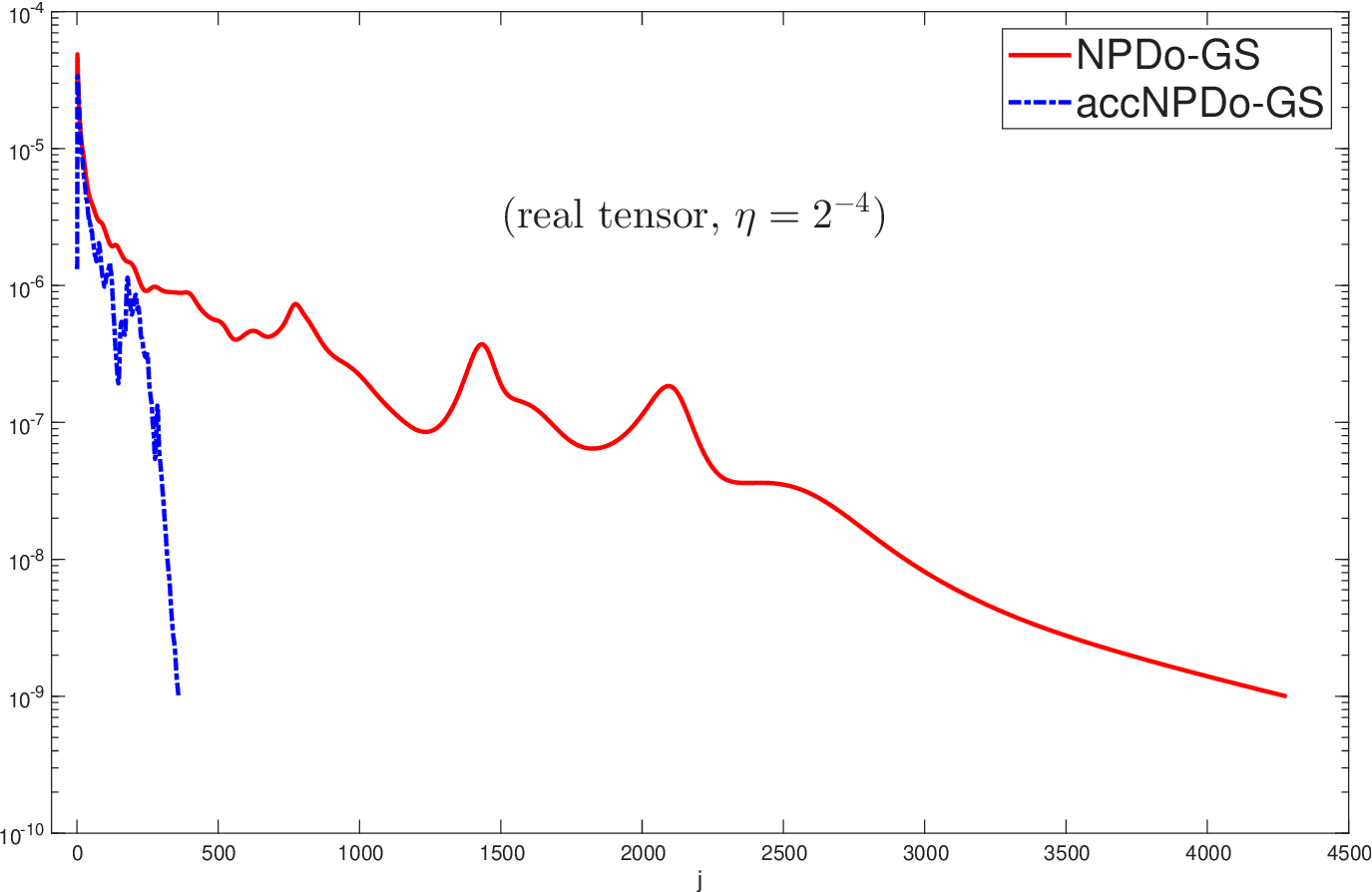}}
  & \resizebox*{0.31\textwidth}{0.17\textheight}{\includegraphics{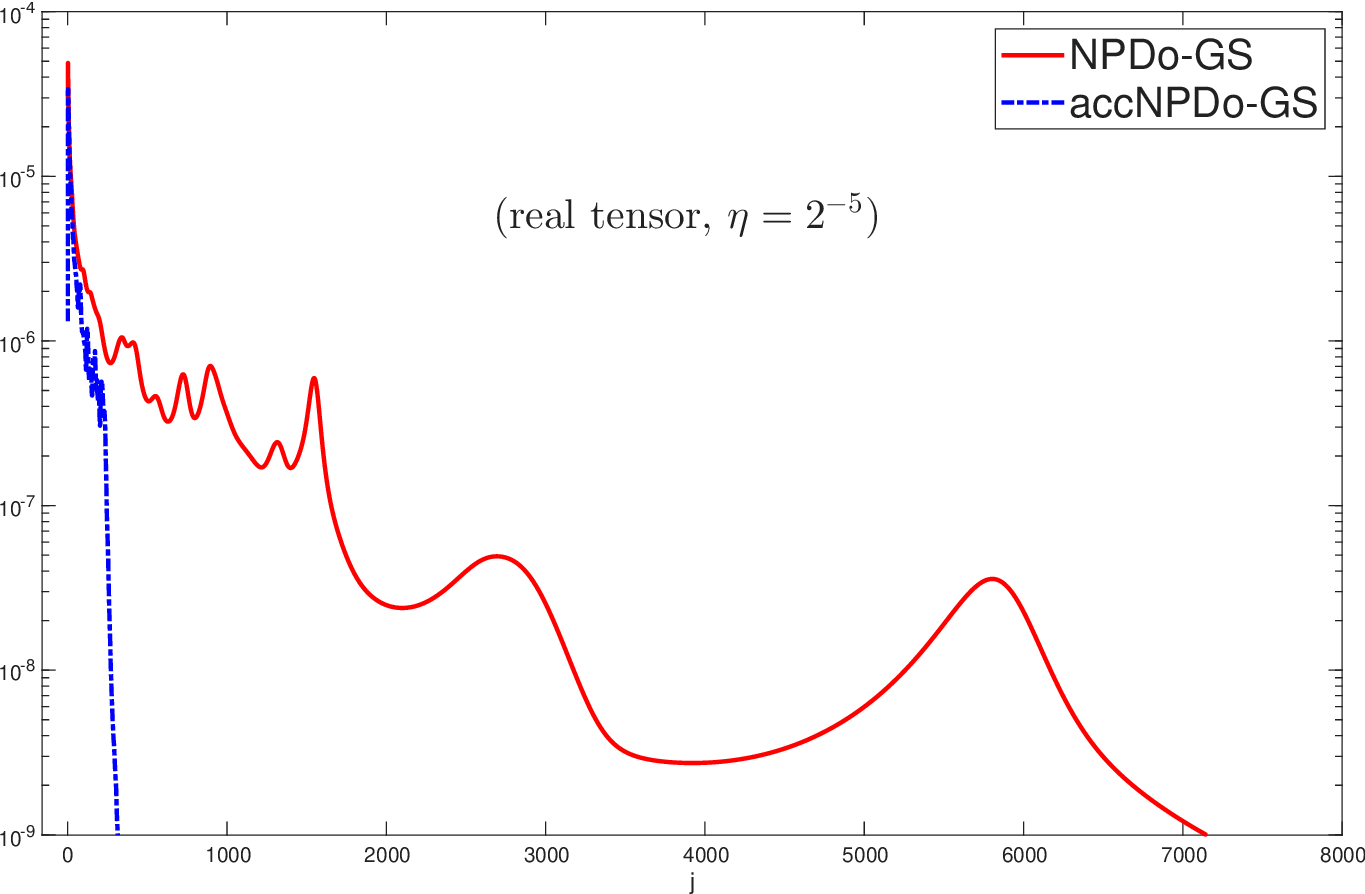}} \\
    \resizebox*{0.31\textwidth}{0.17\textheight}{\includegraphics{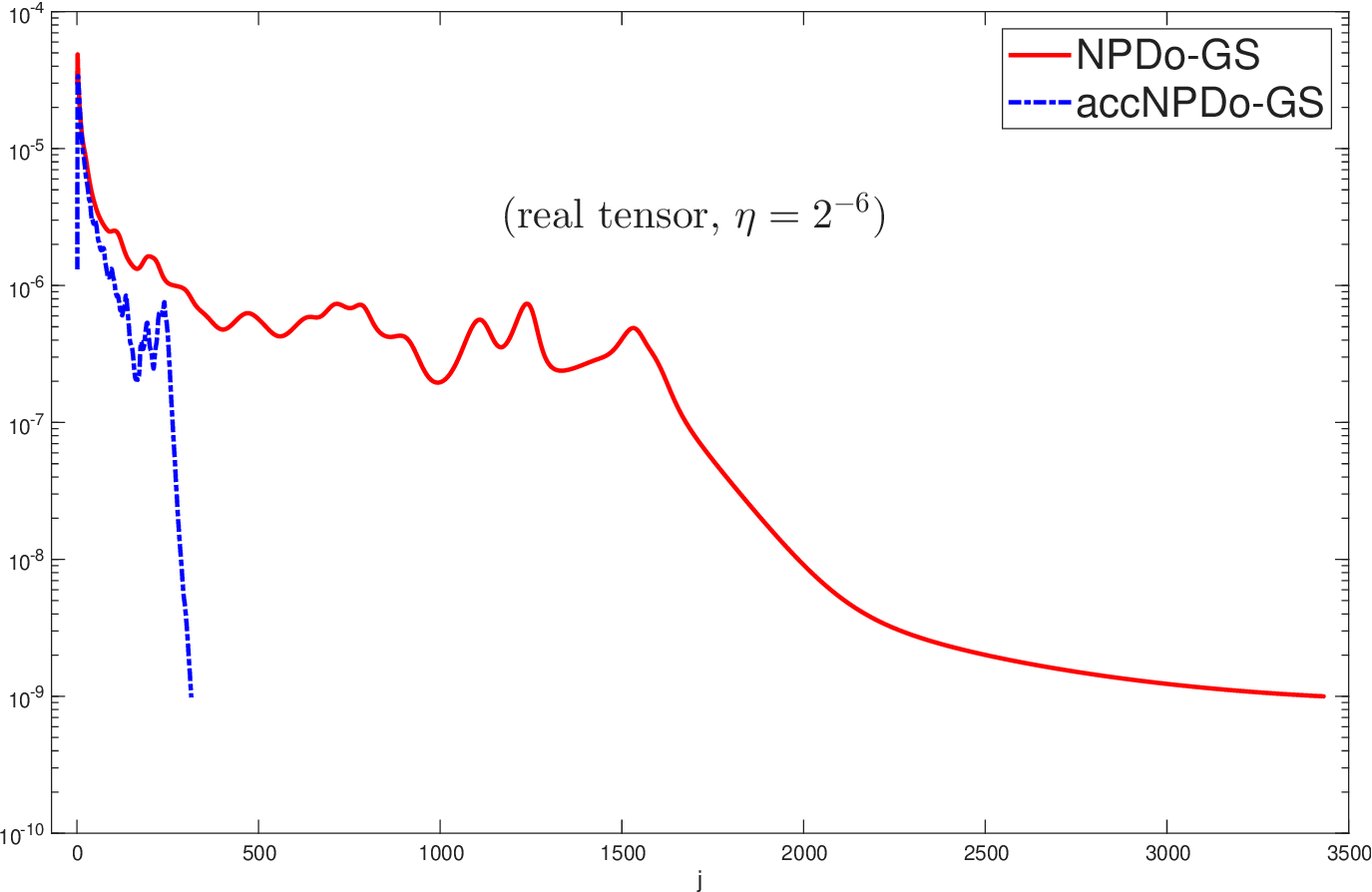}}
  & \resizebox*{0.31\textwidth}{0.17\textheight}{\includegraphics{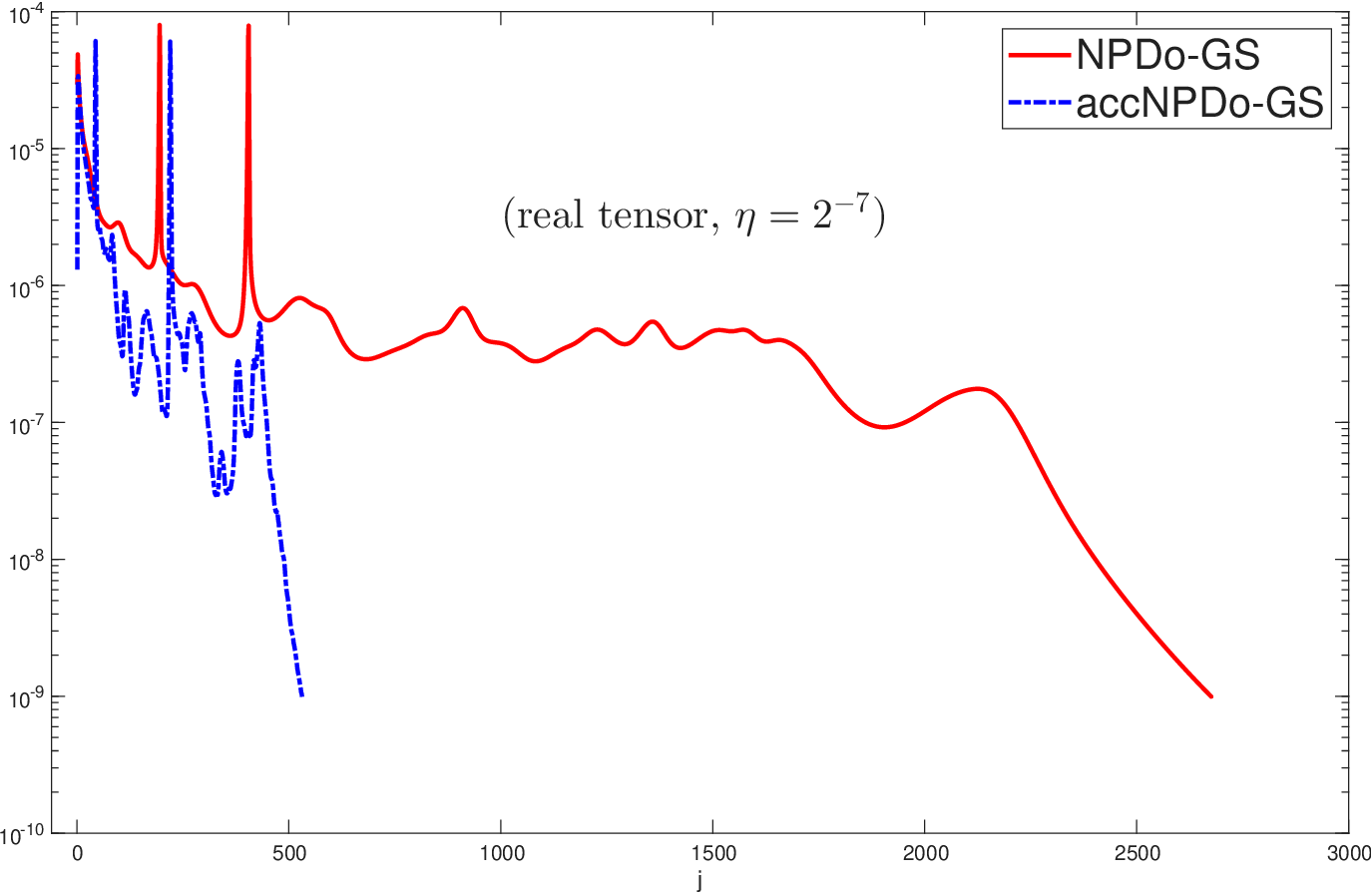}}
  & \resizebox*{0.31\textwidth}{0.17\textheight}{\includegraphics{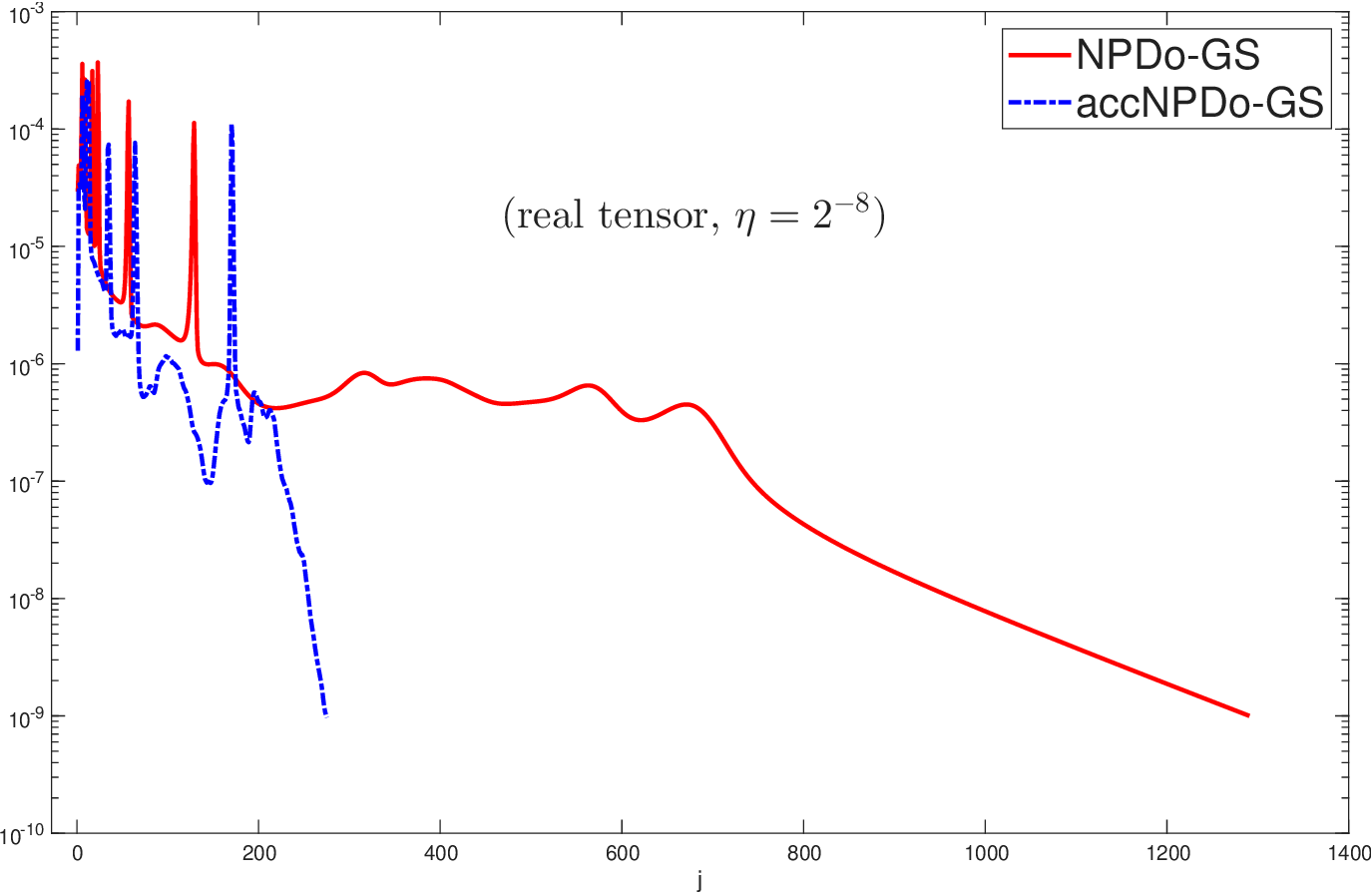}} \\ 
    \resizebox*{0.31\textwidth}{0.17\textheight}{\includegraphics{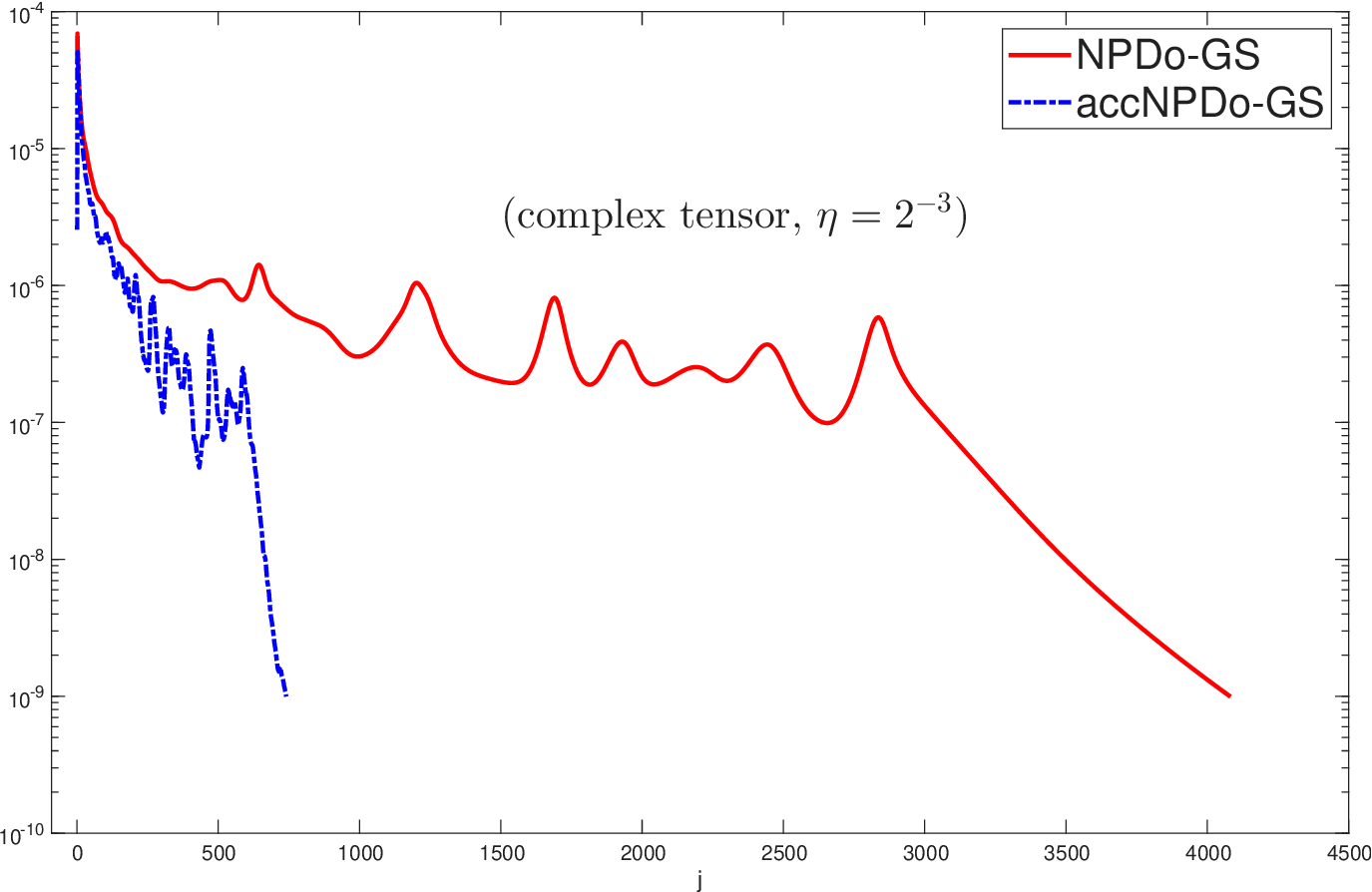}}
  & \resizebox*{0.31\textwidth}{0.17\textheight}{\includegraphics{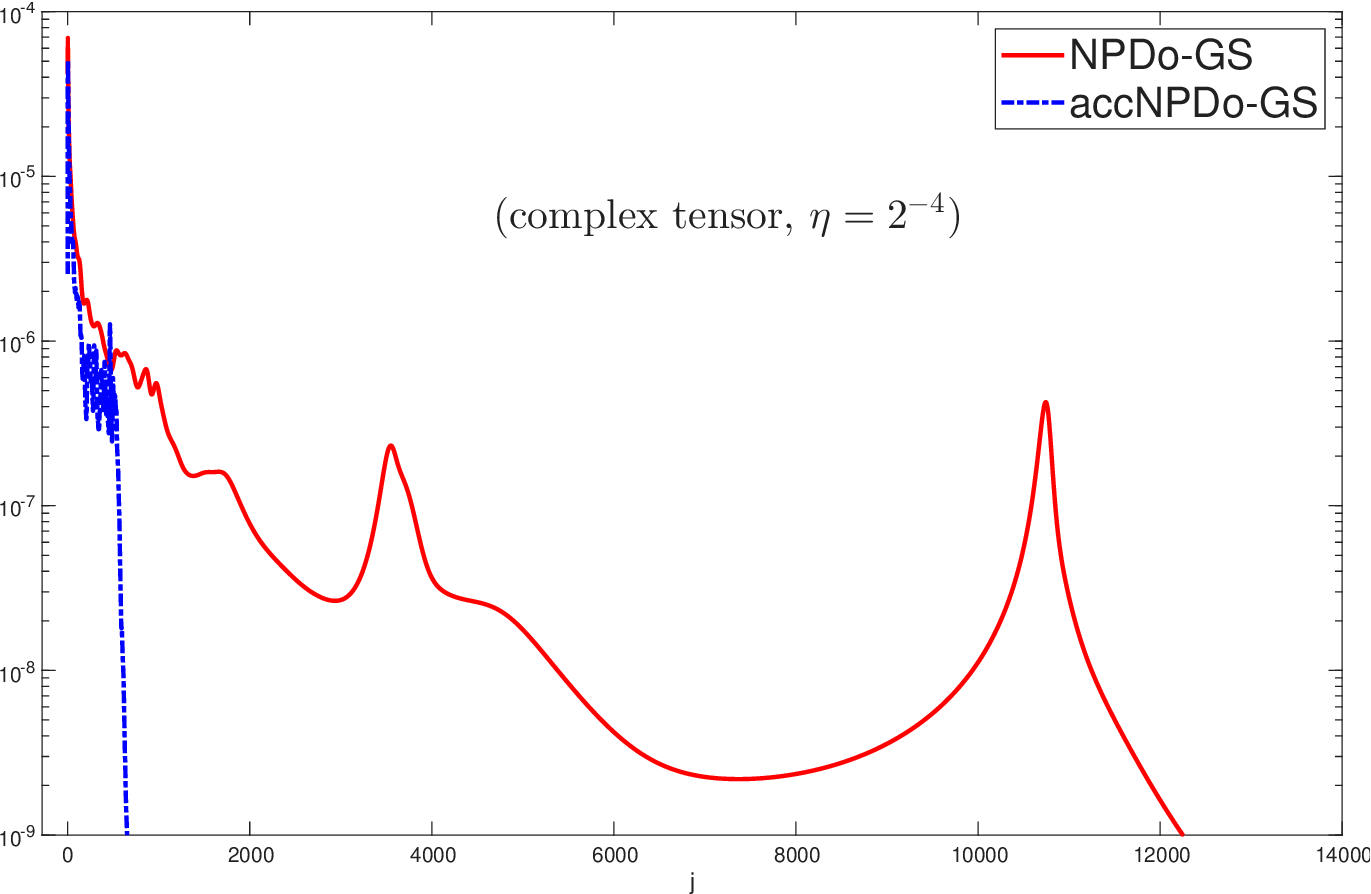}}
  & \resizebox*{0.31\textwidth}{0.17\textheight}{\includegraphics{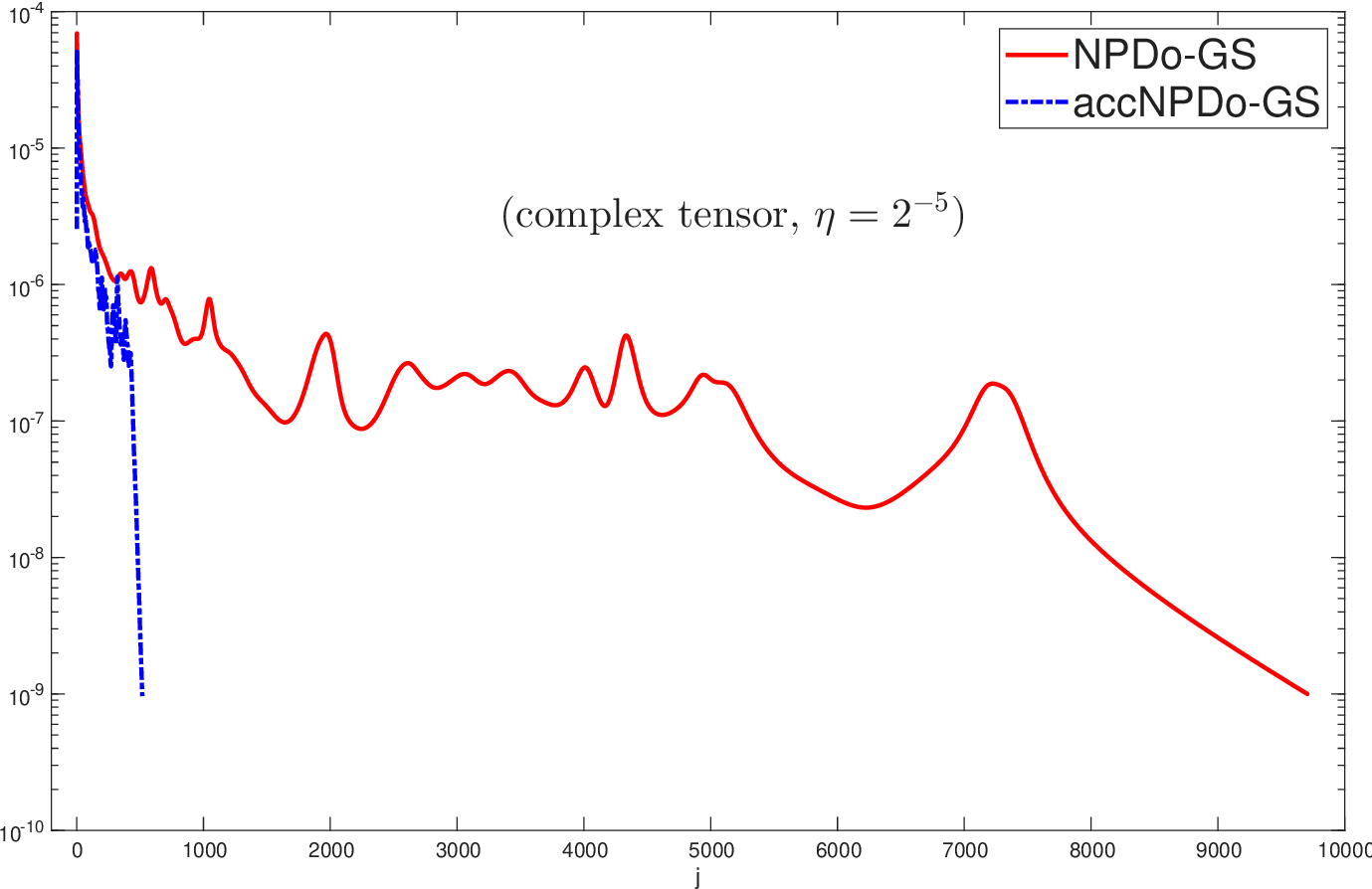}} \\
    \resizebox*{0.31\textwidth}{0.17\textheight}{\includegraphics{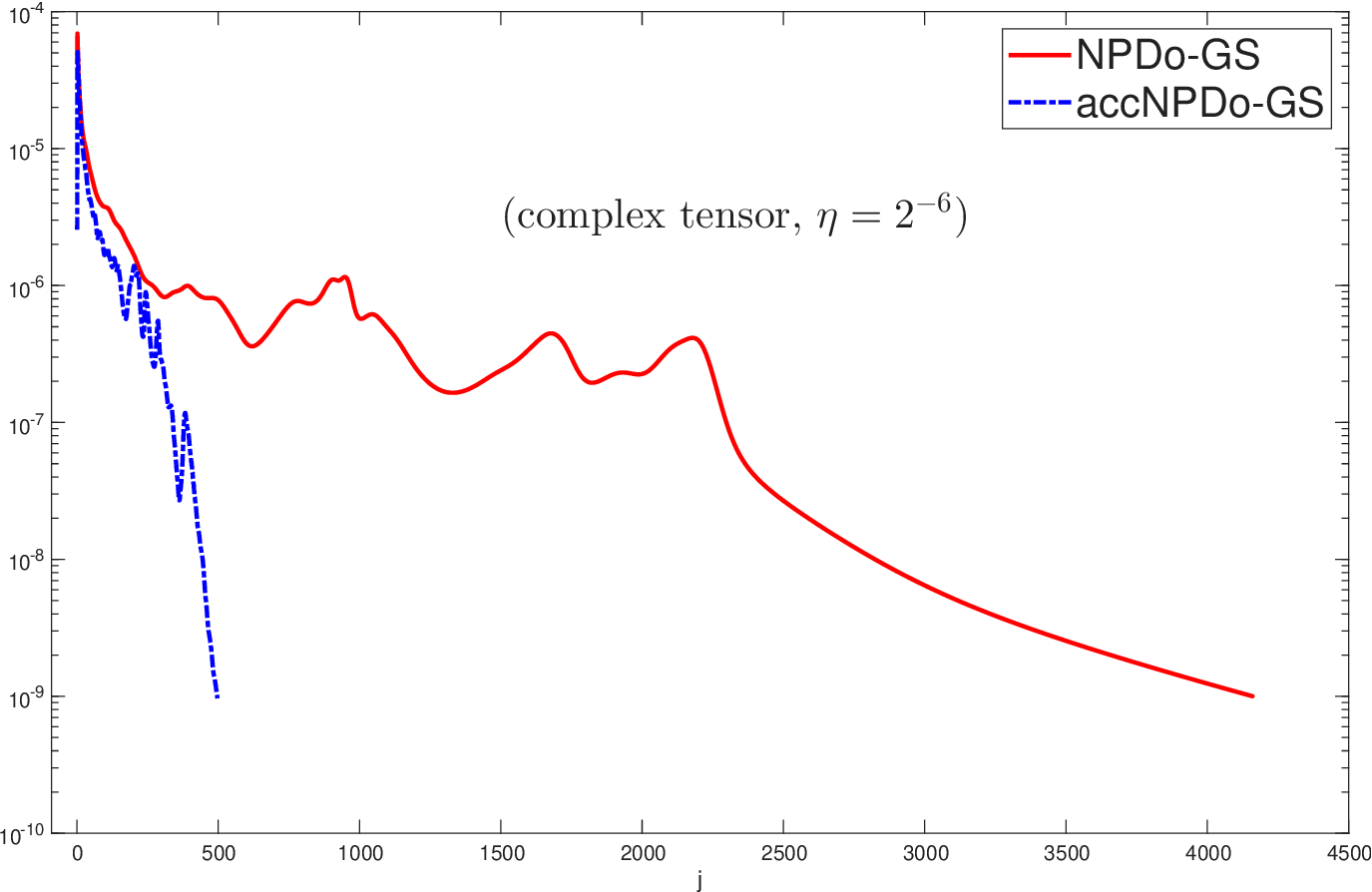}}
  & \resizebox*{0.31\textwidth}{0.17\textheight}{\includegraphics{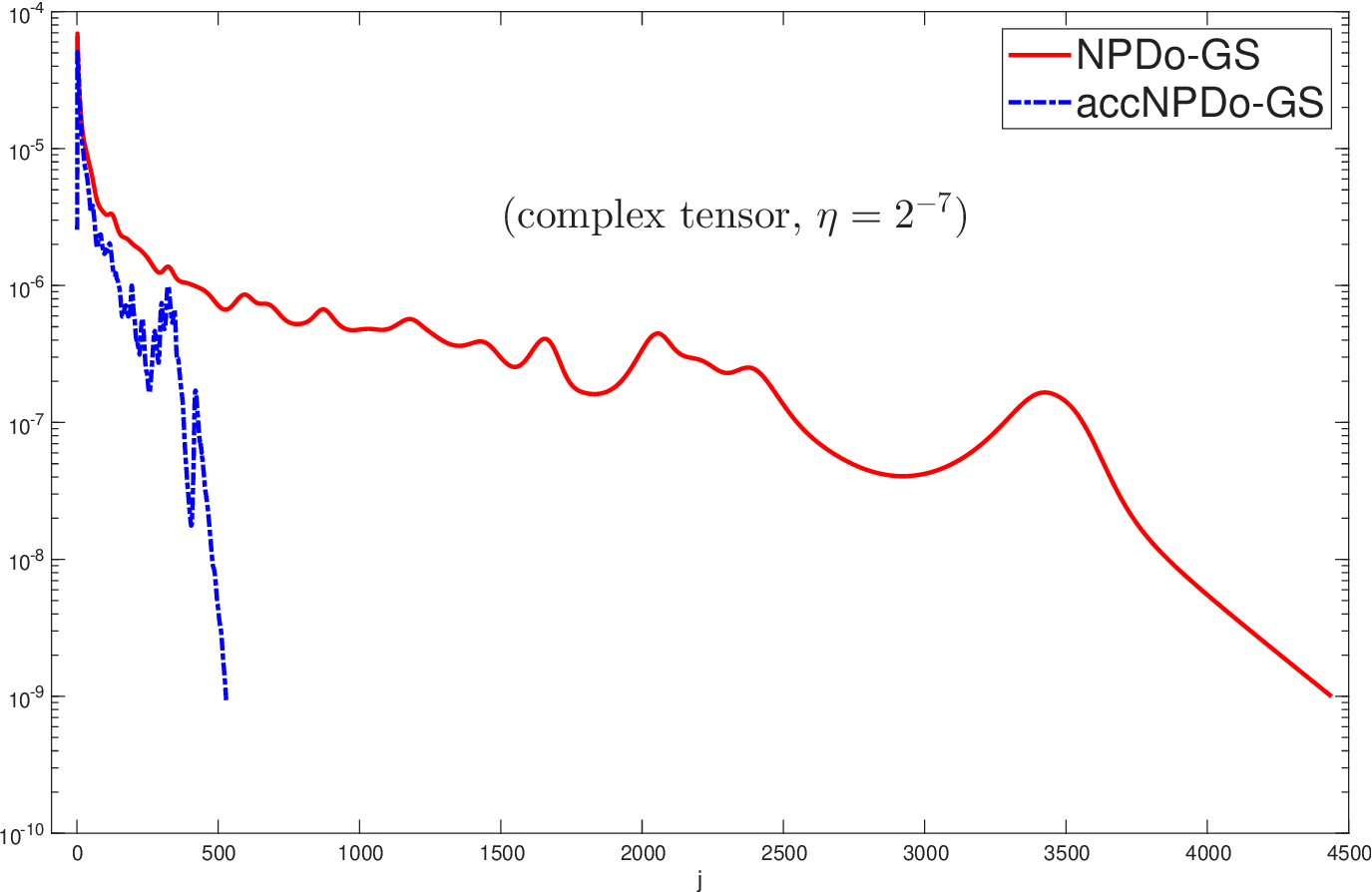}}
  & \resizebox*{0.31\textwidth}{0.17\textheight}{\includegraphics{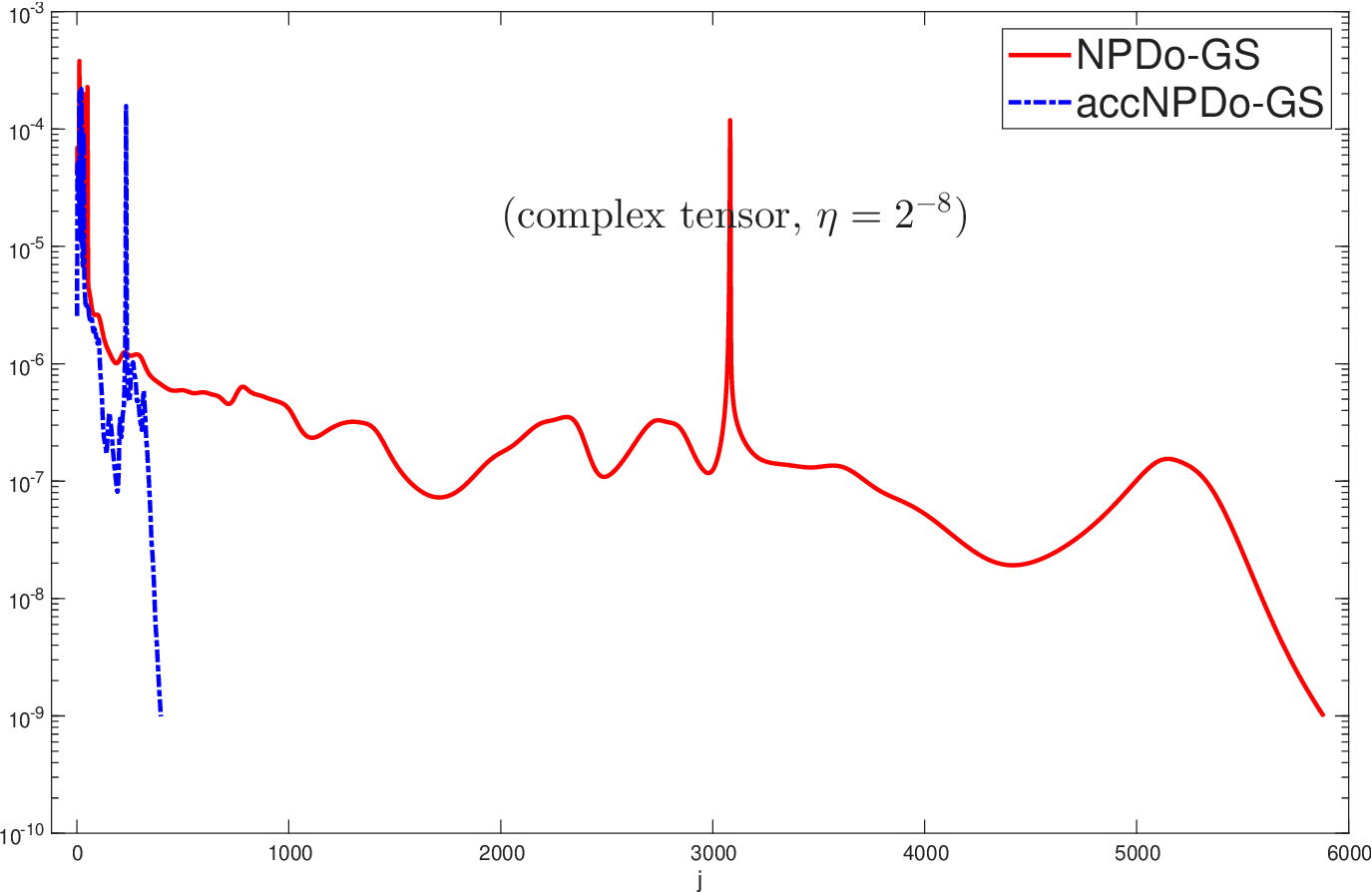}}
\end{tabular}\par
}
\vspace{-0.15 cm}
\caption{\small Principal tensor SVD (\ptsvd) by the NPDo approach -- convergence  in terms of KKT residual $\tilde\epsilon_{\KKT,j}$ as defined in \eqref{eq:stop-BTSVD'} on
  randomly generated tensors according to \eqref{eq:random-B} with varying $\eta$ from $2^{-8}$ down to $2^{-3}$:
  top two rows are for real $B\in\bbR^{500\times 550\times 600}$ and bottom two rows for complex
  $B\in\bbC^{400\times 440\times 480}$.
  }
\label{fig:diagTS-behavior-cvg}
\end{figure}

\begin{figure}[t]
{\centering
\begin{tabular}{ccc}
    \resizebox*{0.31\textwidth}{0.17\textheight}{\includegraphics{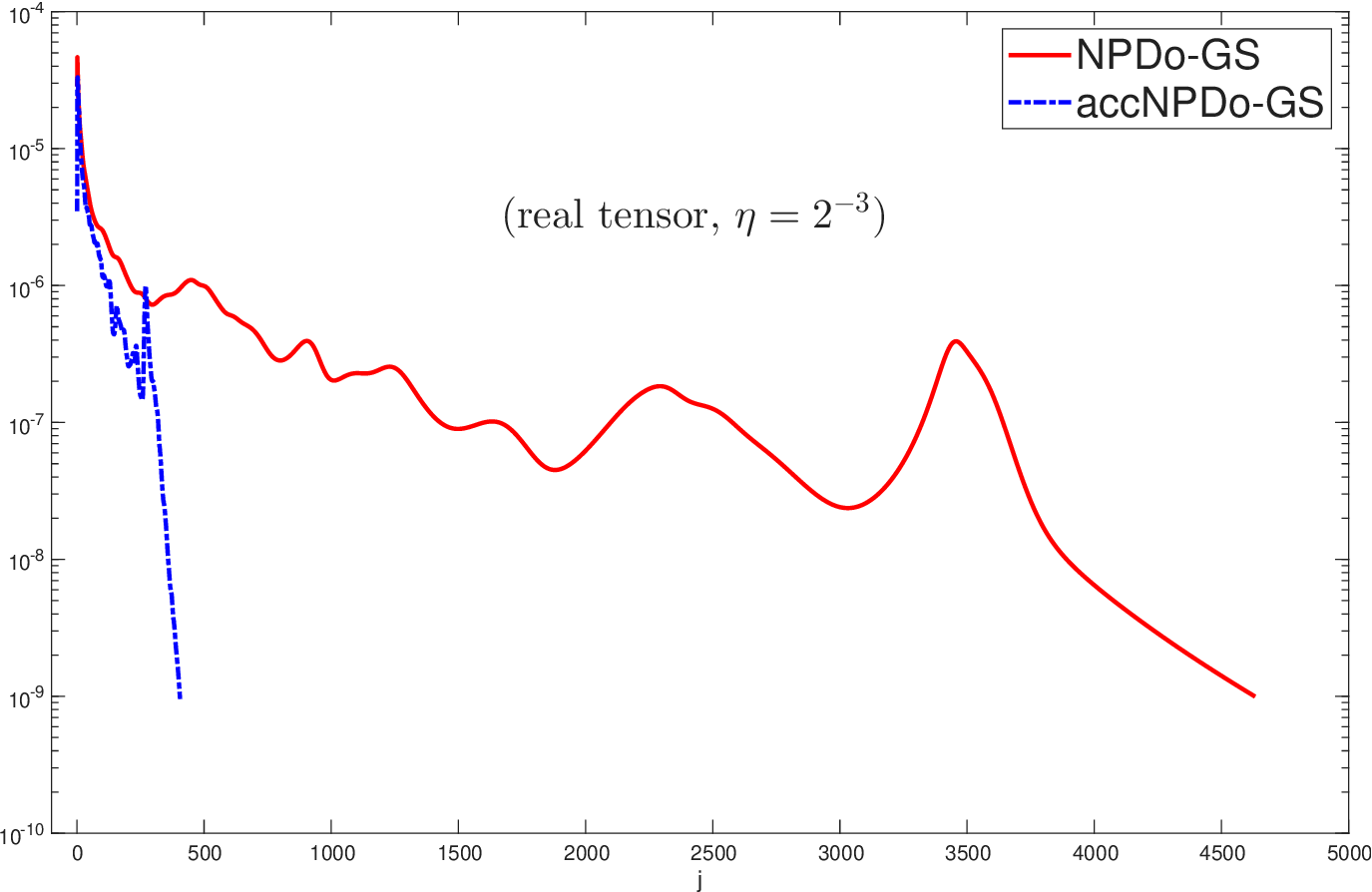}}
  & \resizebox*{0.31\textwidth}{0.17\textheight}{\includegraphics{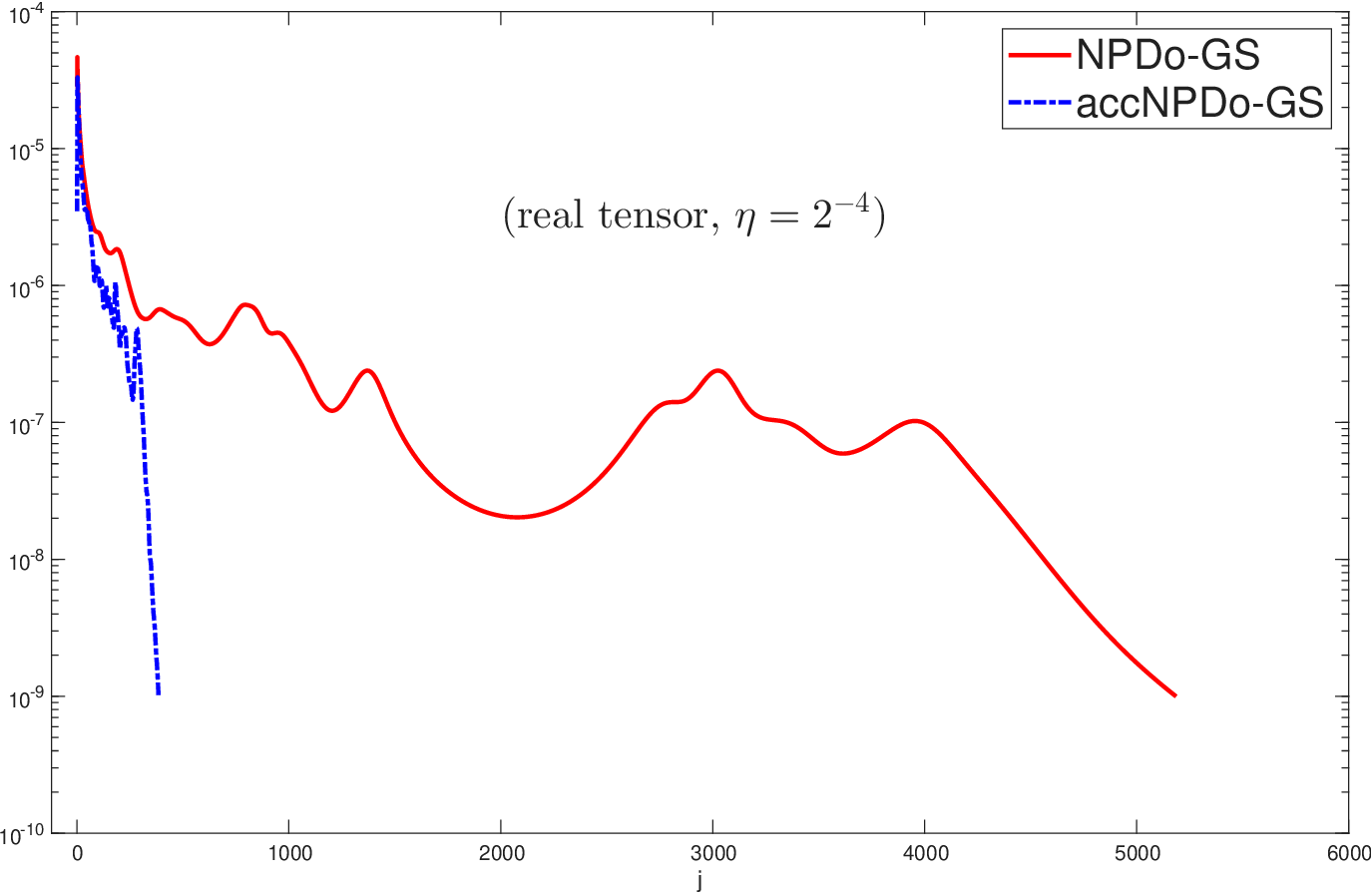}}
  & \resizebox*{0.31\textwidth}{0.17\textheight}{\includegraphics{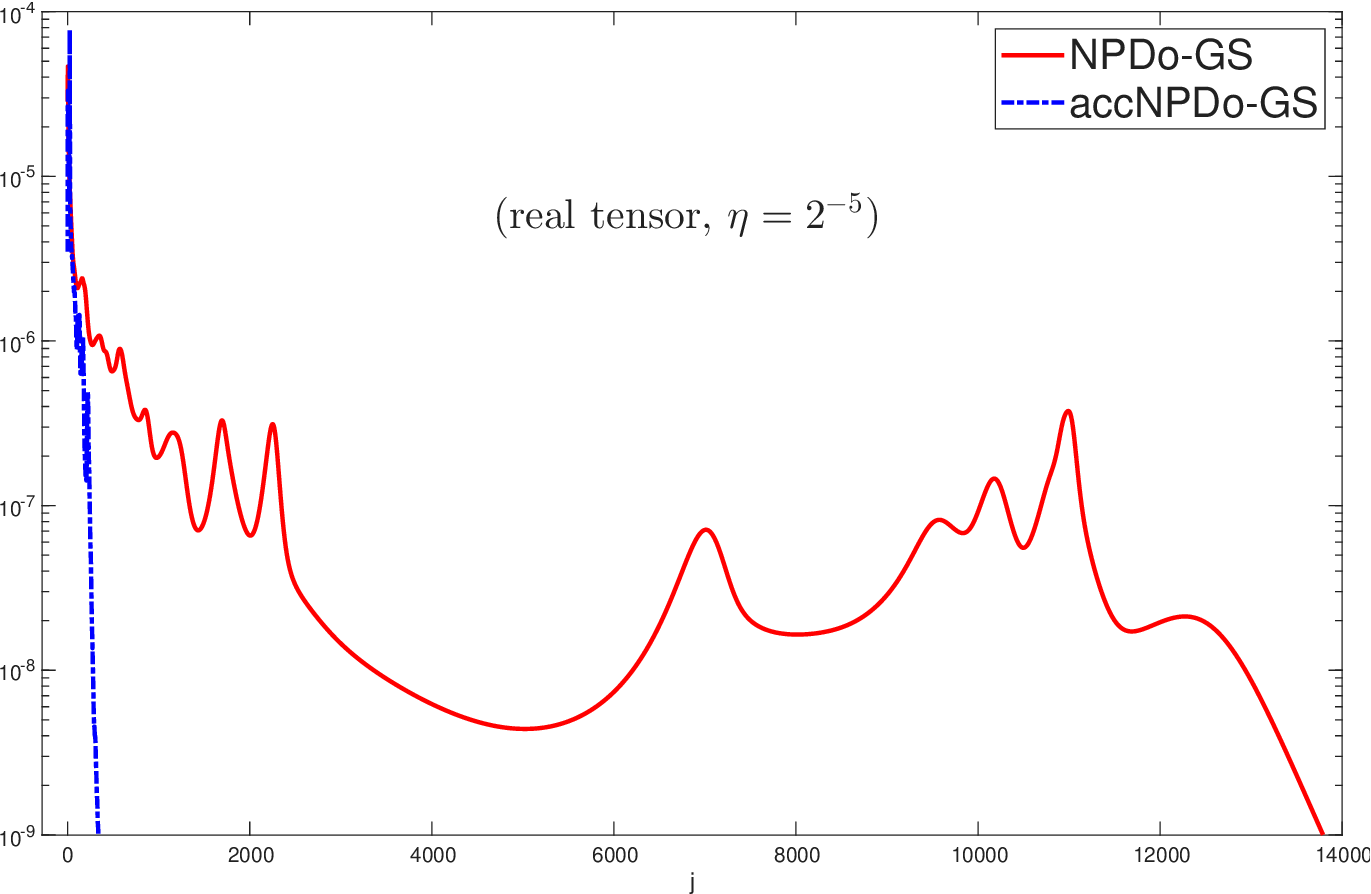}} \\
    \resizebox*{0.31\textwidth}{0.17\textheight}{\includegraphics{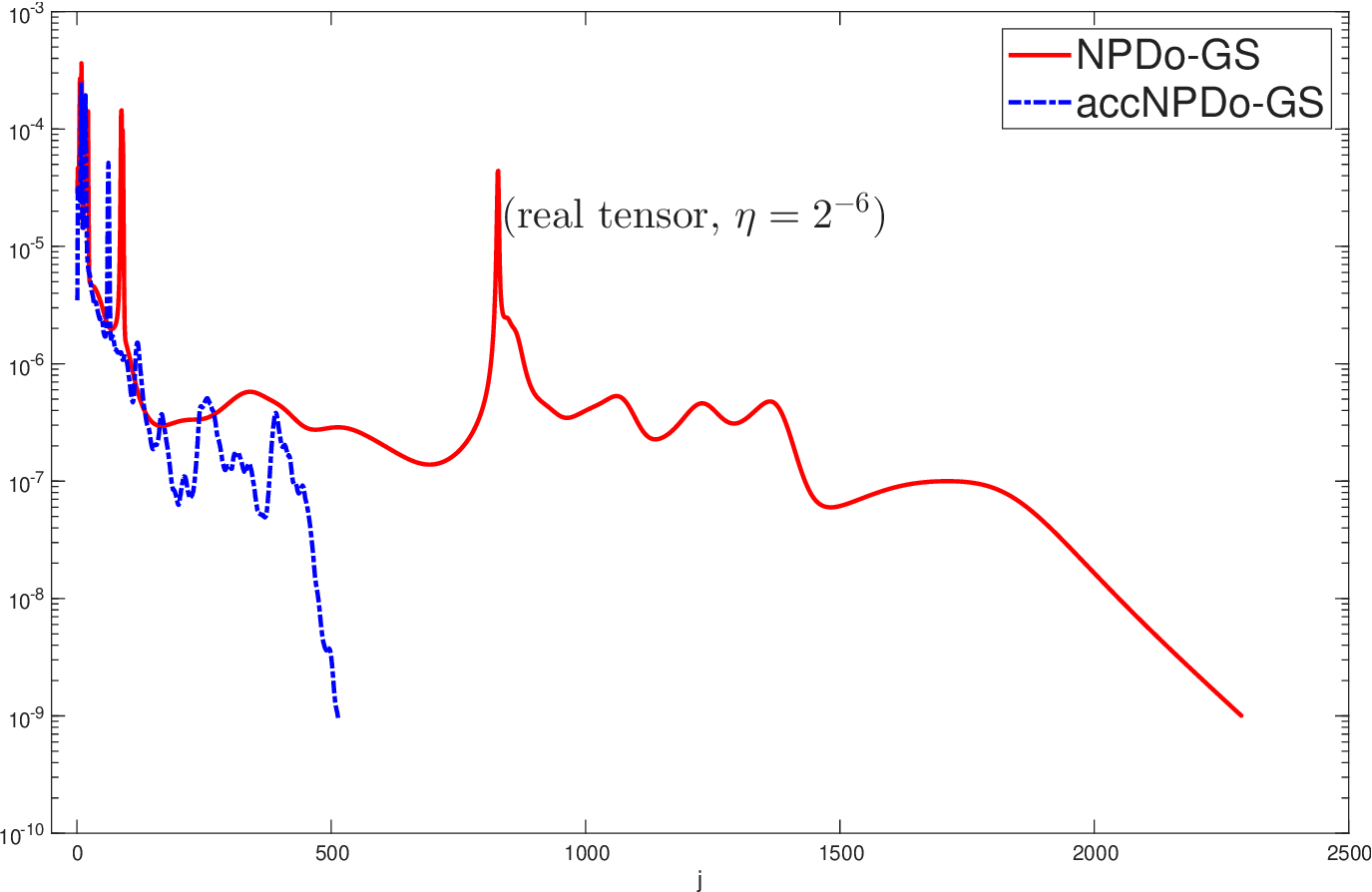}}
  & \resizebox*{0.31\textwidth}{0.17\textheight}{\includegraphics{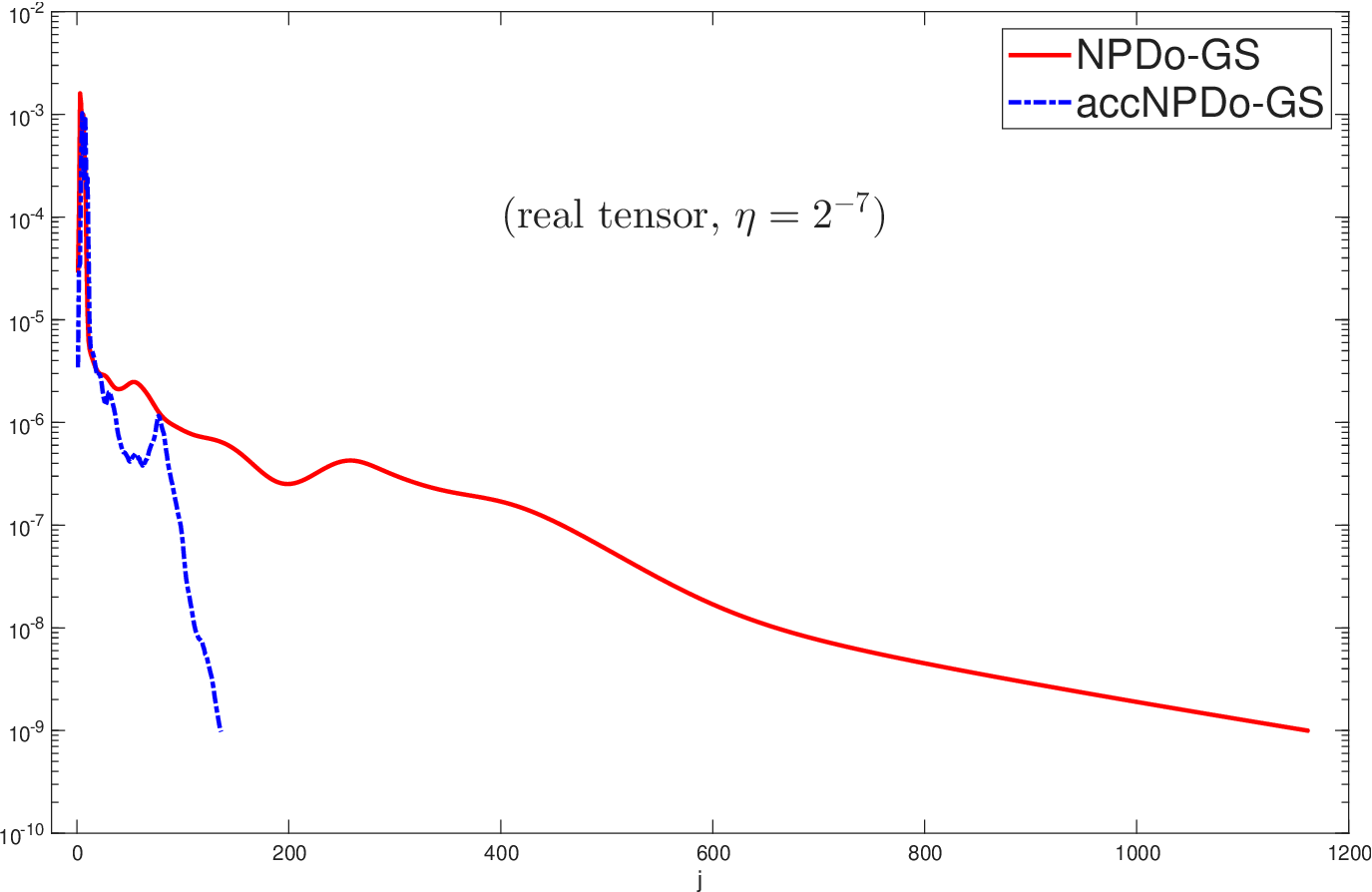}}
  & \resizebox*{0.31\textwidth}{0.17\textheight}{\includegraphics{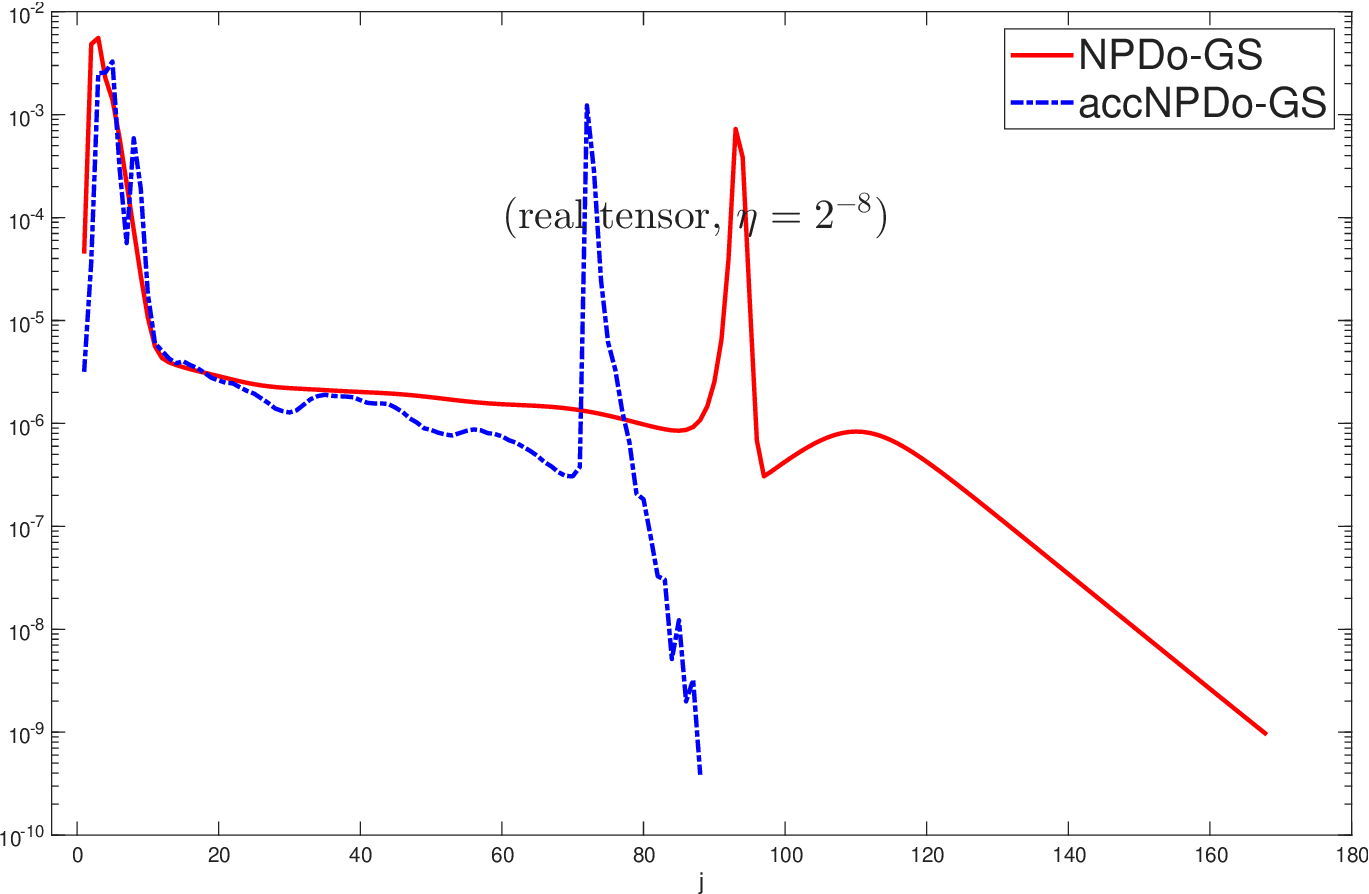}} \\ \hline
    \resizebox*{0.31\textwidth}{0.17\textheight}{\includegraphics{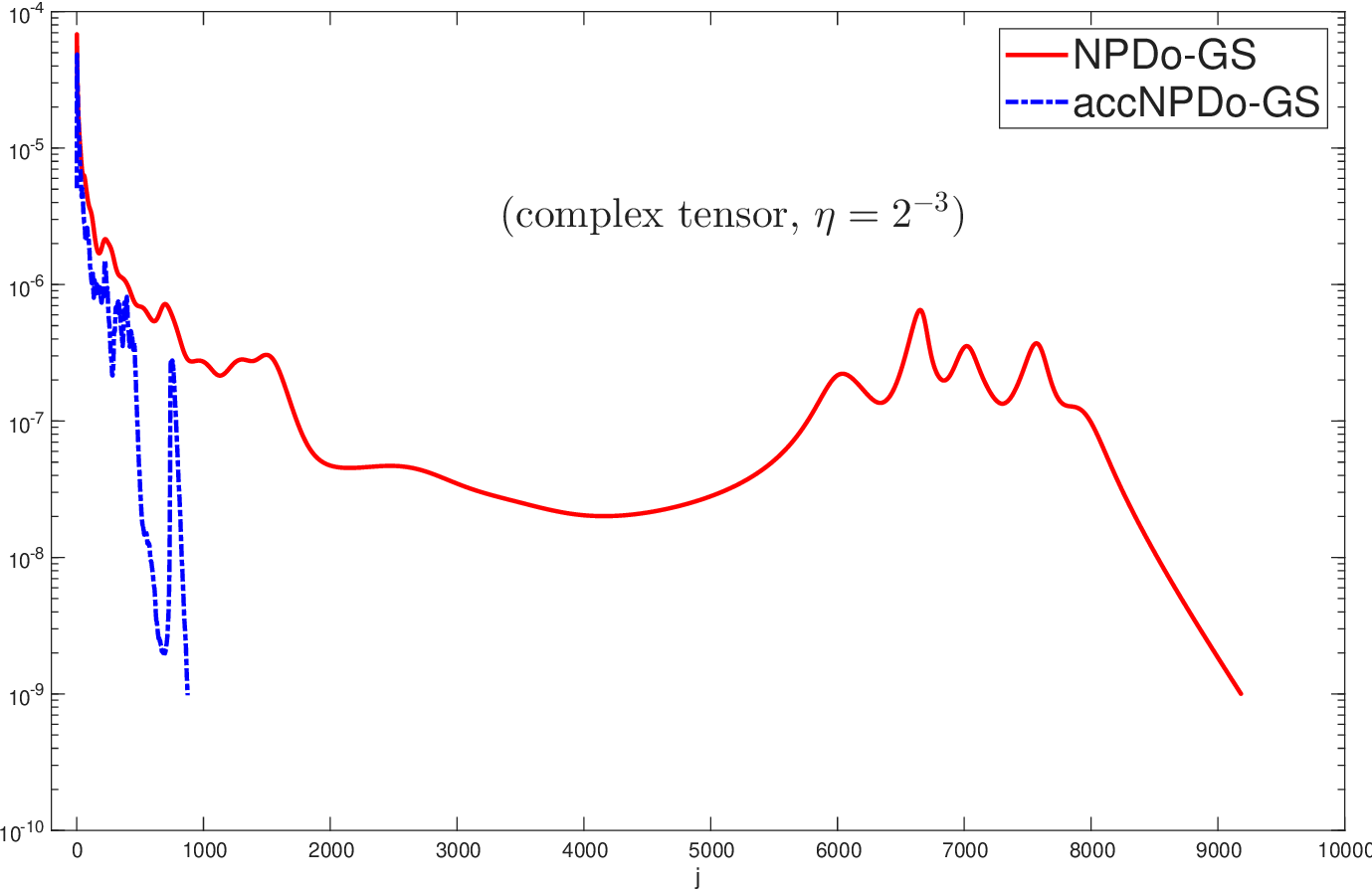}}
  & \resizebox*{0.31\textwidth}{0.17\textheight}{\includegraphics{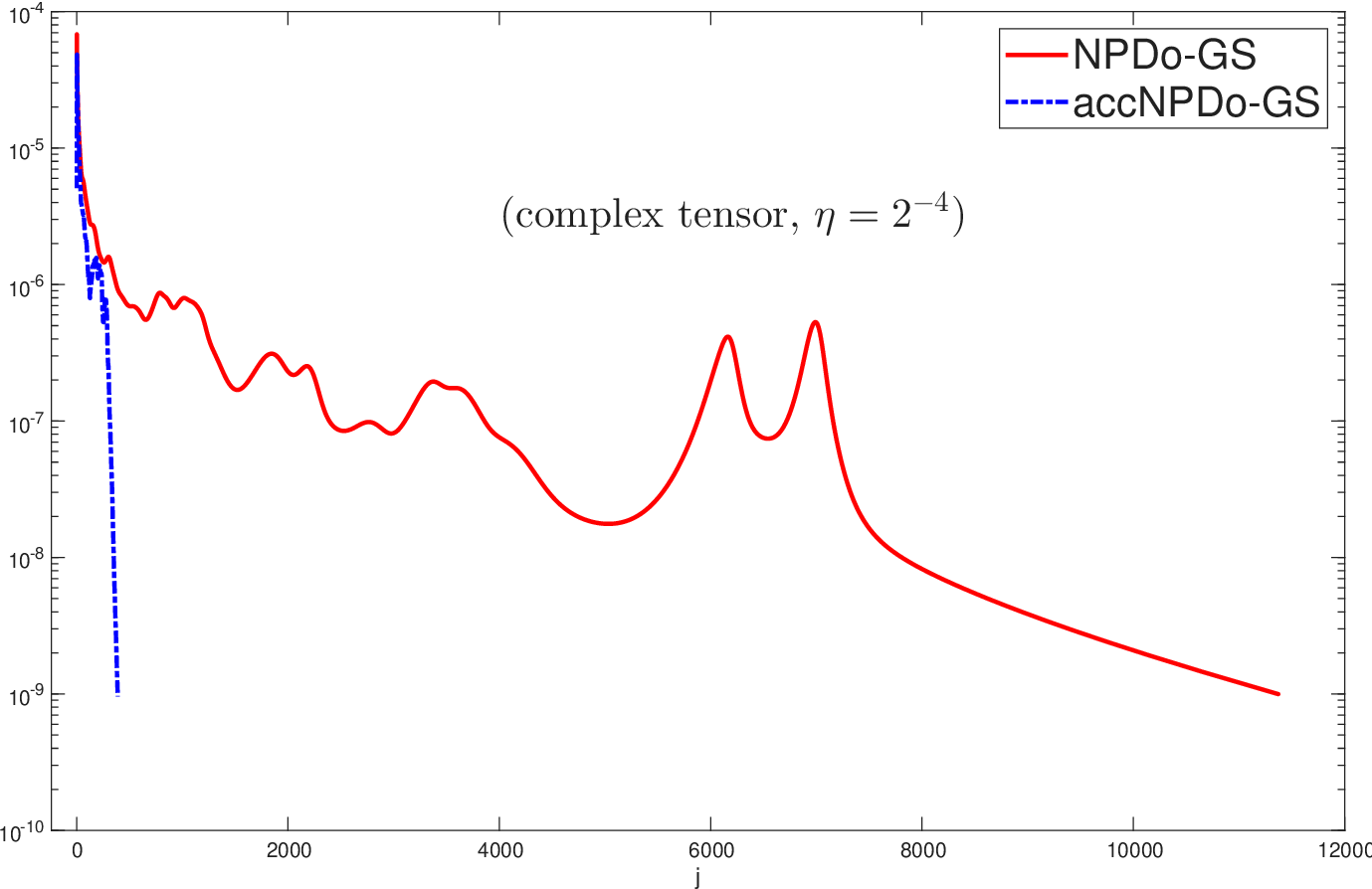}}
  & \resizebox*{0.31\textwidth}{0.17\textheight}{\includegraphics{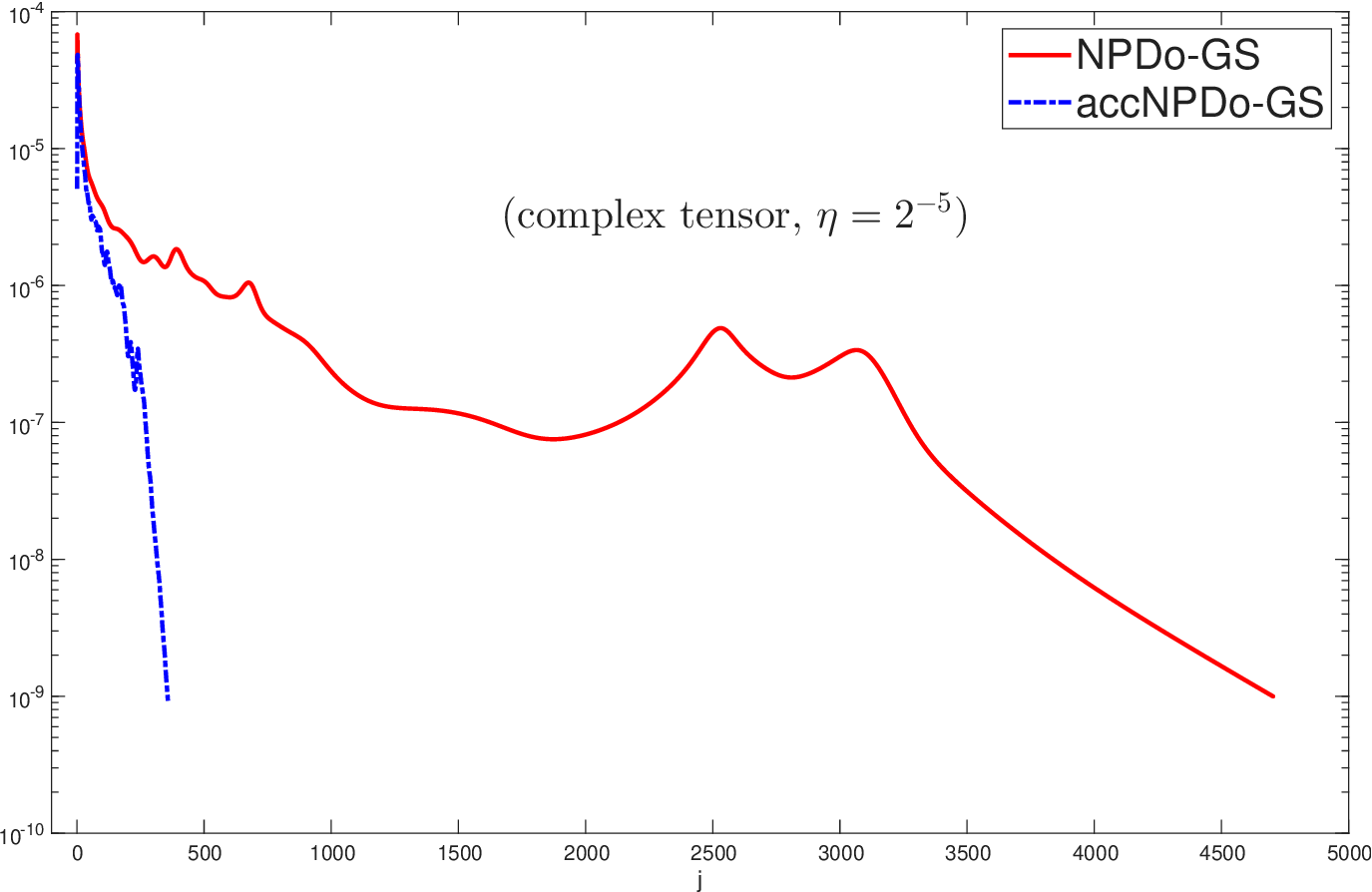}} \\
    \resizebox*{0.31\textwidth}{0.17\textheight}{\includegraphics{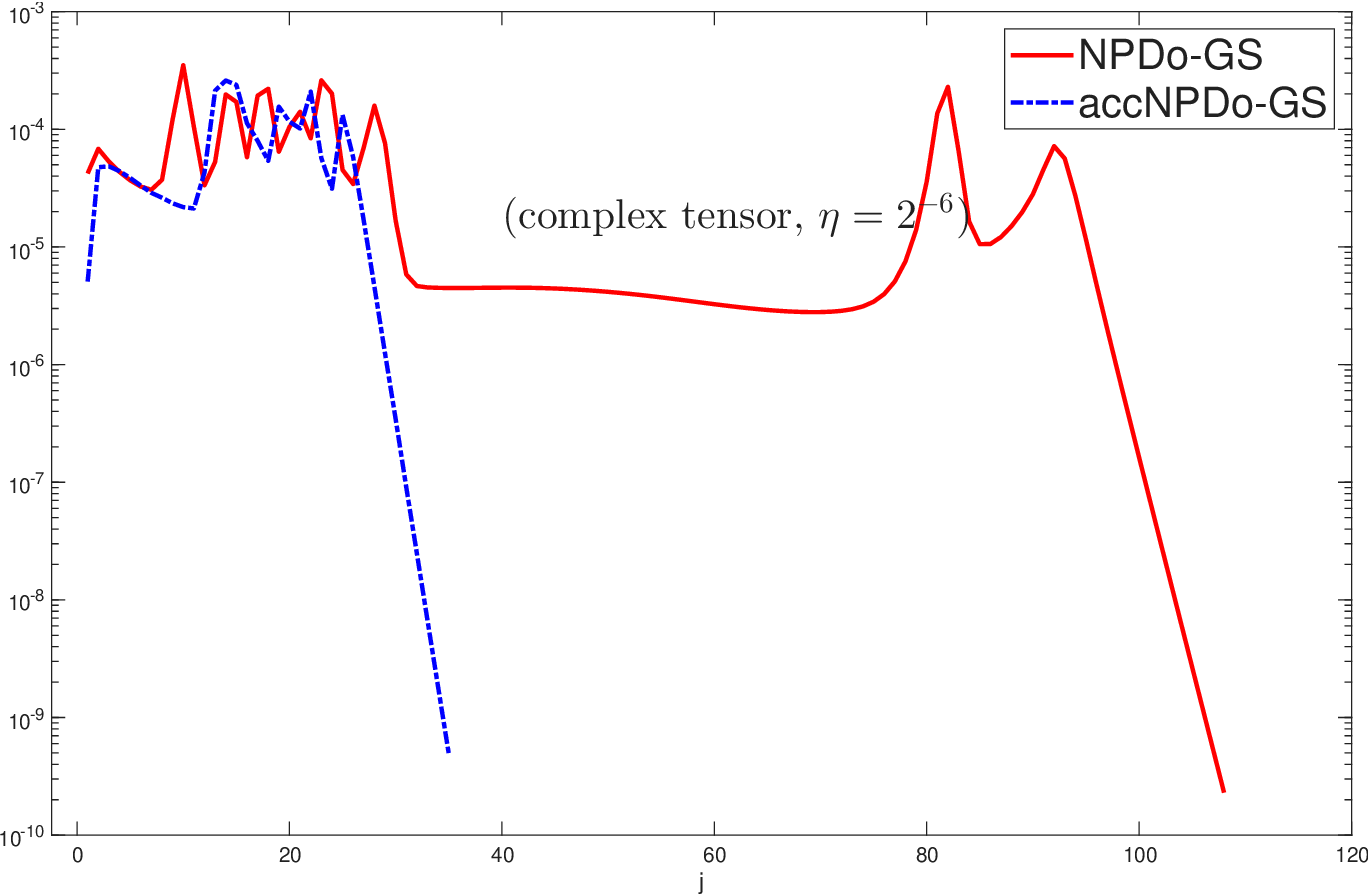}}
  & \resizebox*{0.31\textwidth}{0.17\textheight}{\includegraphics{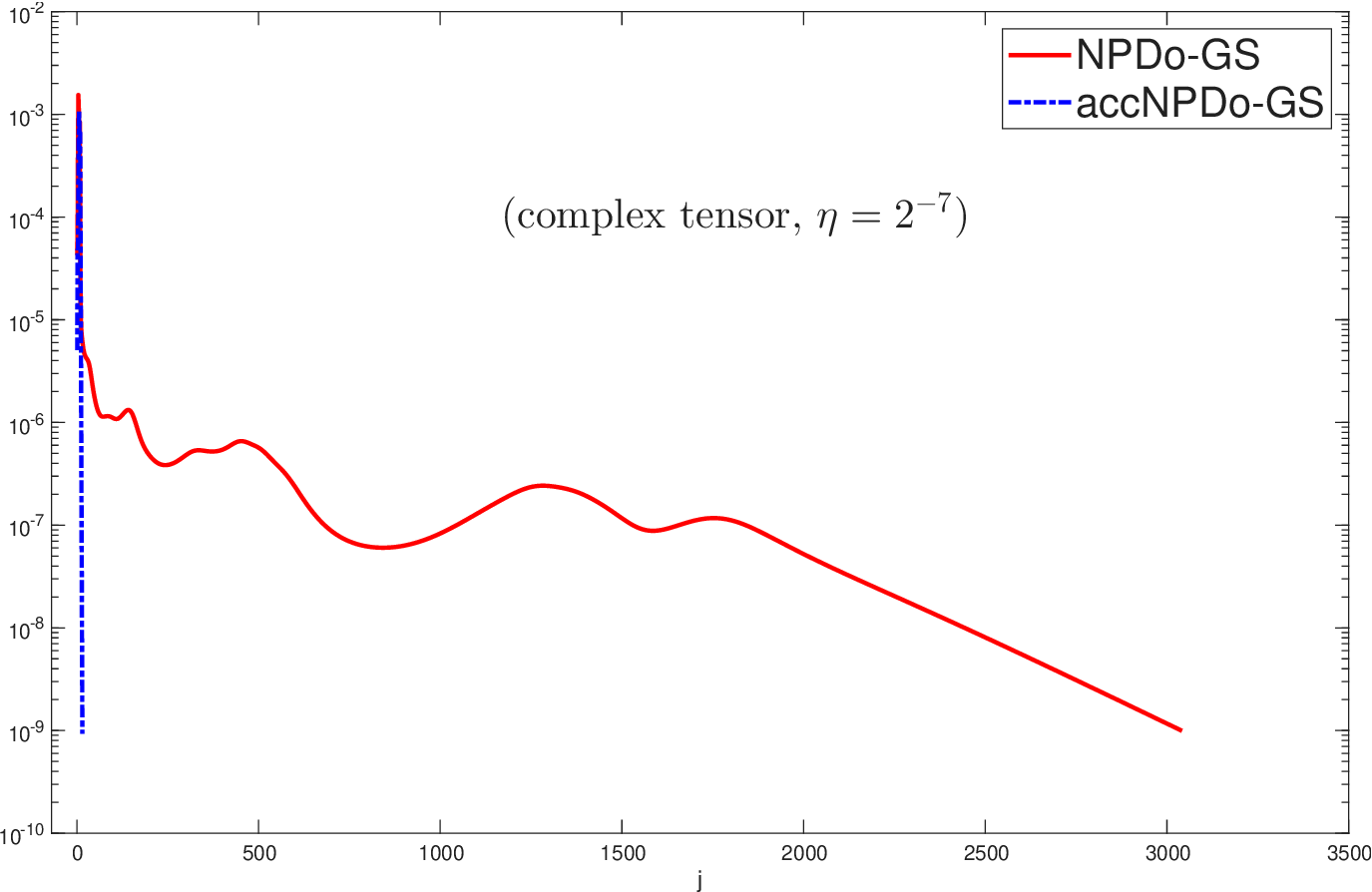}}
  & \resizebox*{0.31\textwidth}{0.17\textheight}{\includegraphics{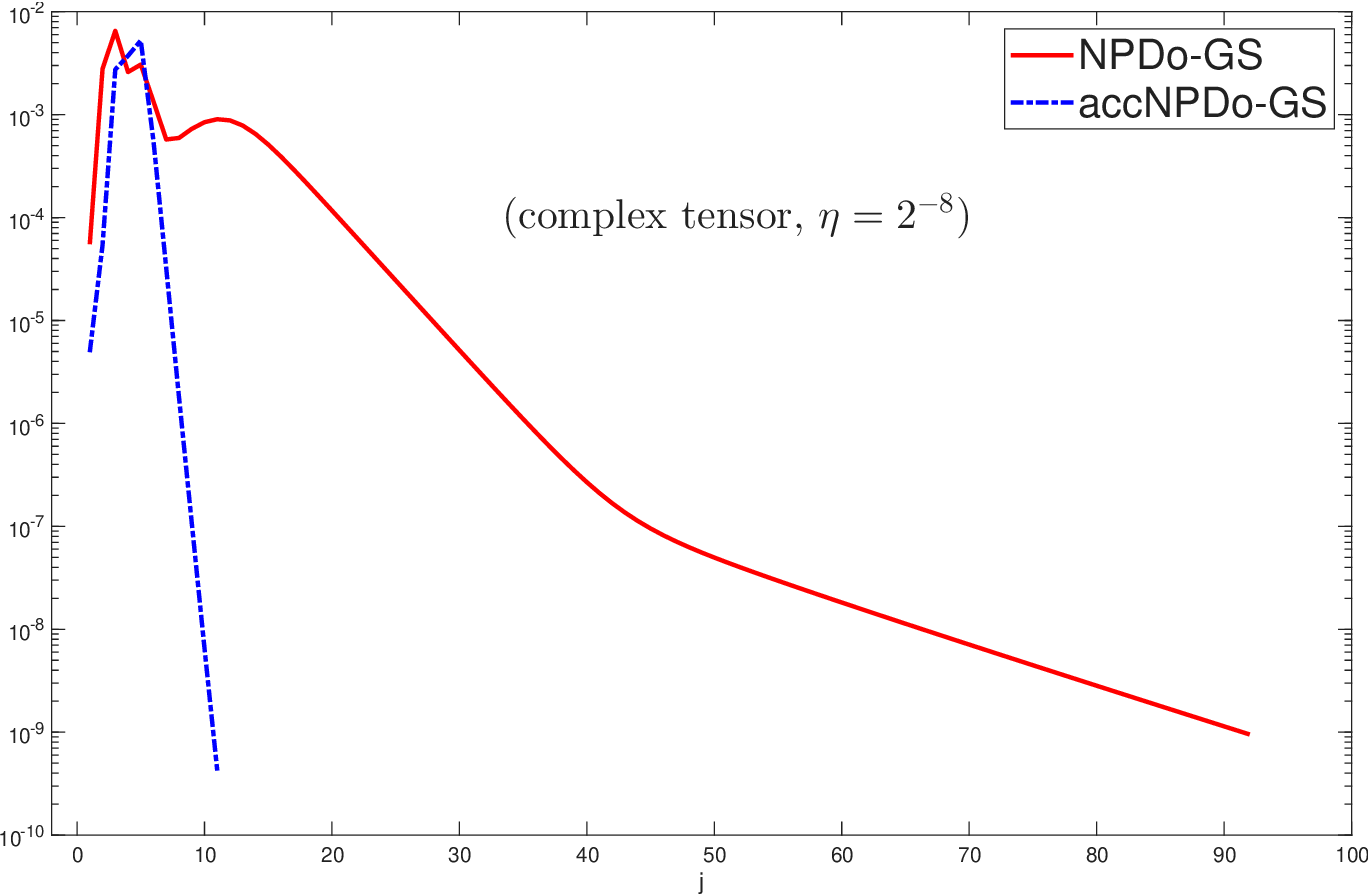}}
\end{tabular}\par
}
\vspace{-0.15 cm}
\caption{\small Principal tensor block-diagonalization (\ptbd) by the NPDo approach --  convergence  in terms of KKT residual $\tilde\epsilon_{\KKT,j}$ as defined in \eqref{eq:stop-BTSVD'} on
  randomly generated tensors according to \eqref{eq:random-B} with varying $\eta$ from $2^{-8}$ down to $2^{-3}$:
  top two rows are for real $B\in\bbR^{500\times 550\times 600}$ and bottom two rows for complex
  $B\in\bbC^{400\times 440\times 480}$.
  }
\label{fig:bdiagTS-behavior-cvg}
\end{figure}

\begin{figure}[t]
{\centering
\begin{tabular}{ccc}
    \resizebox*{0.31\textwidth}{0.17\textheight}{\includegraphics{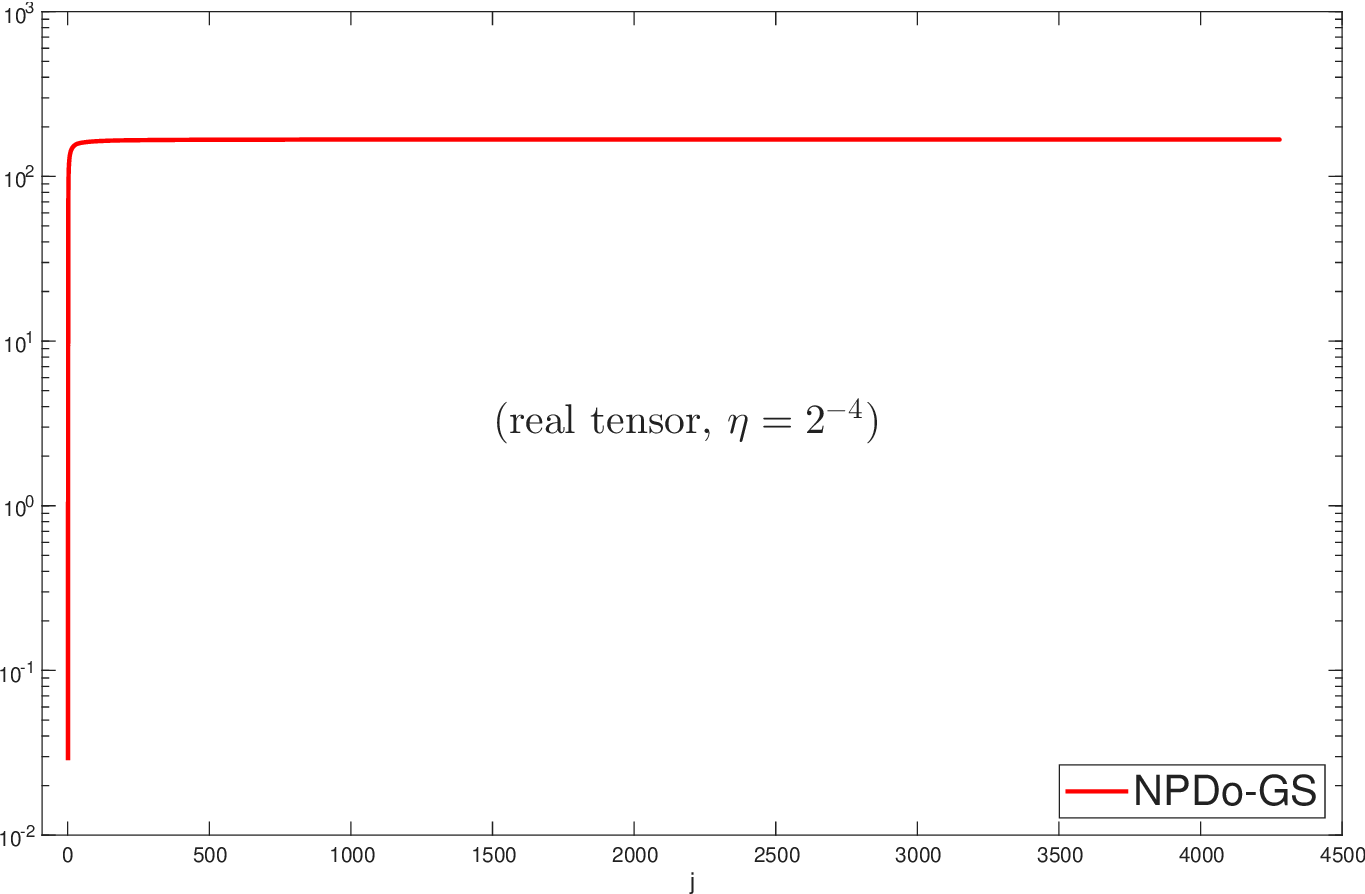}}
  & \resizebox*{0.31\textwidth}{0.17\textheight}{\includegraphics{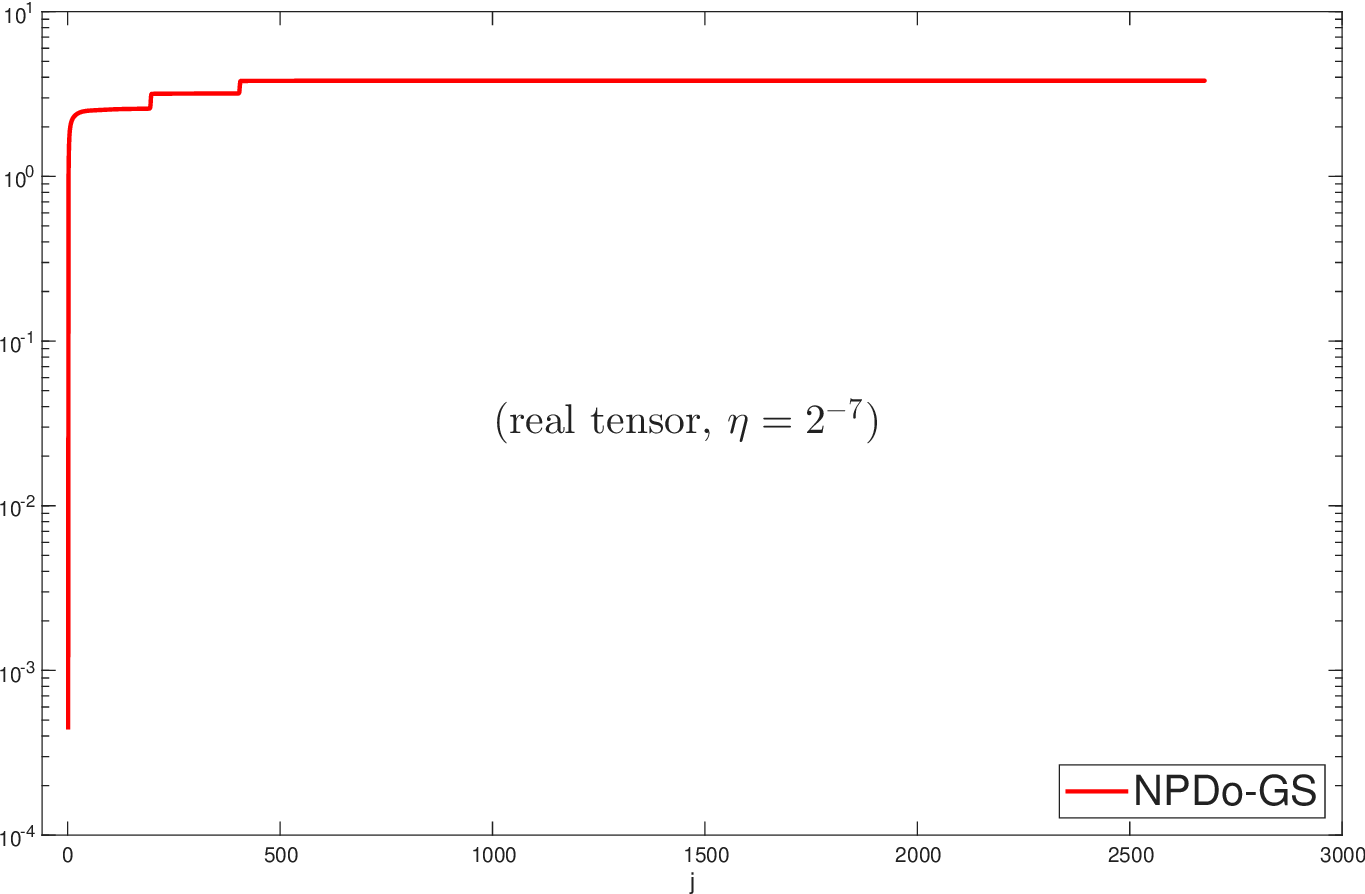}}
  & \resizebox*{0.31\textwidth}{0.17\textheight}{\includegraphics{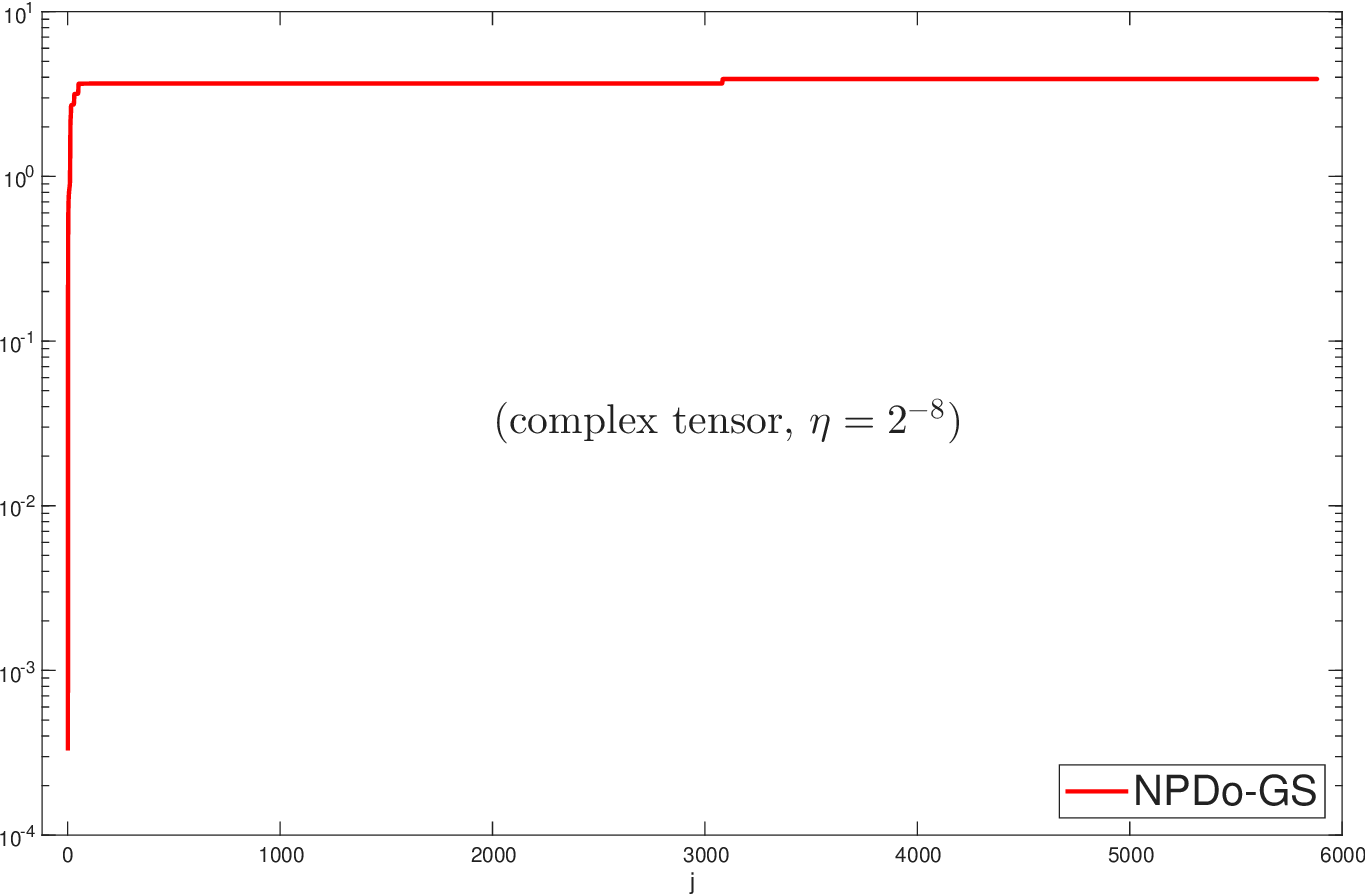}} \\
    \resizebox*{0.31\textwidth}{0.17\textheight}{\includegraphics{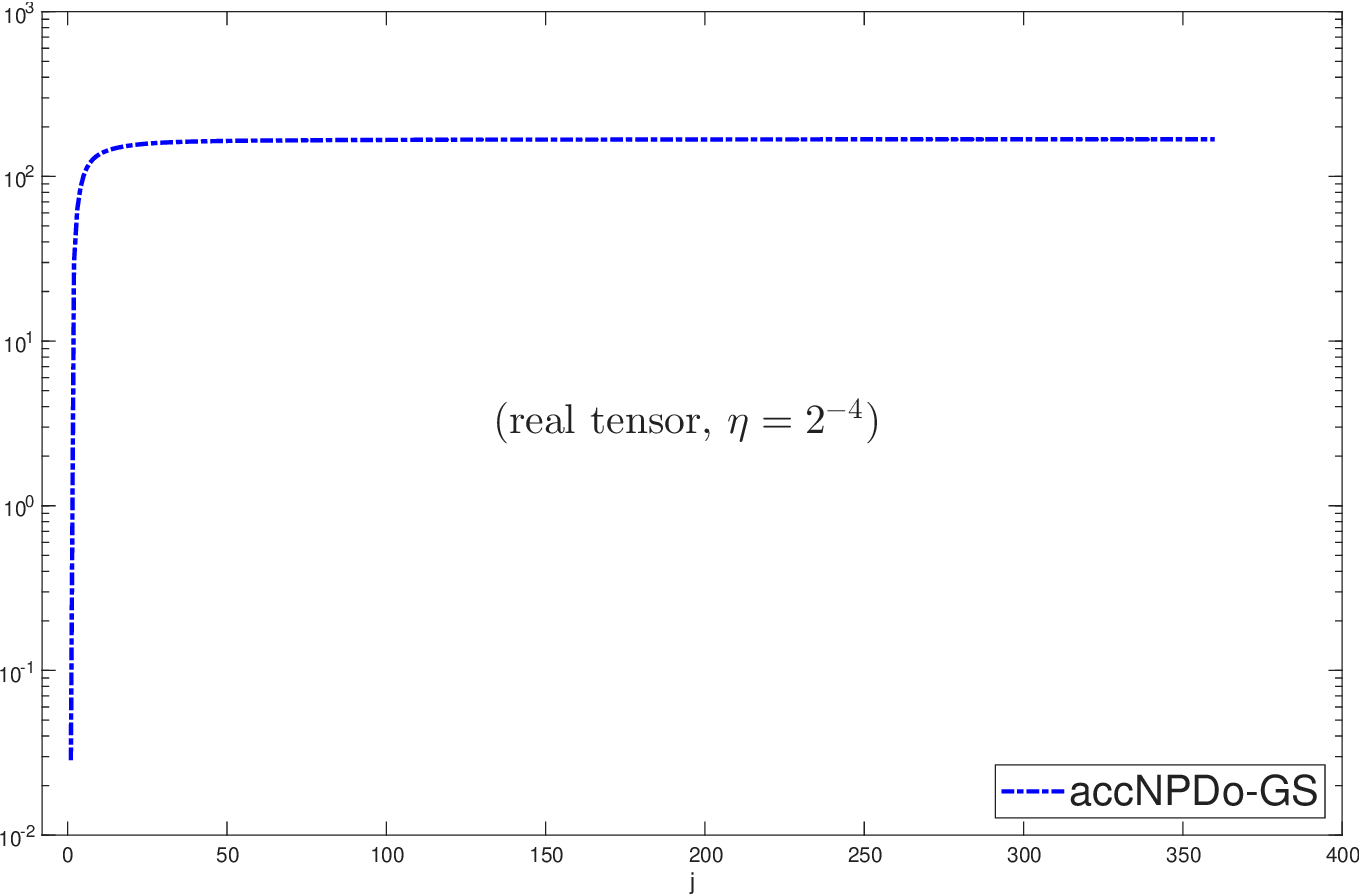}}
  & \resizebox*{0.31\textwidth}{0.17\textheight}{\includegraphics{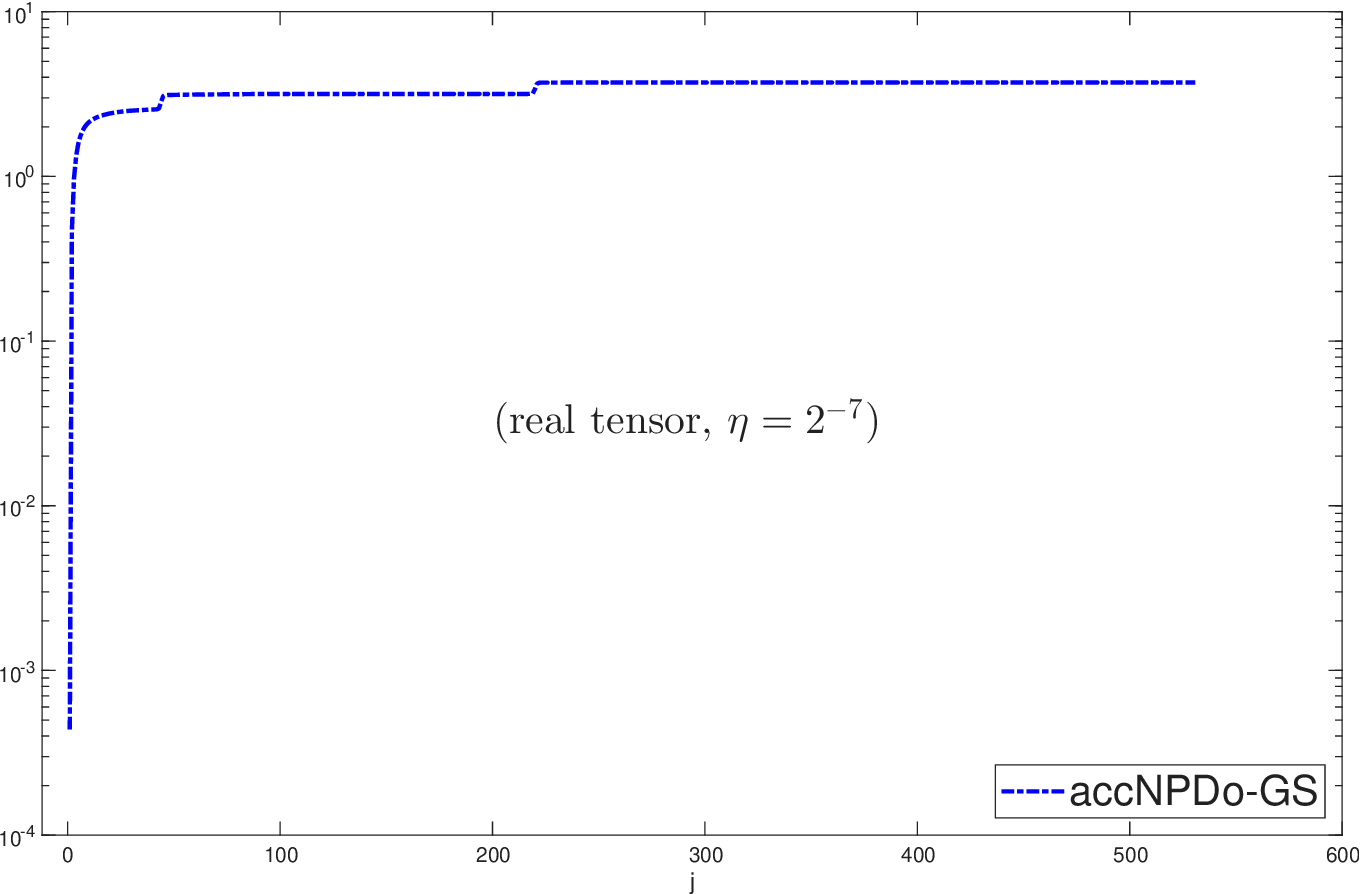}}
  & \resizebox*{0.31\textwidth}{0.17\textheight}{\includegraphics{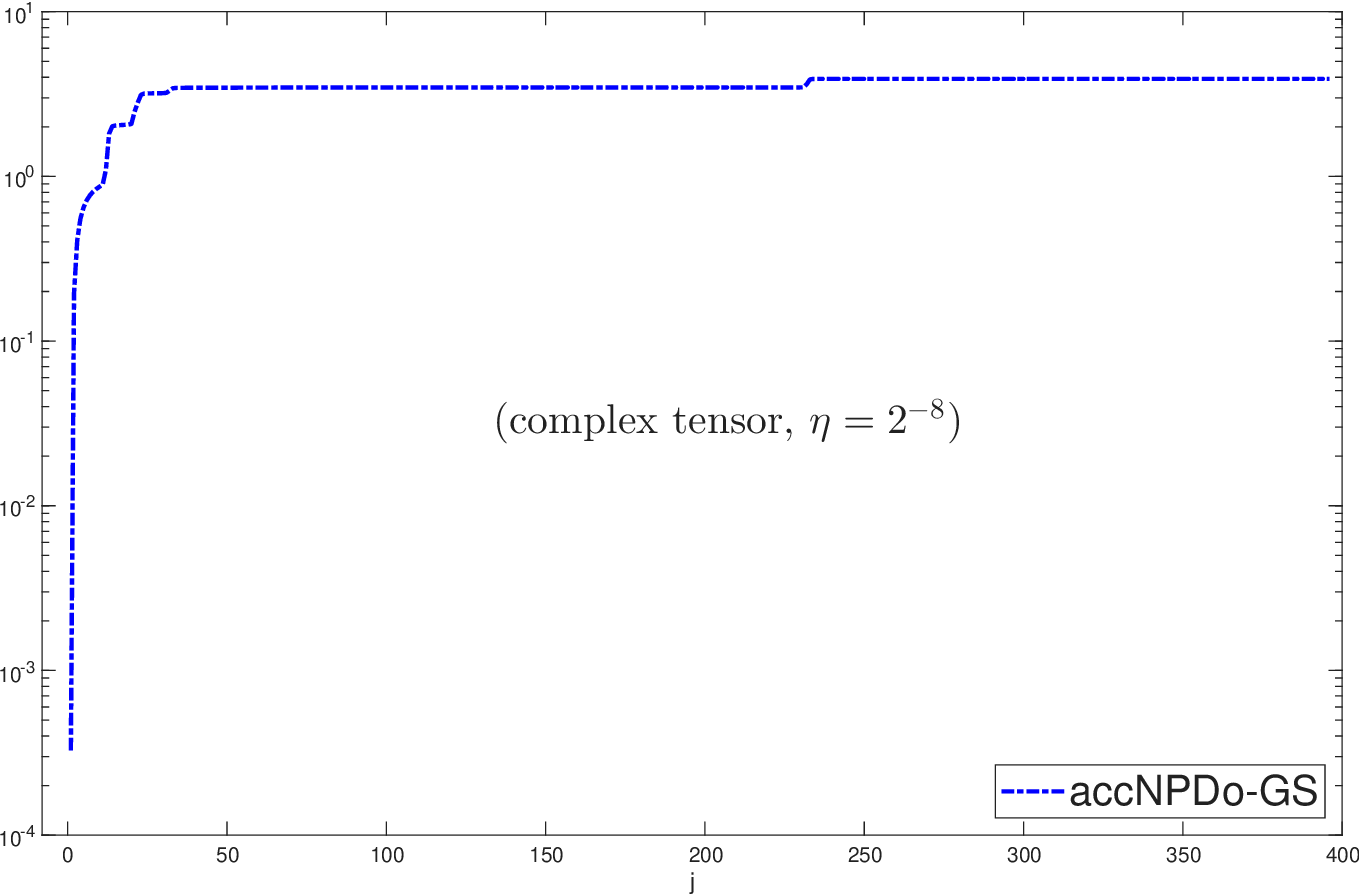}}
\end{tabular}\par
}
\vspace{-0.15 cm}
\caption{\small Principal tensor SVD (\ptsvd) by the NPDo approach -- convergence in terms of
         objective value associated with six of the twelve examples in \Cref{fig:diagTS-behavior-cvg}:
         the first row is for the plain NPDo while the second row is for accNPDo.
  }
\label{fig:diagTS-behavior-cvg-obj}
\end{figure}

\begin{figure}
{\centering
\begin{tabular}{ccc}
    \resizebox*{0.31\textwidth}{0.17\textheight}{\includegraphics{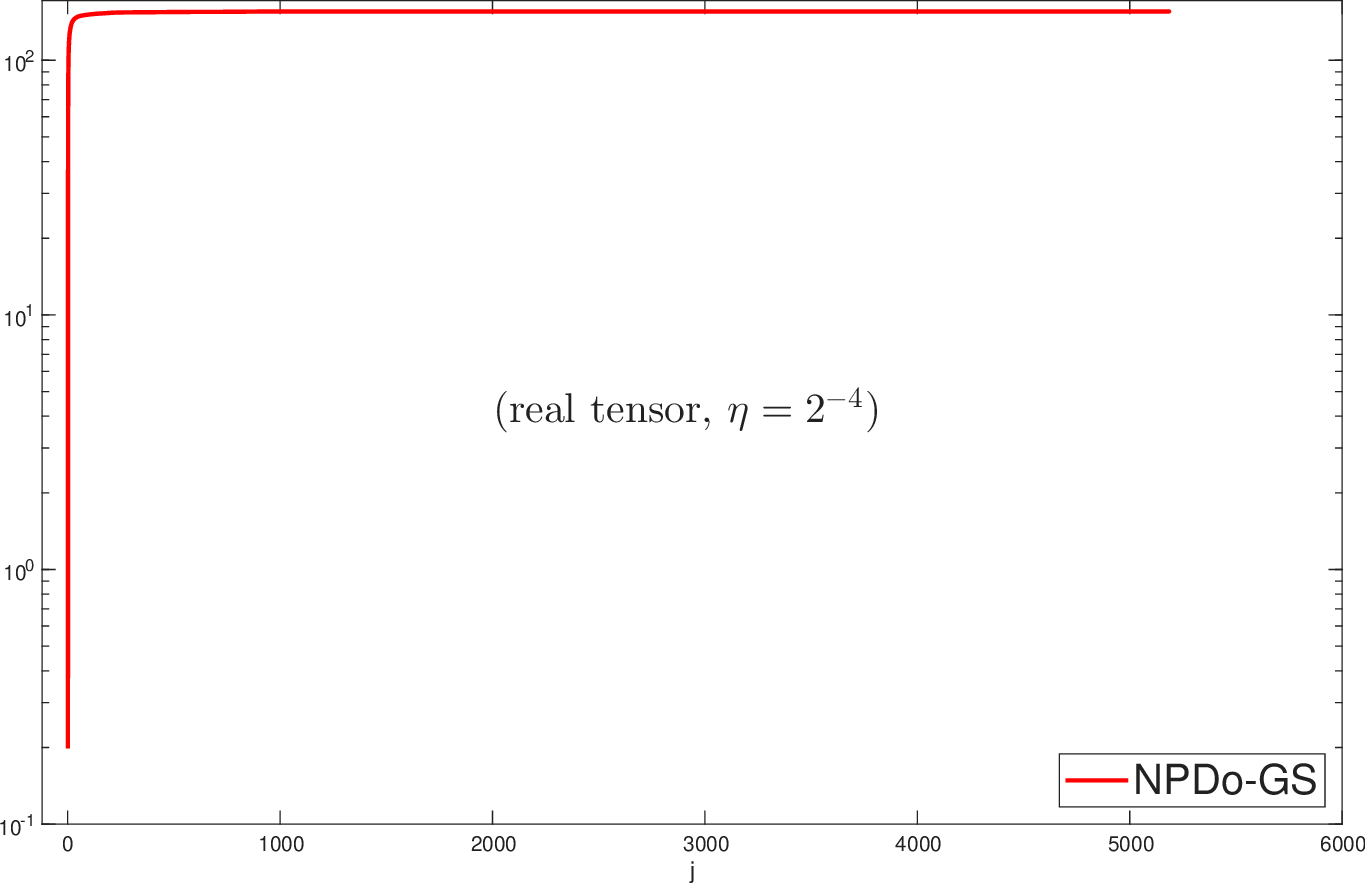}}
  & \resizebox*{0.31\textwidth}{0.17\textheight}{\includegraphics{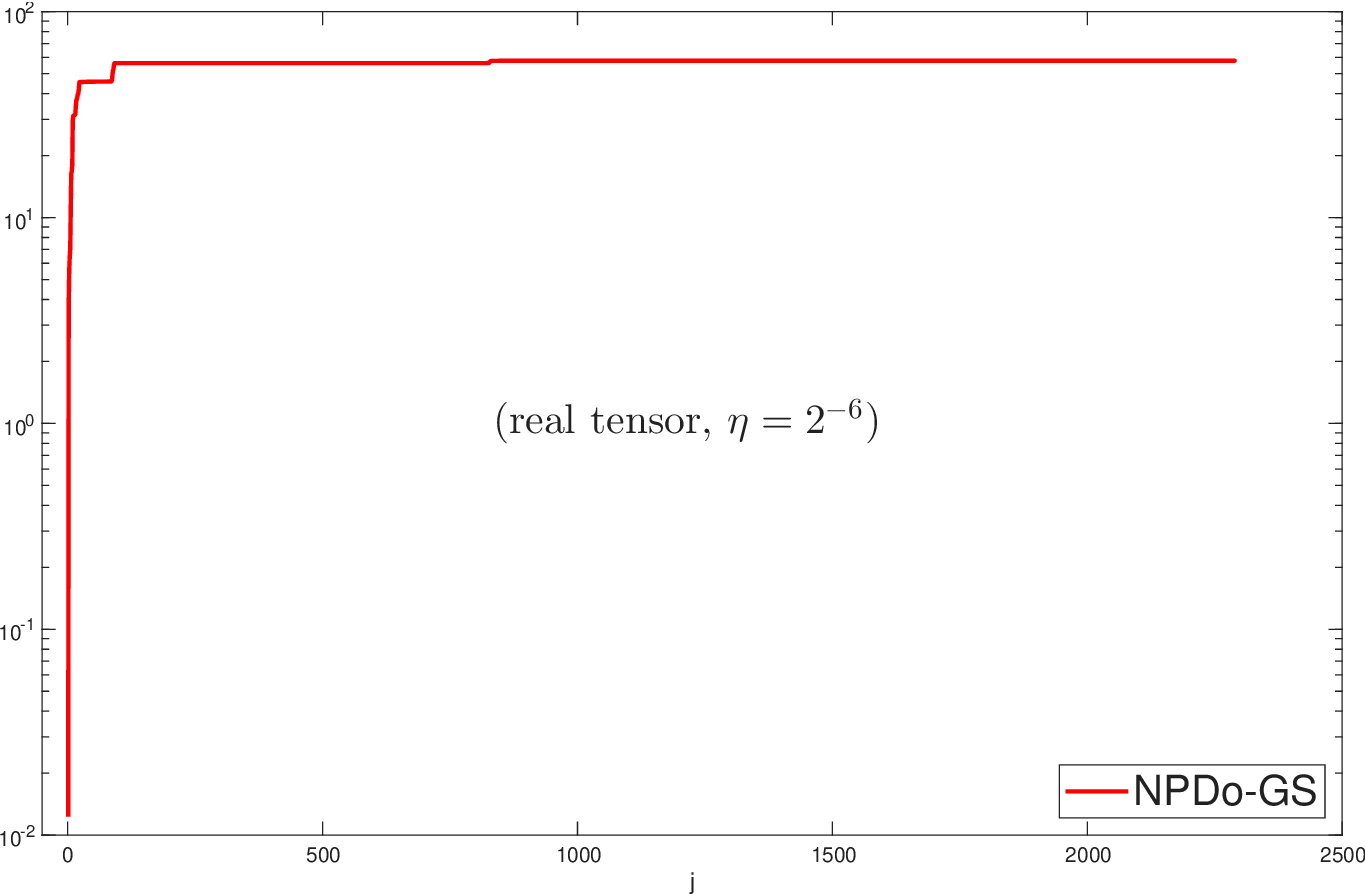}}
  & \resizebox*{0.31\textwidth}{0.17\textheight}{\includegraphics{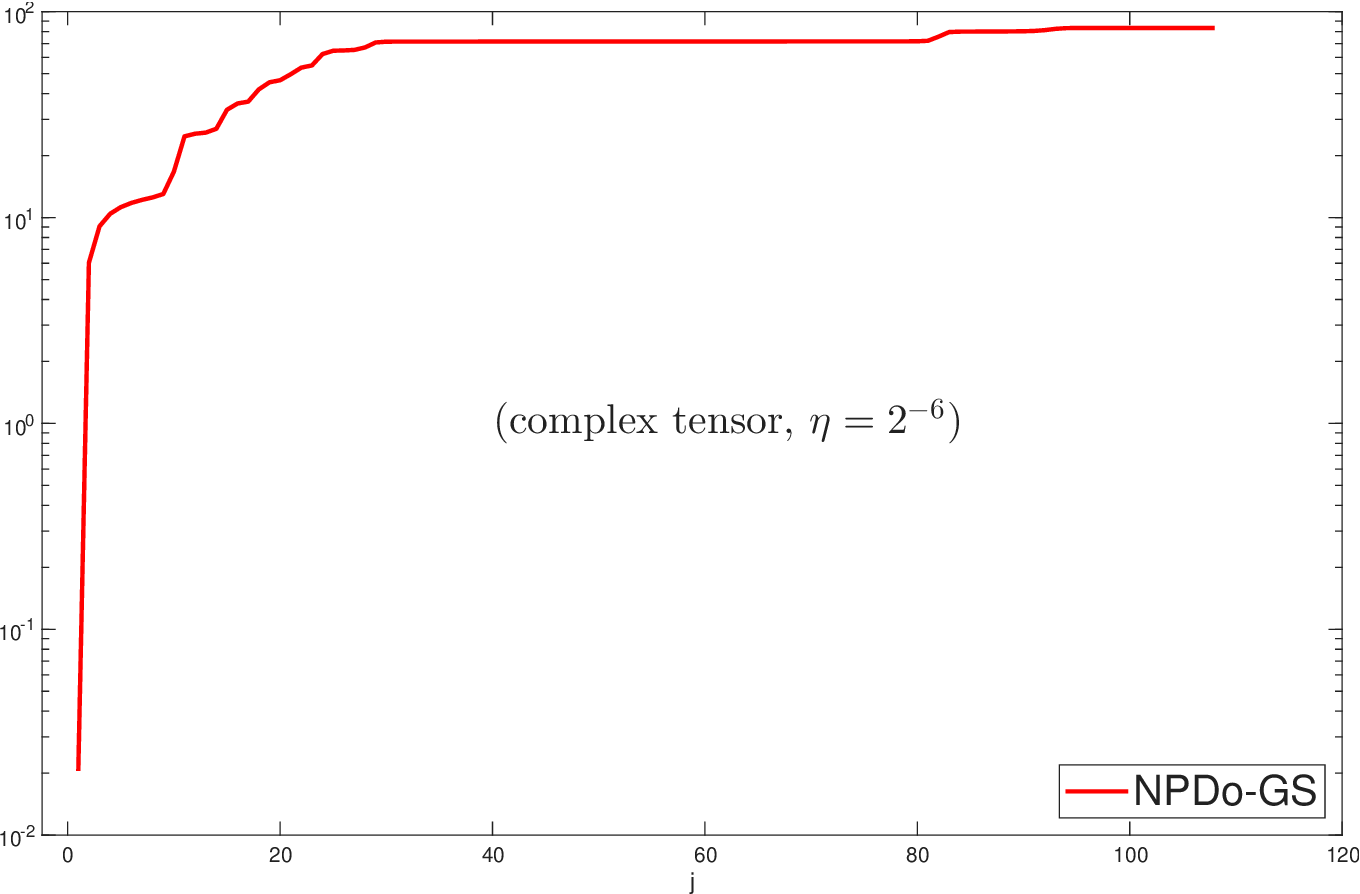}} \\
    \resizebox*{0.31\textwidth}{0.17\textheight}{\includegraphics{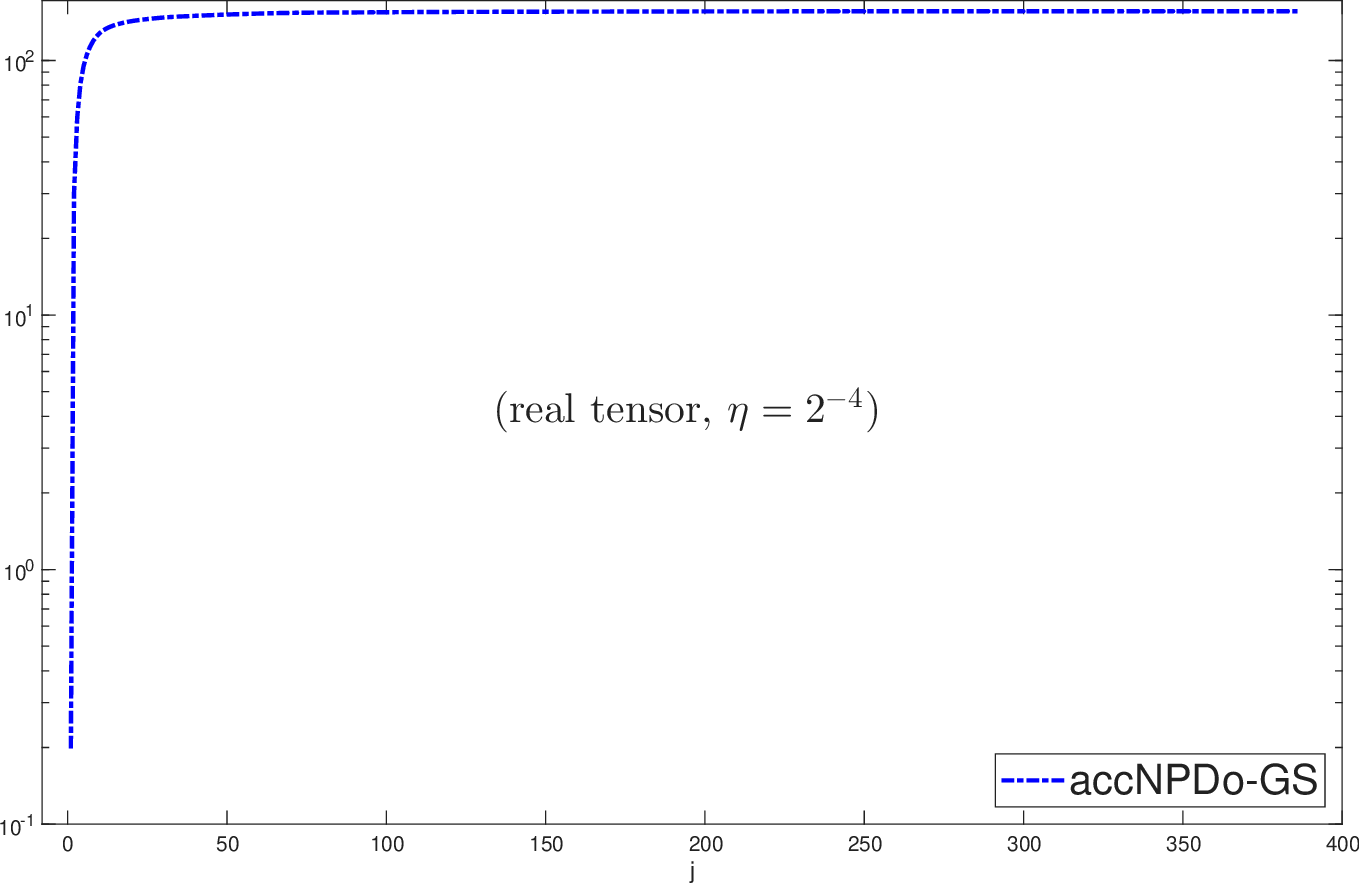}}
  & \resizebox*{0.31\textwidth}{0.17\textheight}{\includegraphics{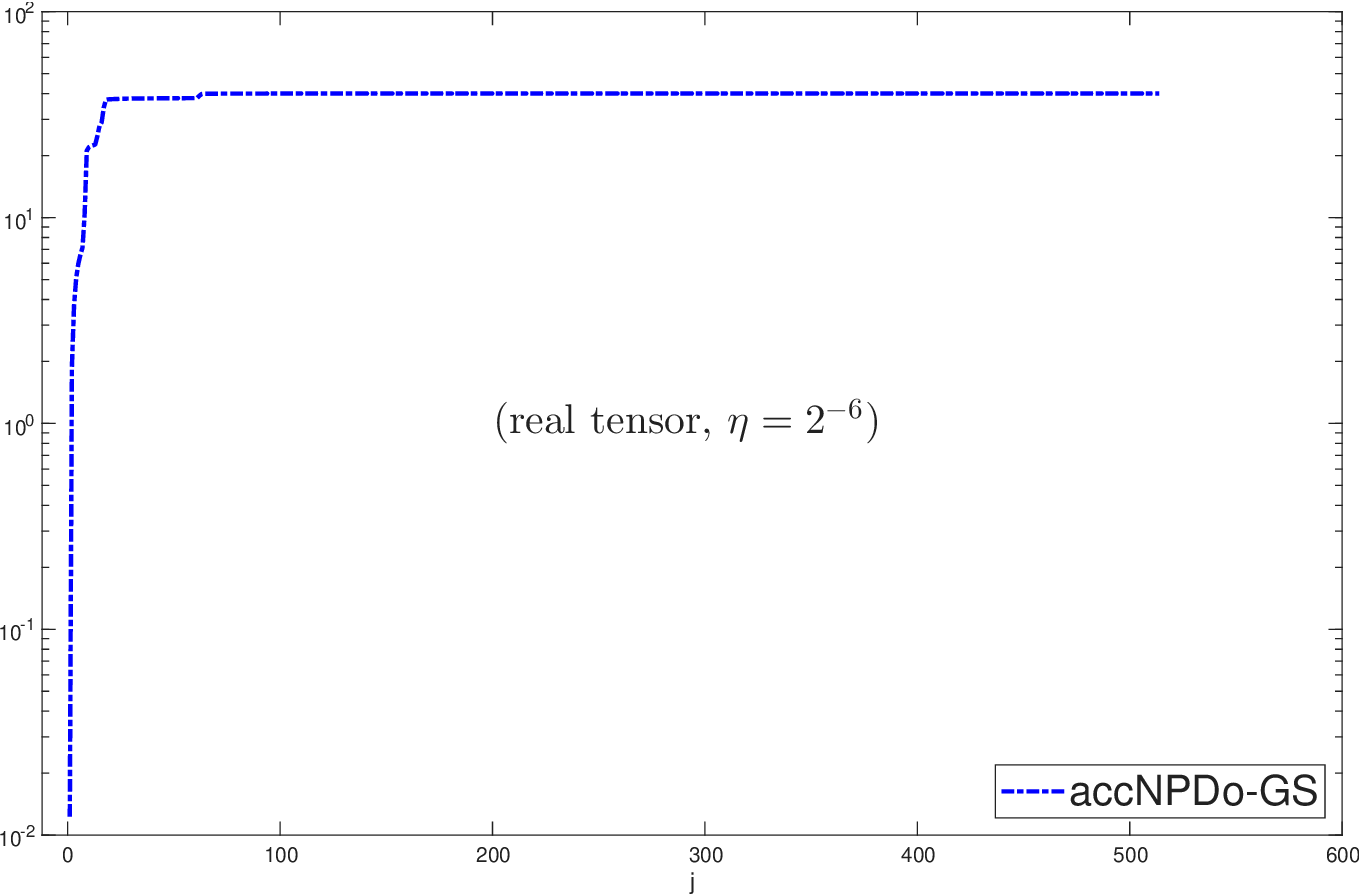}}
  & \resizebox*{0.31\textwidth}{0.17\textheight}{\includegraphics{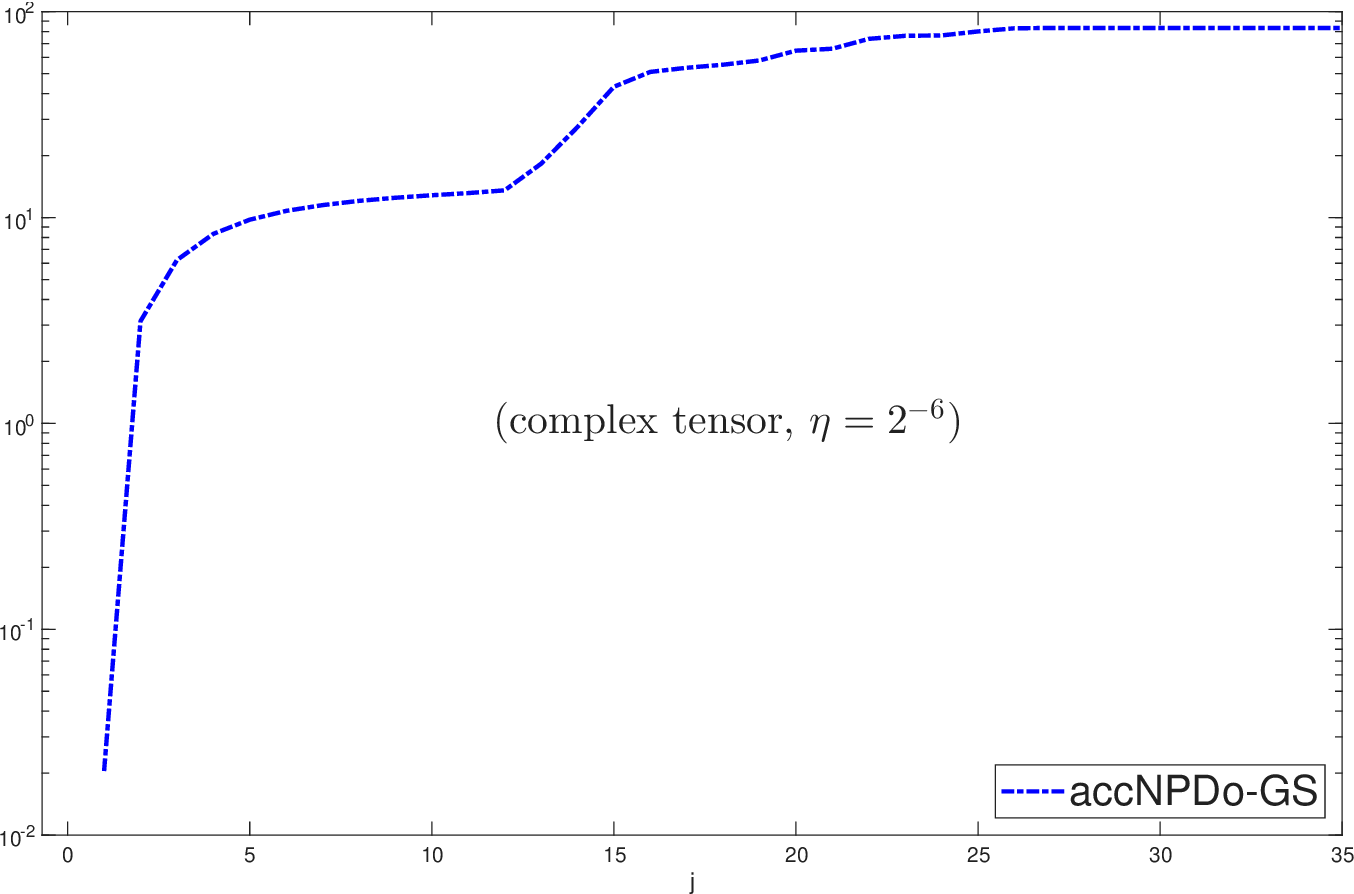}}
\end{tabular}\par
}
\vspace{-0.15 cm}
\caption{\small Principal tensor tensor block-diagonalization (\ptbd) by the NPDo approach -- convergence in terms of
         objective value associated with six of the twelve examples in \Cref{fig:bdiagTS-behavior-cvg}:
         the first row is for the plain NPDo while the second row is for accNPDo.
  }
\label{fig:bdiagTS-behavior-cvg-obj}
\end{figure}

\begin{figure}
{\centering
\begin{tabular}{cc}
    \resizebox*{0.35\textwidth}{0.20\textheight}{\includegraphics{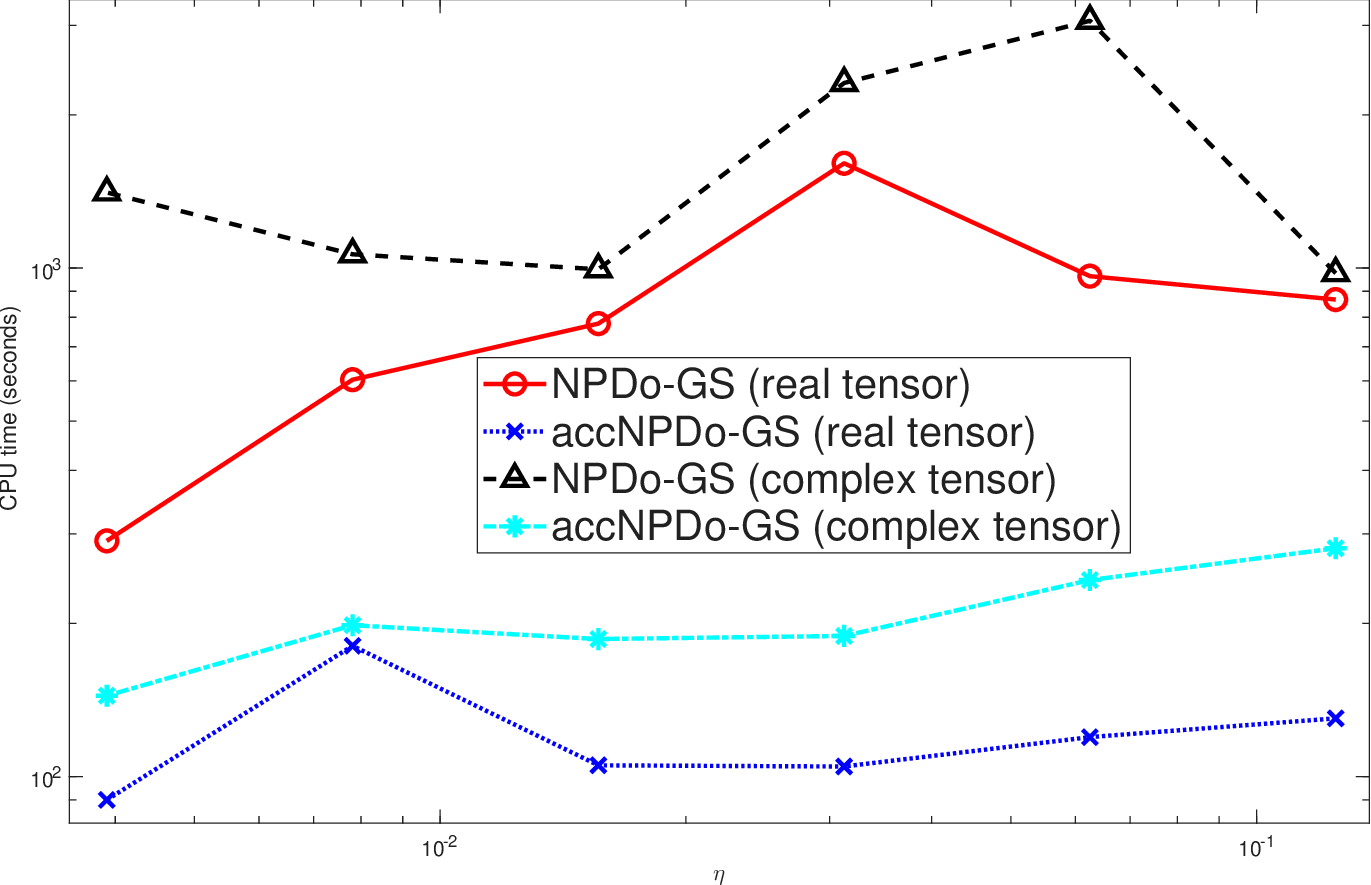}}
  & \resizebox*{0.35\textwidth}{0.20\textheight}{\includegraphics{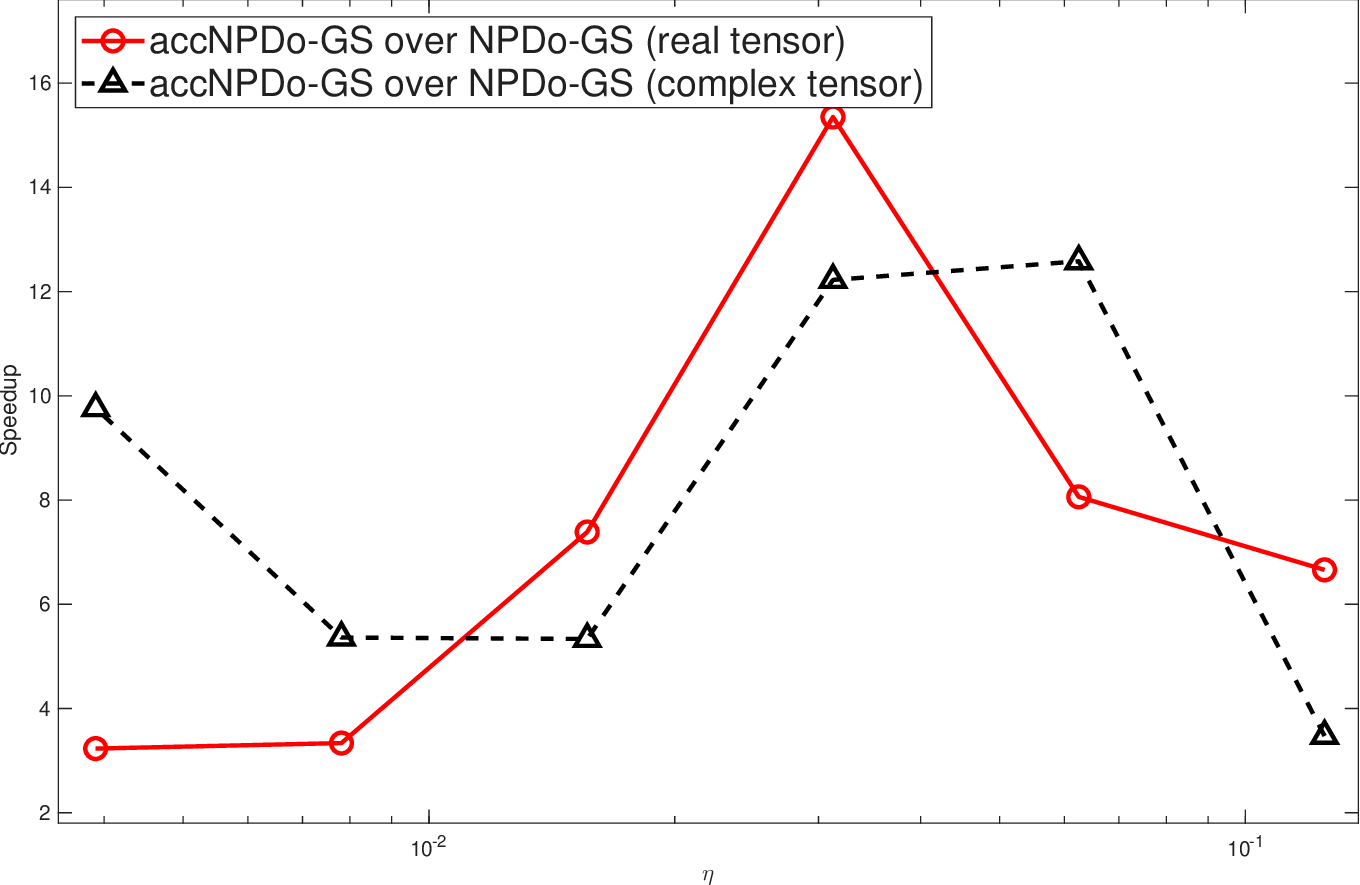}} \\
    \resizebox*{0.35\textwidth}{0.20\textheight}{\includegraphics{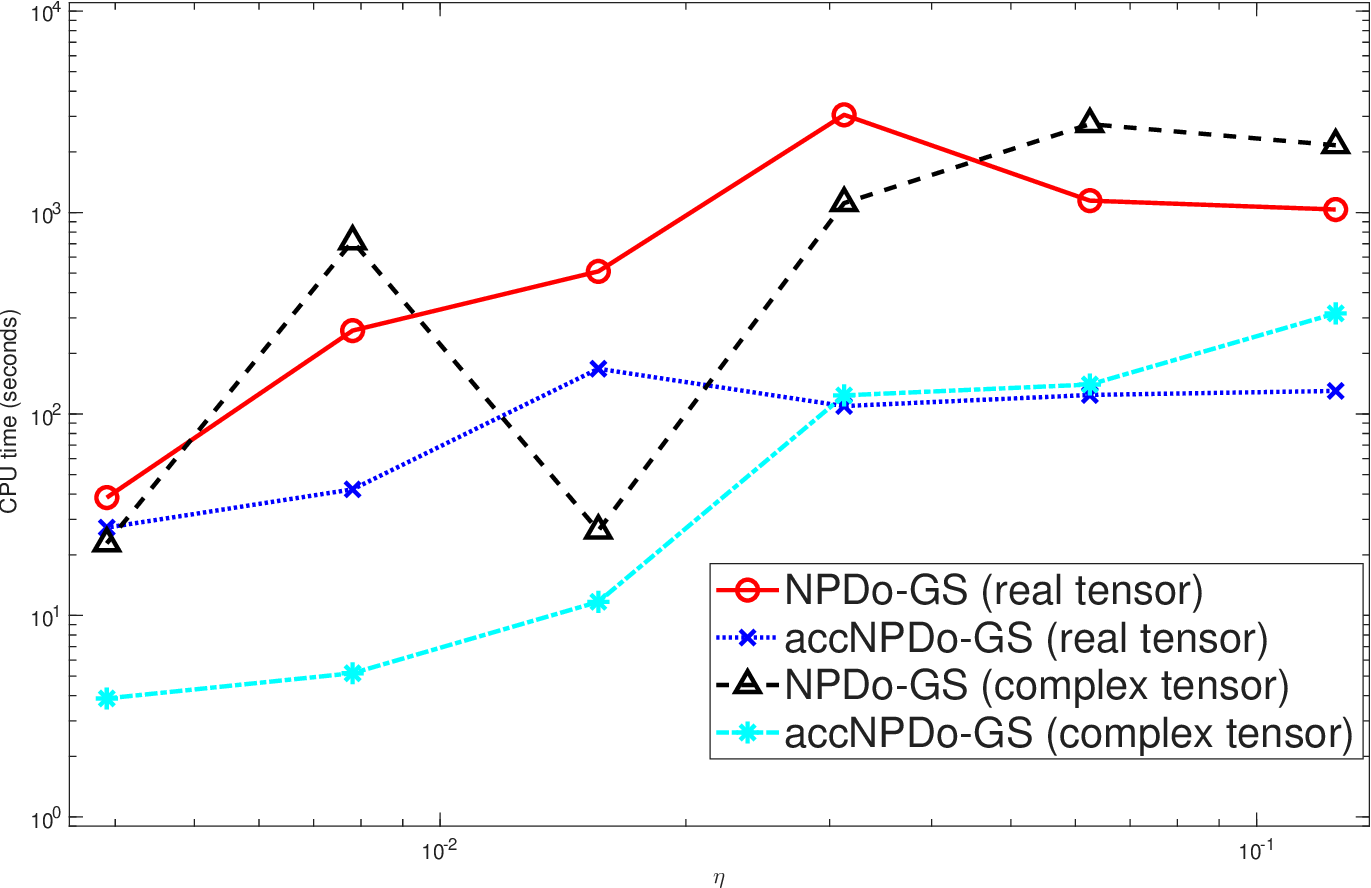}}
  & \resizebox*{0.35\textwidth}{0.20\textheight}{\includegraphics{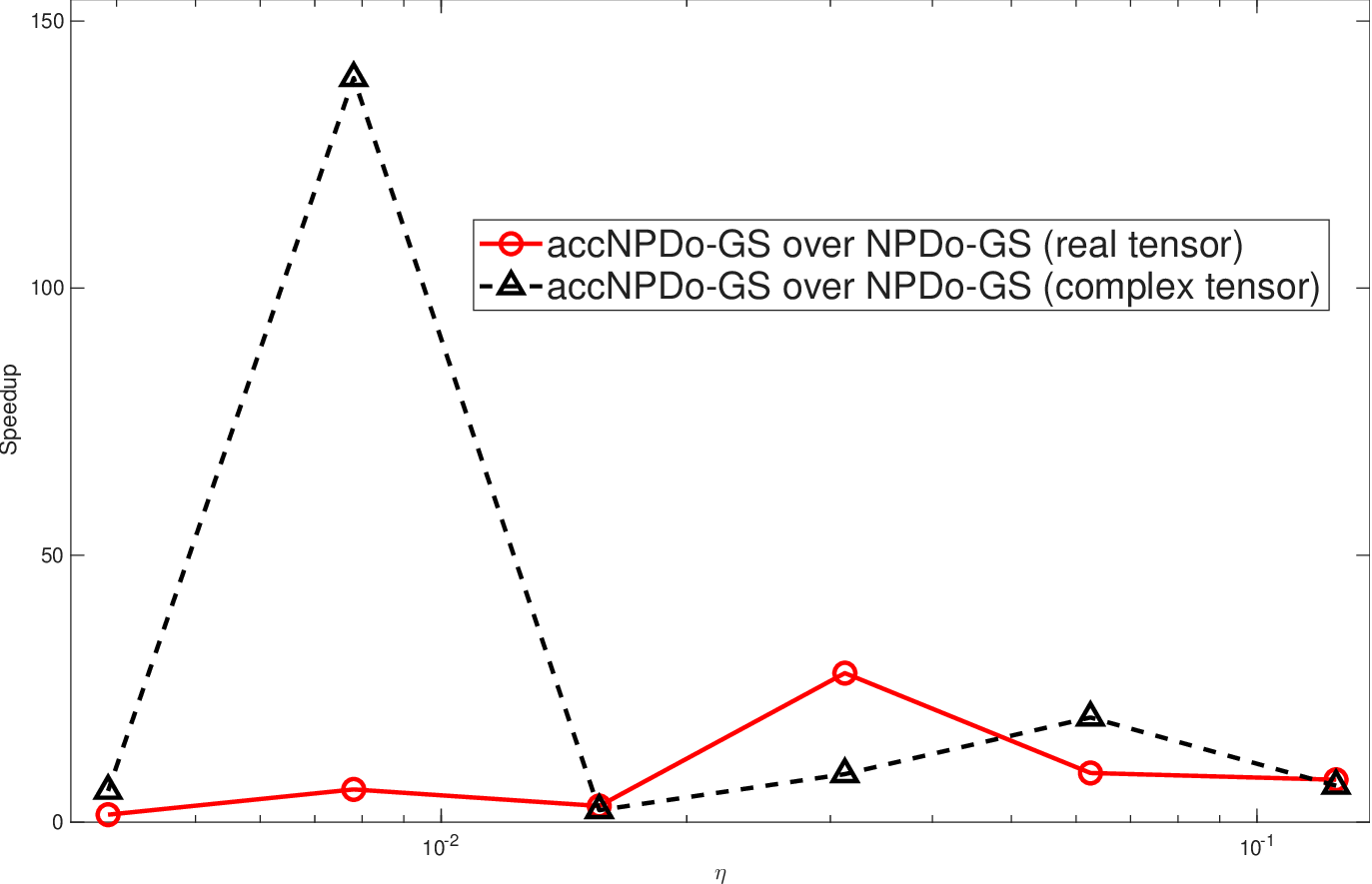}}
\end{tabular}\par
}
\vspace{-0.15 cm}
\caption{\small Principal tensor SVD (\ptsvd) by the NPDo approach -- -- CPU time in seconds (left) and
      and the speedup of accNPDo over NPDo (right): the first row is for
      \ptsvd\ on the examples in \Cref{fig:diagTS-behavior-cvg} while the second row is for
      \ptbd\ on the examples in \Cref{fig:bdiagTS-behavior-cvg}.
  }
\label{fig:(b)diagTS-behavior-speed}
\end{figure}


\Cref{fig:diagTS-behavior-cvg,fig:bdiagTS-behavior-cvg}
plot convergence behavior in terms of KKT residual $\tilde\epsilon_{\KKT,j}$ as defined in \eqref{eq:stop-BTSVD'}.
%
We observe the following:
\begin{enumerate}[(1)]
  \item At convergence, all KKT residuals $\tilde\epsilon_{\KKT,j}$ are $10^{-9}$ or smaller,
        indicating approximate stationary points are found accurately as dictated by stopping criterion
        $\tilde\epsilon_{\KKT,j}\le 10^{-9}$.
  \item As a general trend, the number of iterations required for convergence decreases as $\eta$ gets smaller.
        By construction, the smaller $\eta$ is, the closer $B$ is to be principally (block)-diagonalizable.
        That is probably the reason behind the trend.
  \item With the LOCG acceleration technique, \Cref{alg:accNPDo:PBTD}
        cuts down the number of (outer) iterations dramatically in comparison to \Cref{alg:NPDo:PBTD}.
        We note that \Cref{alg:accNPDo:PBTD} employs an inner-outer iterative scheme in which the inner
        iteration is concealed in \Cref{alg:NPDo:PBTD} with an adaptive stopping criterion
        as explained in \Cref{rk:SCF4npd+LOCG}(v).
        On the same issue, \Cref{fig:(b)diagTS-behavior-speed} plots the CUP time against $\eta$ and the speedup
        of accNPDo over NPDo.
  \item The overall trend of the KKT residual $\tilde\epsilon_{\KKT,j}$ is moving towards $0$, although it is not monotonic. In fact, we can see some of $\tilde\epsilon_{\KKT,j}$ during the iteration spikes to approximately
      the level at the beginning of whole iterative process. Upon reviewing the trajectory of objective
      value, we find that the spike causes a jump in objective value, too, implying that
      the underlying iteration likely hits some ``saddle point'' and moves on to the next stationary point
      with a larger objective value.
      \Cref{fig:diagTS-behavior-cvg-obj,fig:bdiagTS-behavior-cvg-obj} sample three examples each
      from the experiments in \Cref{fig:diagTS-behavior-cvg,fig:bdiagTS-behavior-cvg}, respectively,
      plotting the objective value trajectory, where we can see sudden jumps in the plots for $\eta\le 2^{-6}$,
      while all objective value trajectories monotonically move up as theoretically guaranteed.
%
\end{enumerate}


\subsection{Scalability}

\begin{figure}[t]
{\centering
\begin{tabular}{ccc}
  \resizebox*{0.31\textwidth}{0.17\textheight}{\includegraphics{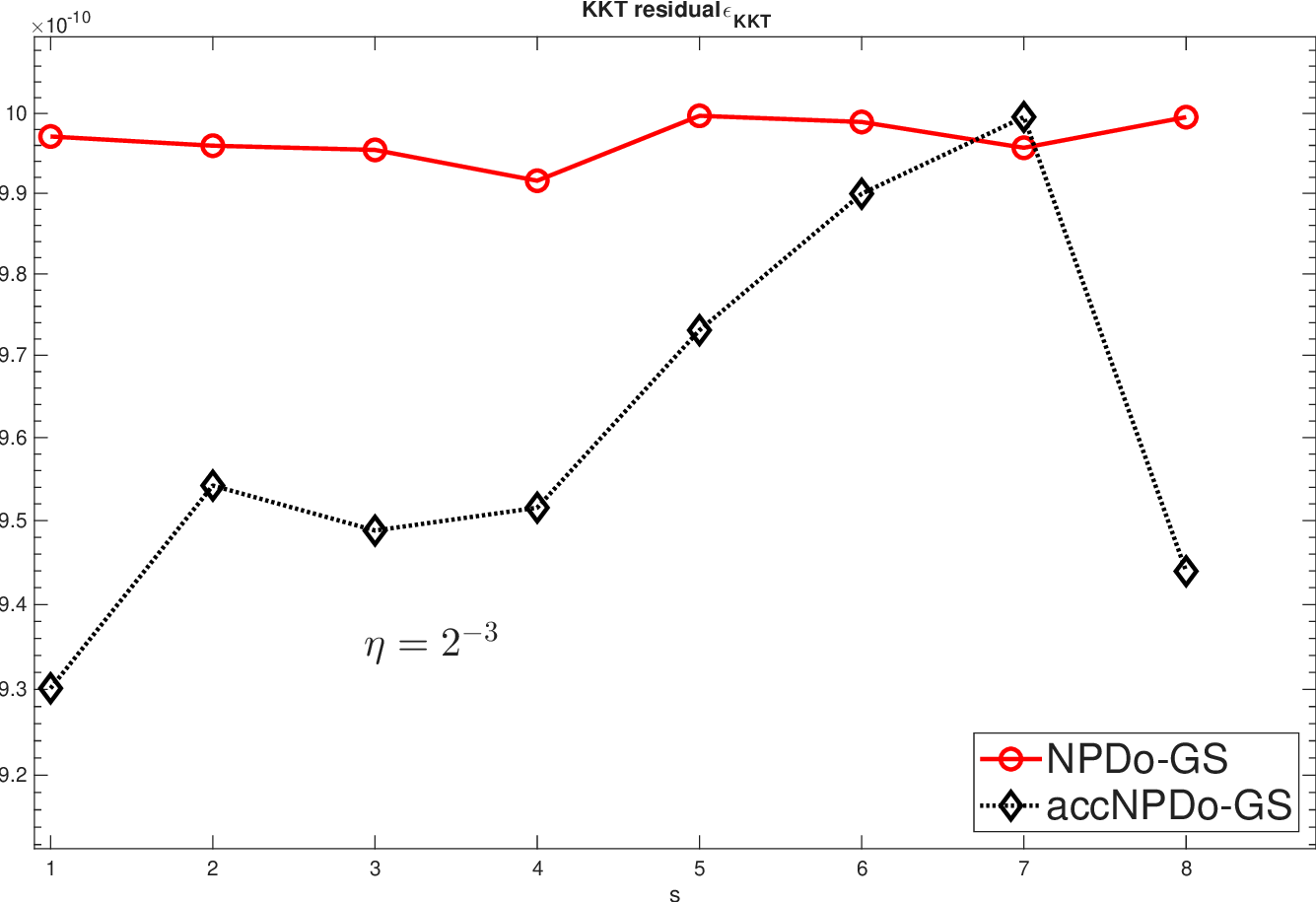}}
  &\resizebox*{0.31\textwidth}{0.17\textheight}{\includegraphics{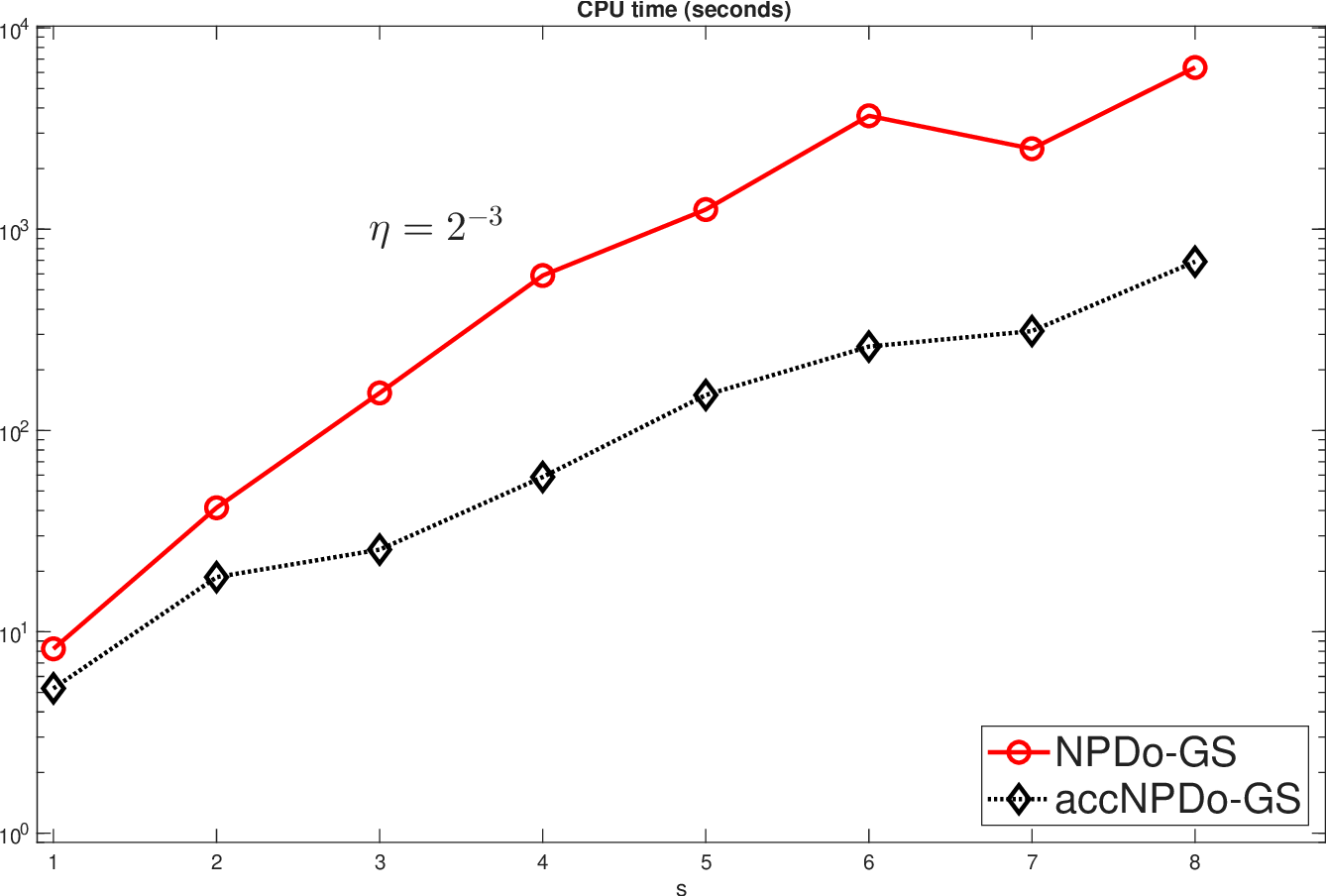}}
  &\resizebox*{0.31\textwidth}{0.17\textheight}{\includegraphics{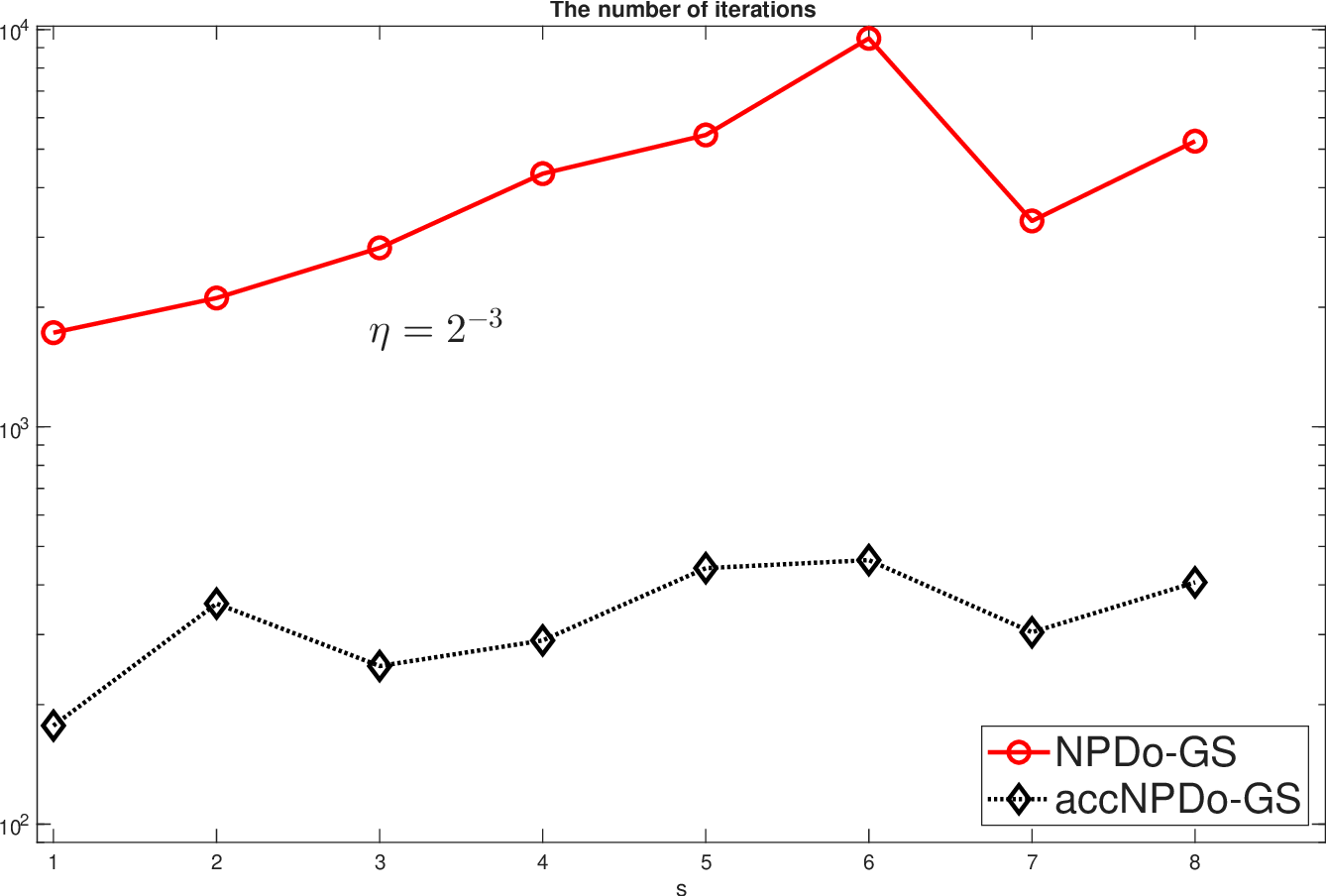}} \\
%
  \resizebox*{0.31\textwidth}{0.17\textheight}{\includegraphics{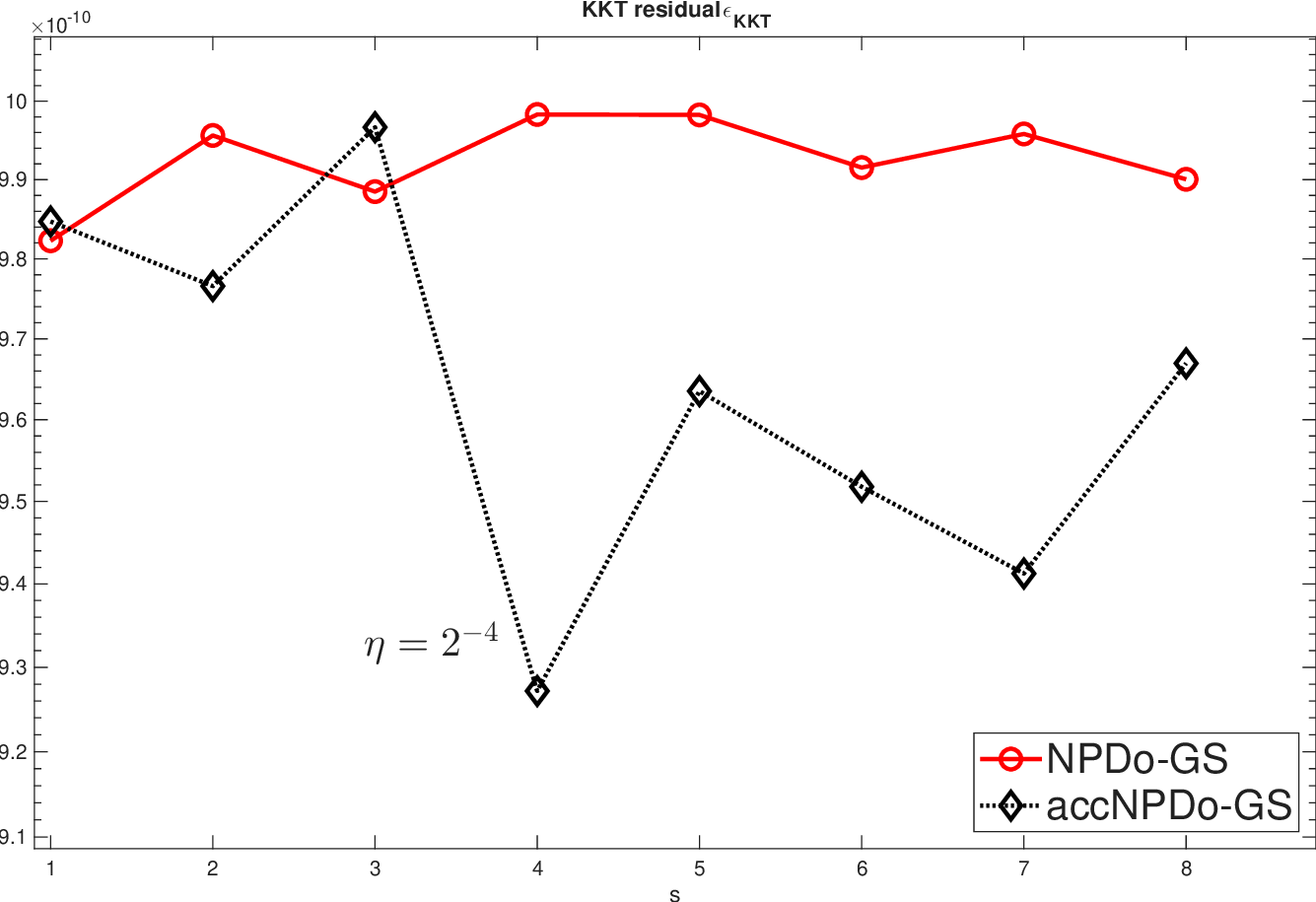}}
  &\resizebox*{0.31\textwidth}{0.17\textheight}{\includegraphics{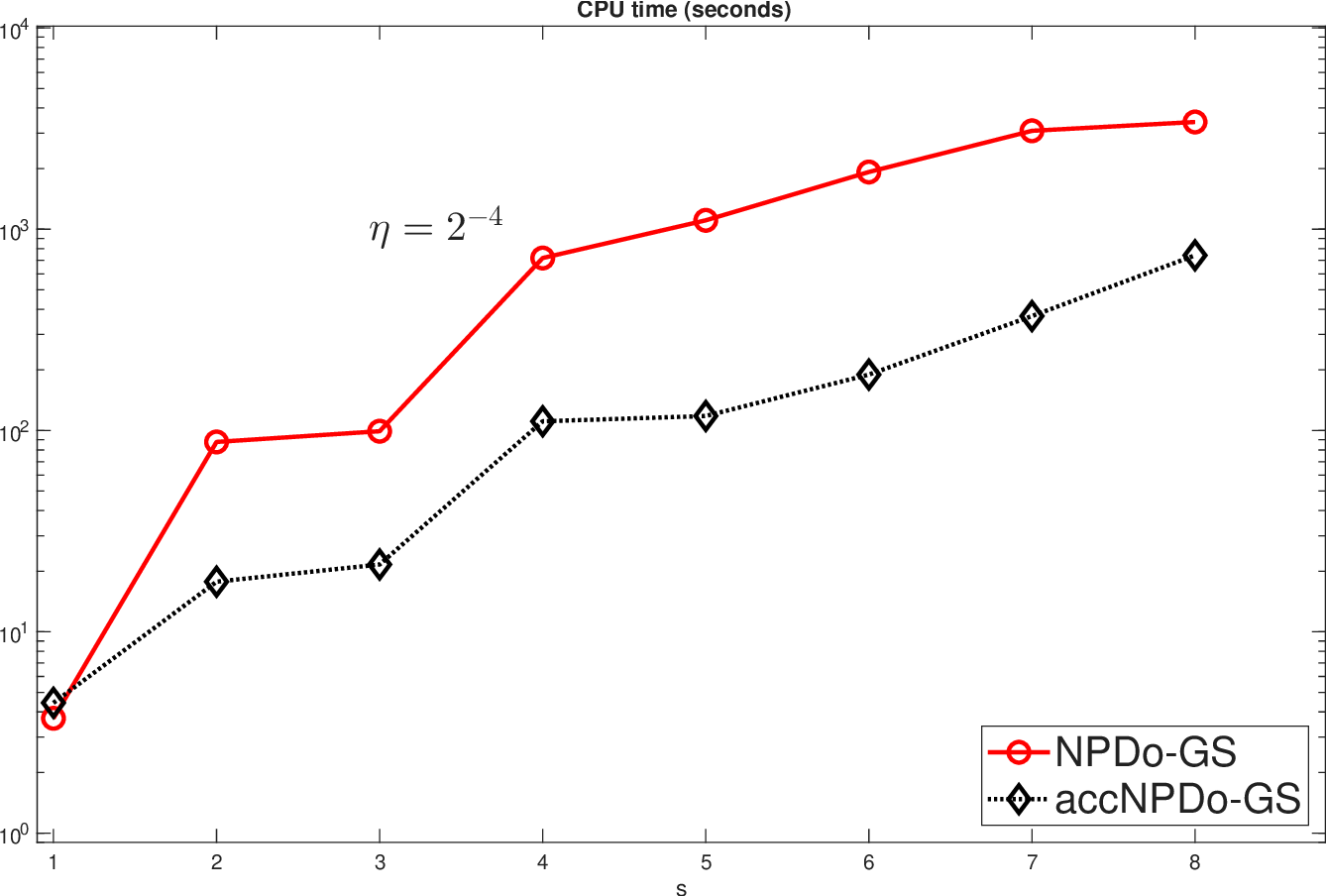}}
  &\resizebox*{0.31\textwidth}{0.17\textheight}{\includegraphics{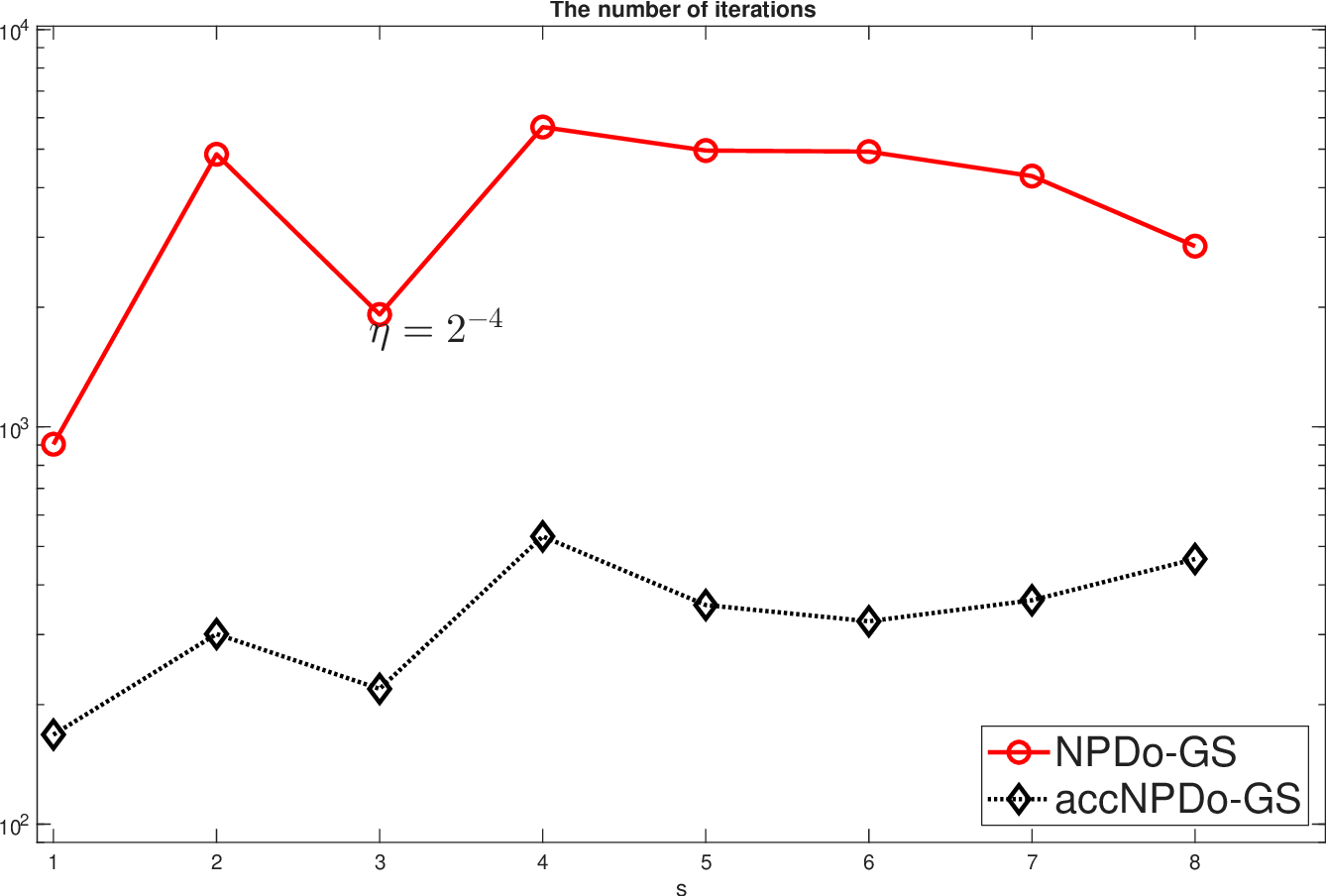}} \\
%
  \resizebox*{0.31\textwidth}{0.17\textheight}{\includegraphics{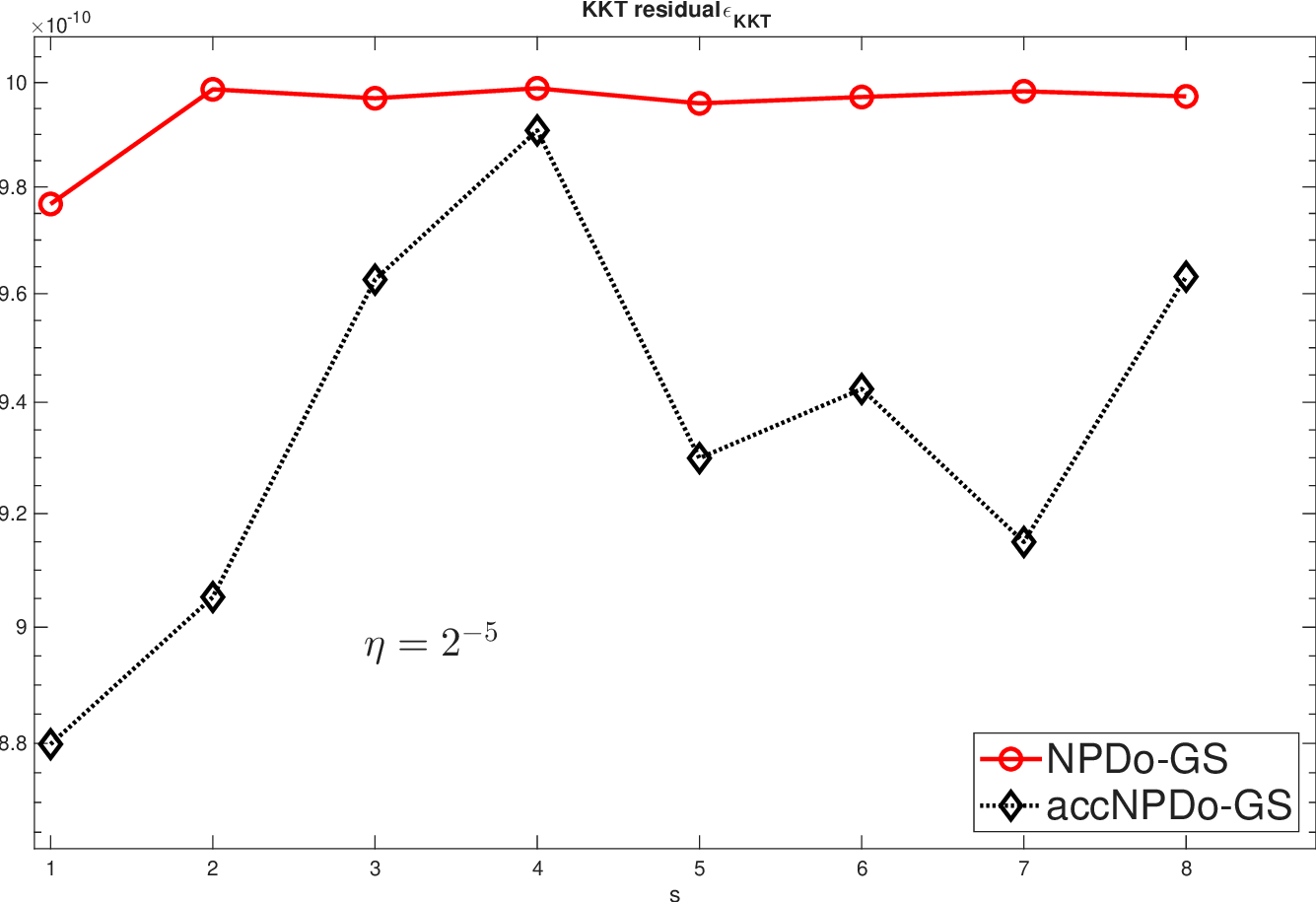}}
  &\resizebox*{0.31\textwidth}{0.17\textheight}{\includegraphics{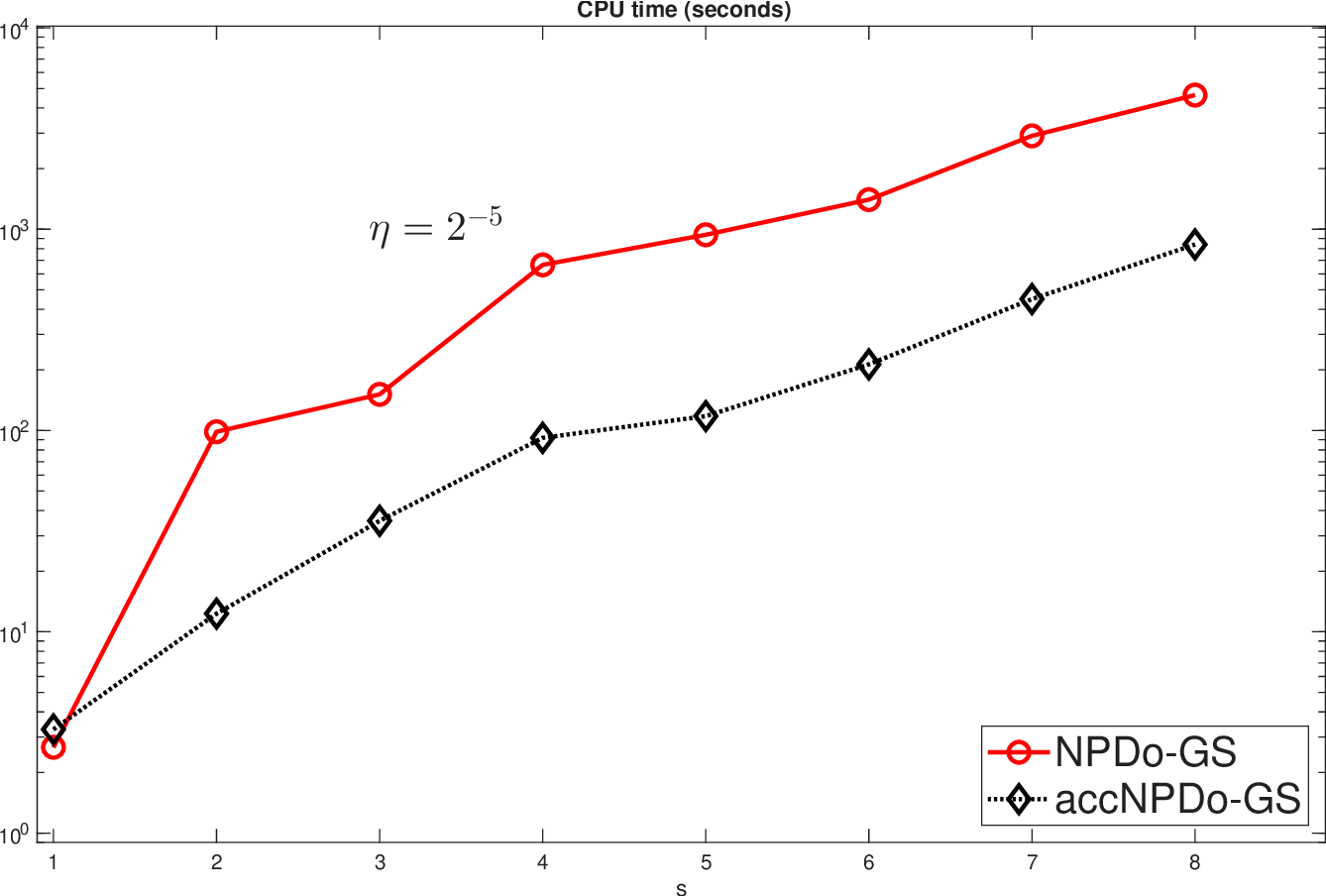}}
  &\resizebox*{0.31\textwidth}{0.17\textheight}{\includegraphics{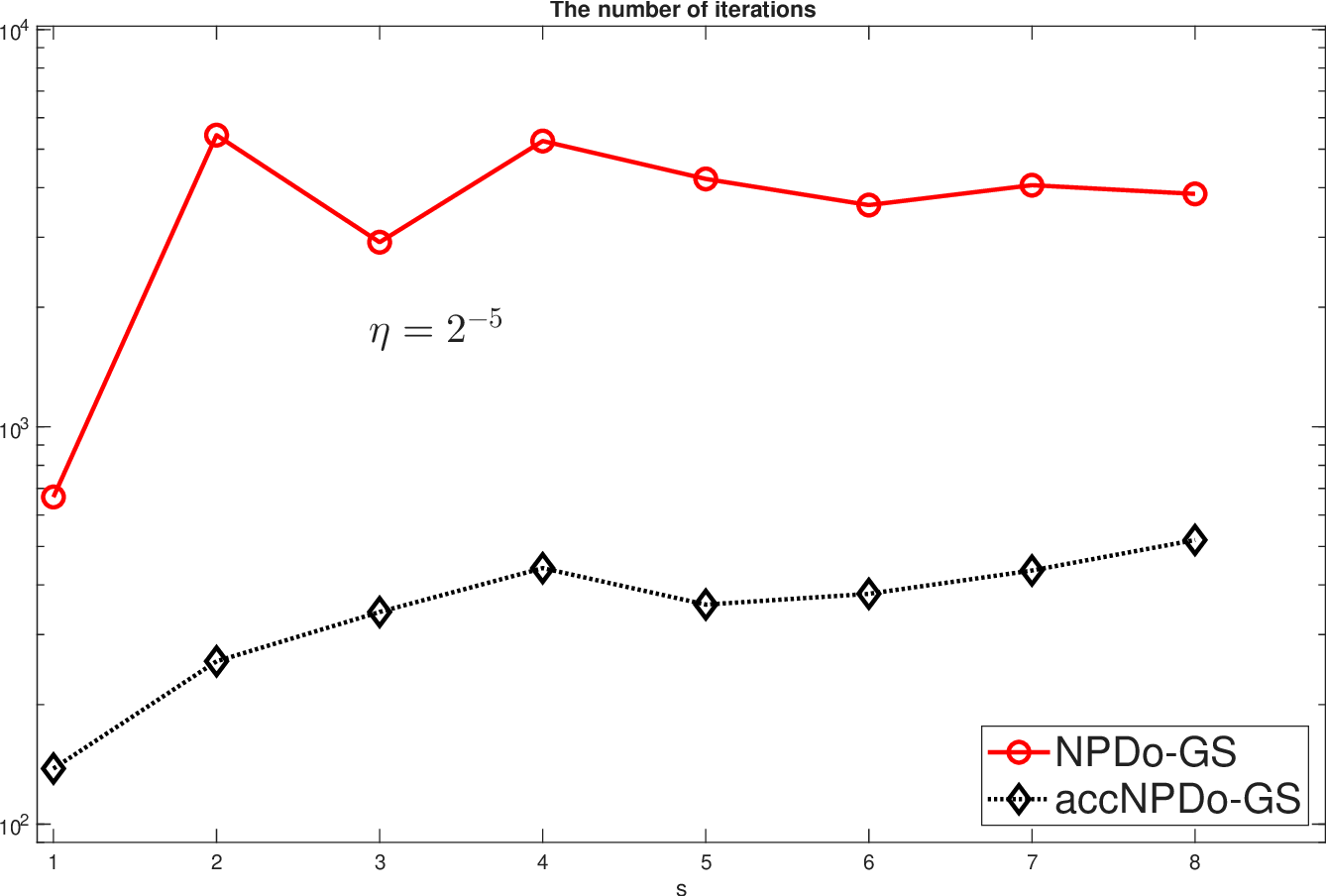}}
\end{tabular}\par
}
\vspace{-0.15 cm}
\caption{\small \ptsvd: scalability of NPDo and accNPDo on real tensors with $[n_1,n_2,n_3]$ varies as in \eqref{eq:tensor-sizes}
    for $1\le s\le 8$ and $k=10$. {\em Left panel:\/} KKT residual $\tilde\epsilon_{\KKT,j}$,
    {\em Middle panel:\/} CPU time, and {\em Right panel:\/} the number of iterations, while
    $\eta=10^{-3}, 10^{-4}, 10^{-5}$ for the first, second, and third row, respectively.
  }
\label{fig:scale-real-PTSVD}
\end{figure}

For this experiment, we run tests on randomly generated tensors without saving them
and thus we can go for tensors of larger sizes than the ones in the previous subsection.
We  let
\begin{equation}\label{eq:tensor-sizes}
[n_1,n_2,n_3]=s\cdot [100, 110, 120]\quad\mbox{for $s=1,2,\ldots, 8$}.
\end{equation}
The last few $s$ yield larger tensors.
We plot the final KKT residual $\tilde\epsilon_{\KKT,j}$, CPU time, and the number of iterations against $s$ varying from 1 to 8, for real and complex tensors, respectively, for \ptsvd\ in \Cref{fig:scale-real-PTSVD,fig:scale-cmpx-PTSVD}  and
for \ptbd\ in \Cref{fig:scale-real-PTBD,fig:scale-cmpx-PTBD}. We summarize our observations
as follows:
\begin{enumerate}[(1)]
  \item Both methods converge with final KKT residual $\tilde\epsilon_{\KKT,j}$ about $10^{-9}$ or smaller,
        and as $\eta$ gets smaller, the random \ptsvd\ problems become easier for both methods in terms of the numbers of iterations required and consumed CPU time.
  \item CPU time seems to grow linearly with respect to $s$.
  \item accNPDo is faster than NPDo.
\end{enumerate}

\begin{figure}[t]
{\centering
\begin{tabular}{ccc}
  \resizebox*{0.31\textwidth}{0.17\textheight}{\includegraphics{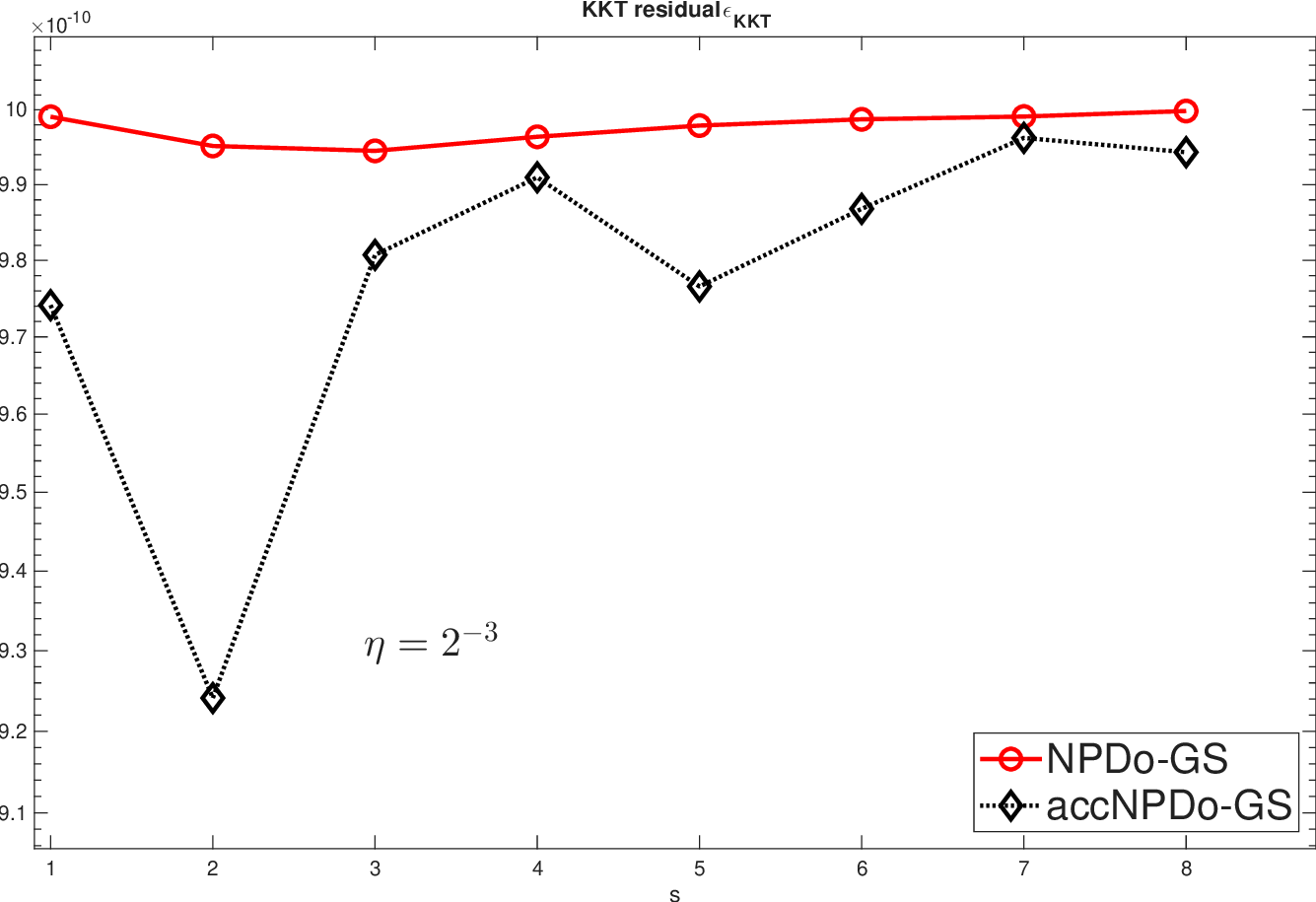}}
  &\resizebox*{0.31\textwidth}{0.17\textheight}{\includegraphics{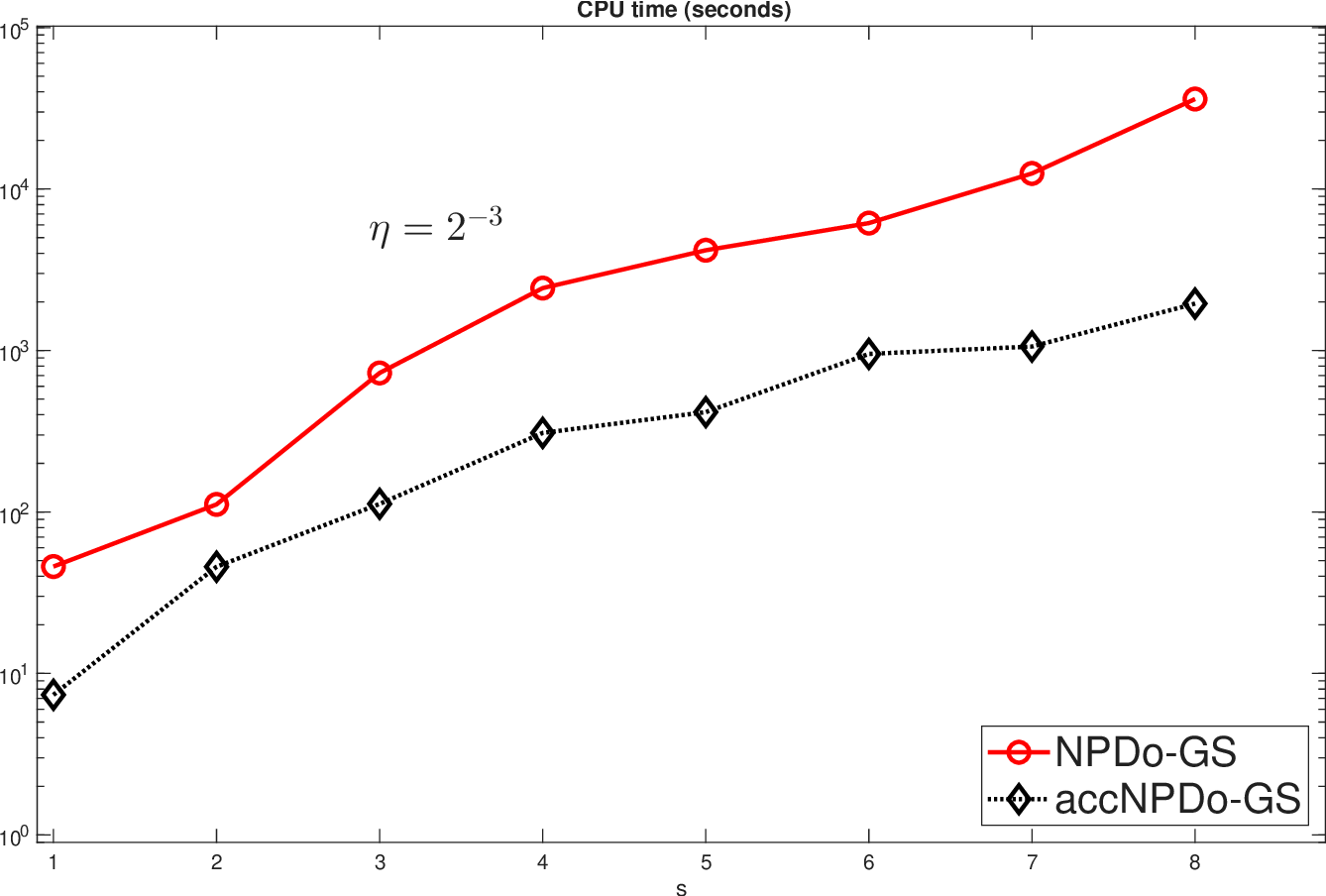}}
  &\resizebox*{0.31\textwidth}{0.17\textheight}{\includegraphics{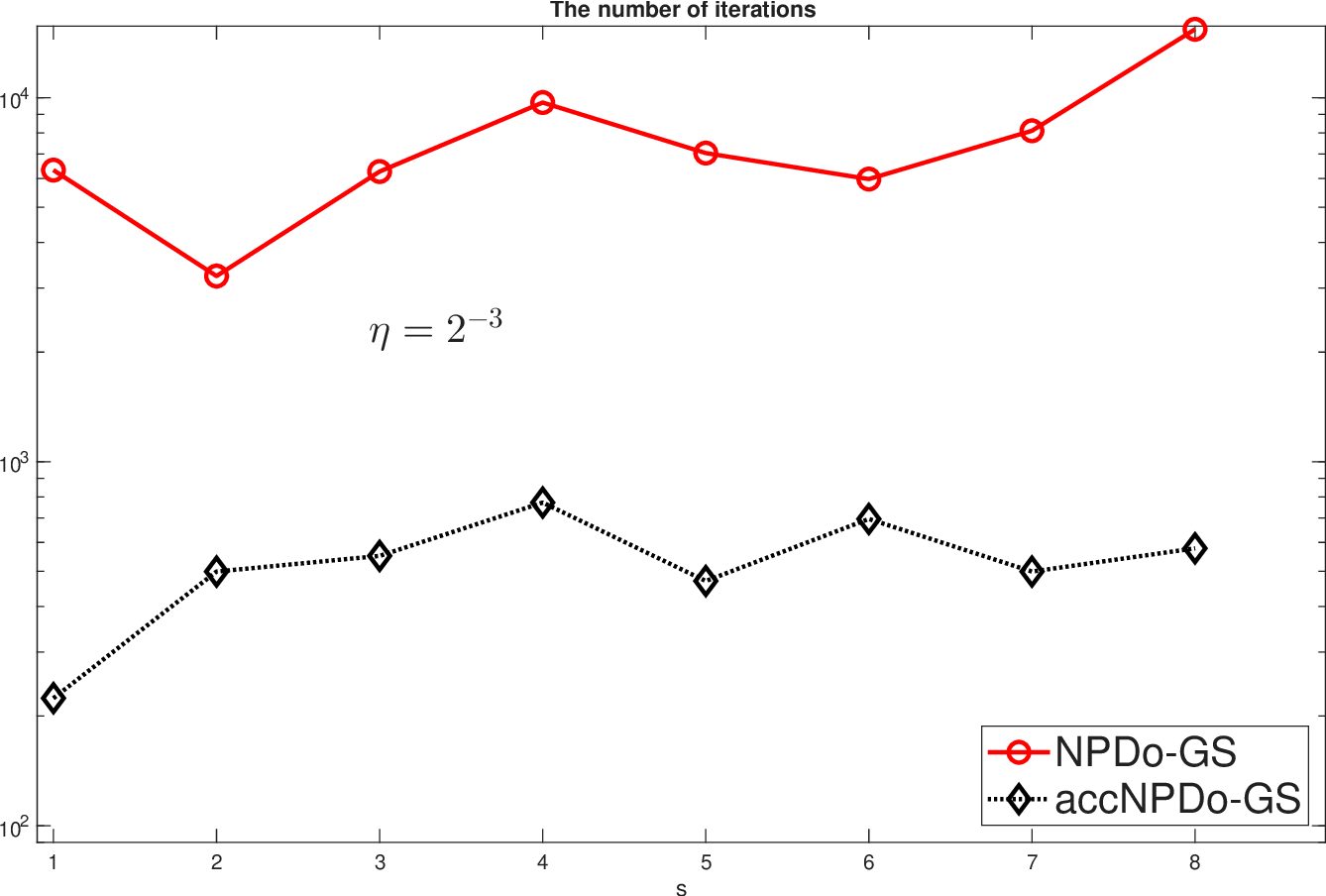}} \\
%
  \resizebox*{0.31\textwidth}{0.17\textheight}{\includegraphics{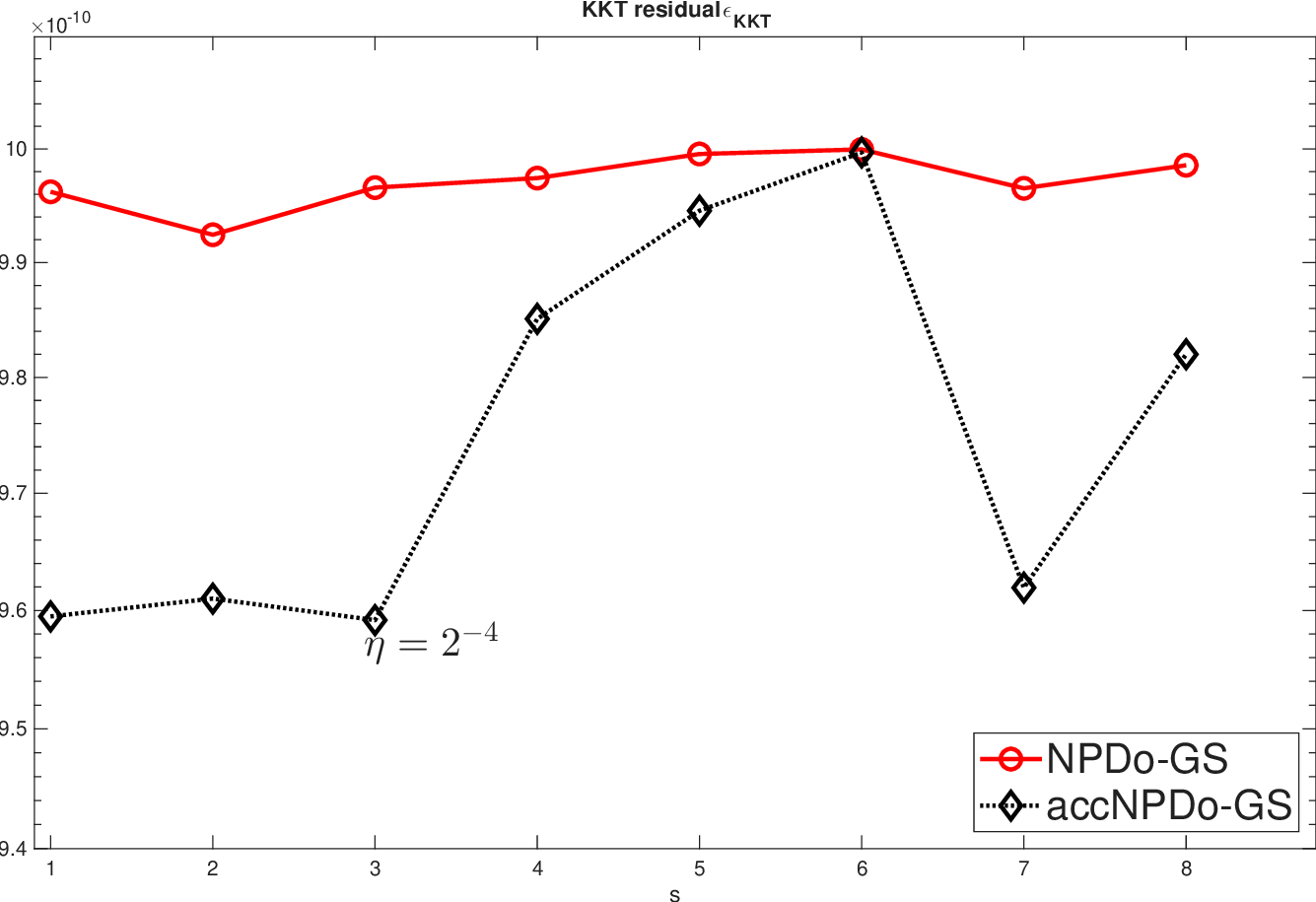}}
  &\resizebox*{0.31\textwidth}{0.17\textheight}{\includegraphics{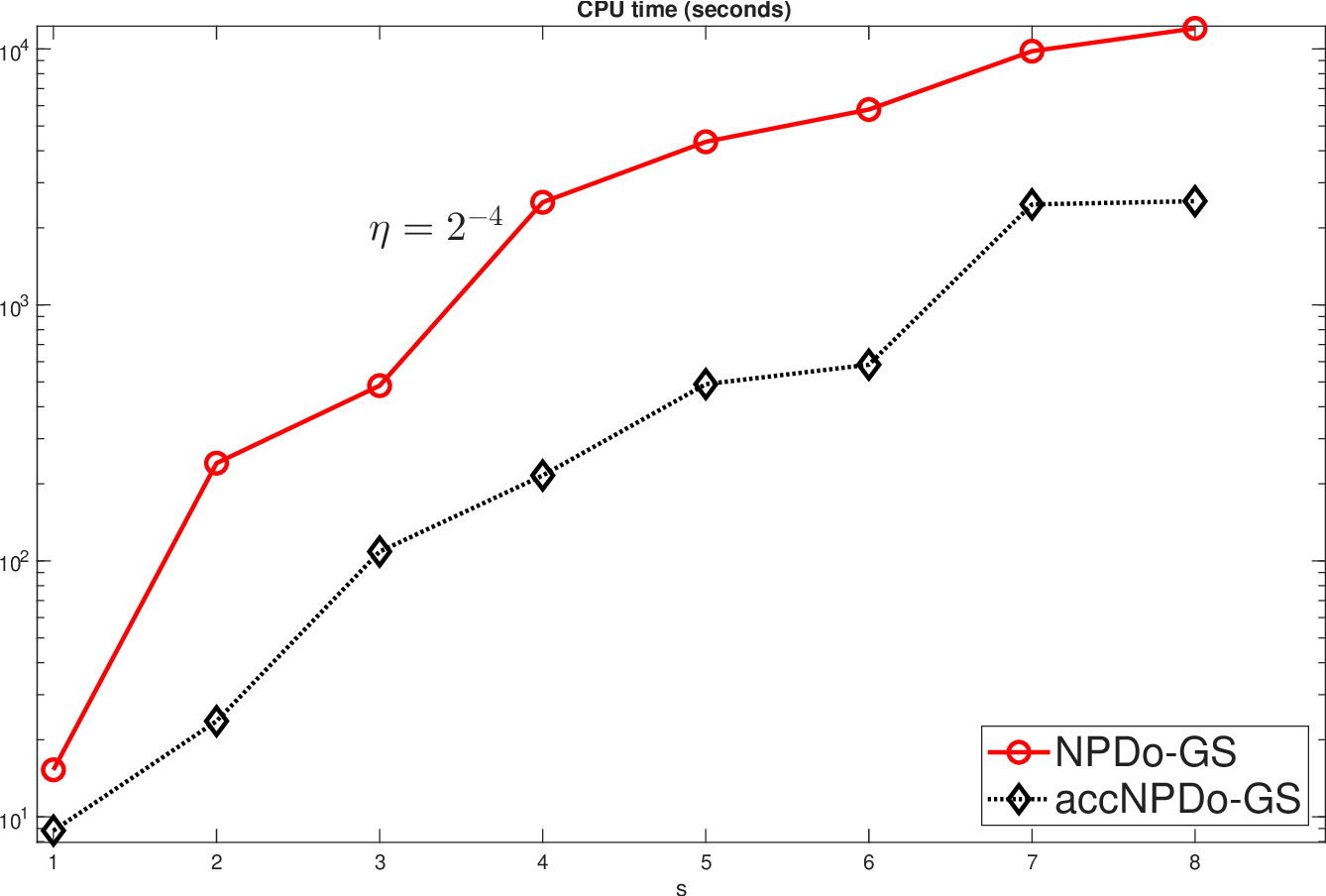}}
  &\resizebox*{0.31\textwidth}{0.17\textheight}{\includegraphics{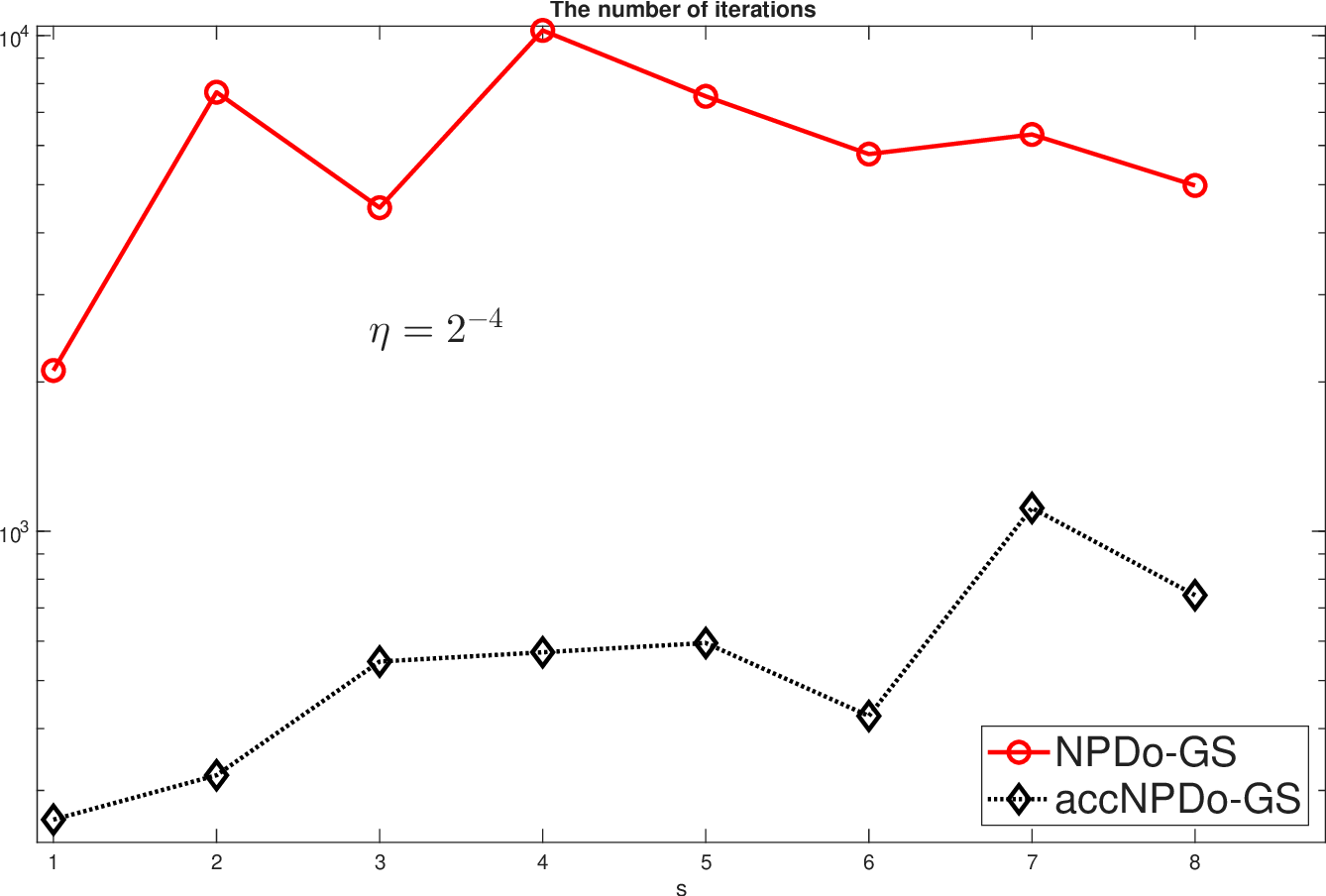}} \\
%
  \resizebox*{0.31\textwidth}{0.17\textheight}{\includegraphics{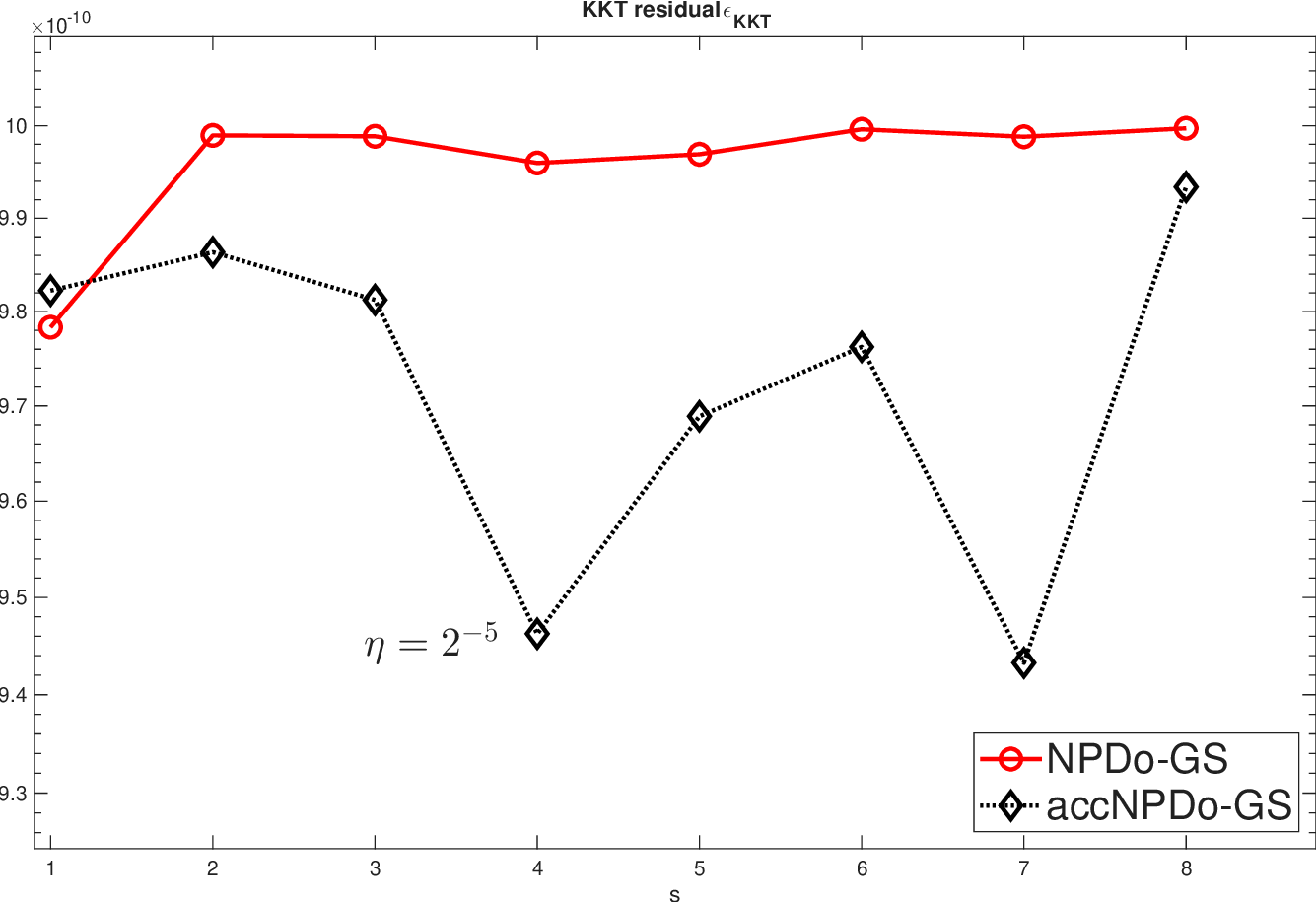}}
  &\resizebox*{0.31\textwidth}{0.17\textheight}{\includegraphics{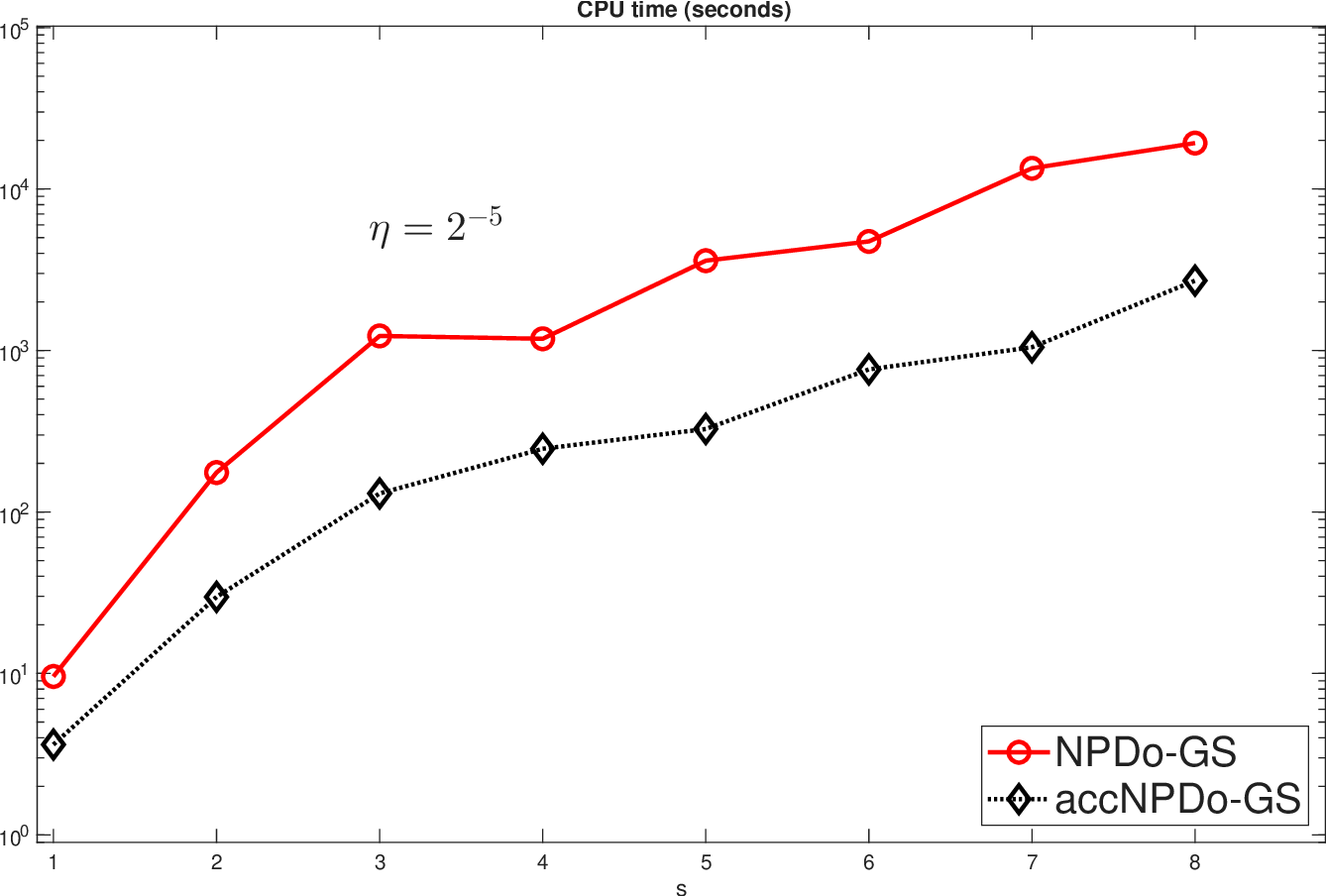}}
  &\resizebox*{0.31\textwidth}{0.17\textheight}{\includegraphics{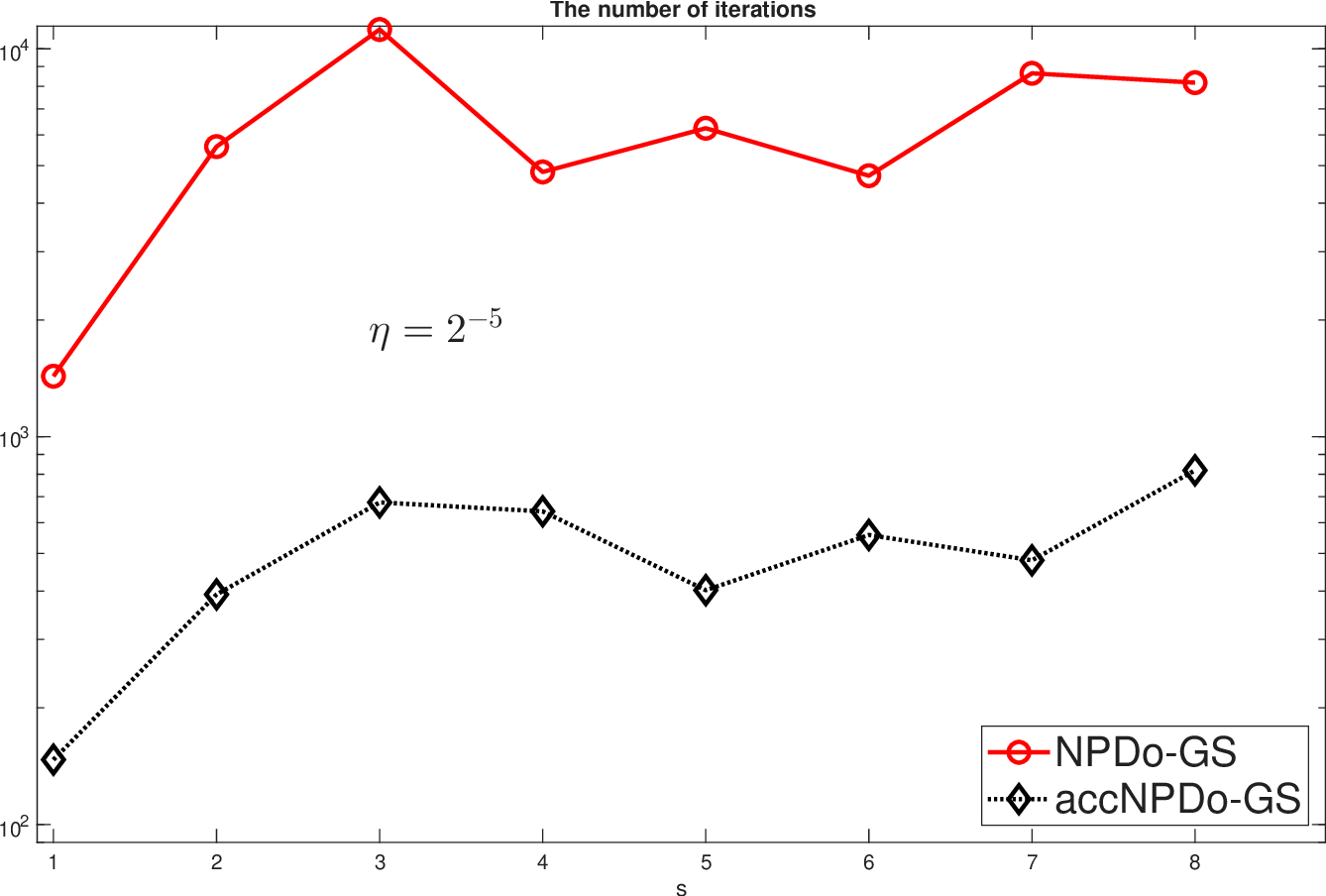}}
\end{tabular}\par
}
\vspace{-0.15 cm}
\caption{\small \ptsvd: scalability of NPDo and accNPDo on complex tensors with $[n_1,n_2,n_3]$ varies as in \eqref{eq:tensor-sizes} for $1\le s\le 8$ and $k=10$. {\em Left panel:\/} KKT residual $\tilde\epsilon_{\KKT,j}$,
    {\em Middle panel:\/} CPU time, and {\em Right panel:\/} the number of iterations, while
    $\eta=10^{-3}, 10^{-4}, 10^{-5}$ for the first, second, and third row, respectively.
  }
\label{fig:scale-cmpx-PTSVD}
\end{figure}

\begin{figure}[t]
{\centering
\begin{tabular}{ccc}
  \resizebox*{0.31\textwidth}{0.17\textheight}{\includegraphics{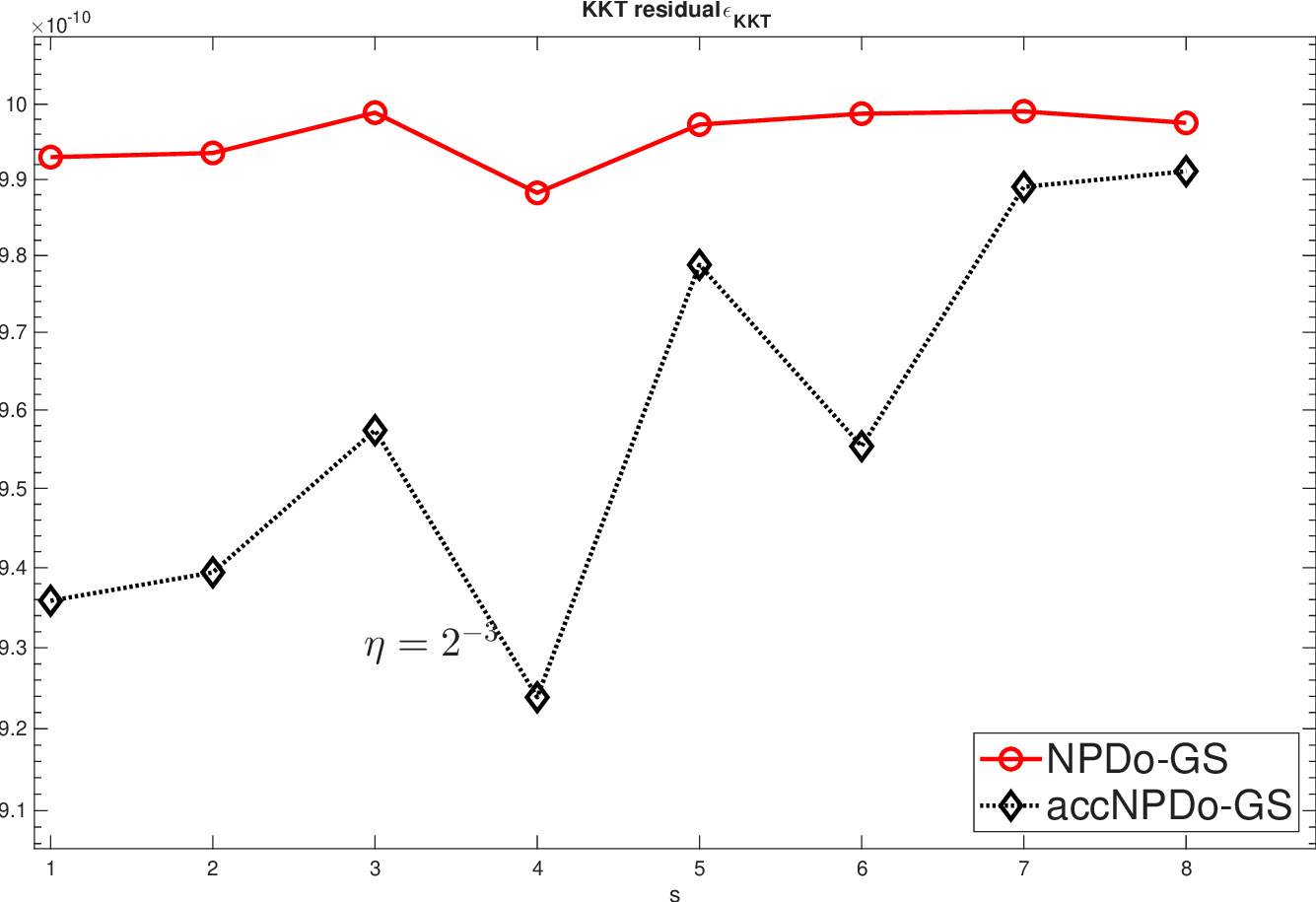}}
  &\resizebox*{0.31\textwidth}{0.17\textheight}{\includegraphics{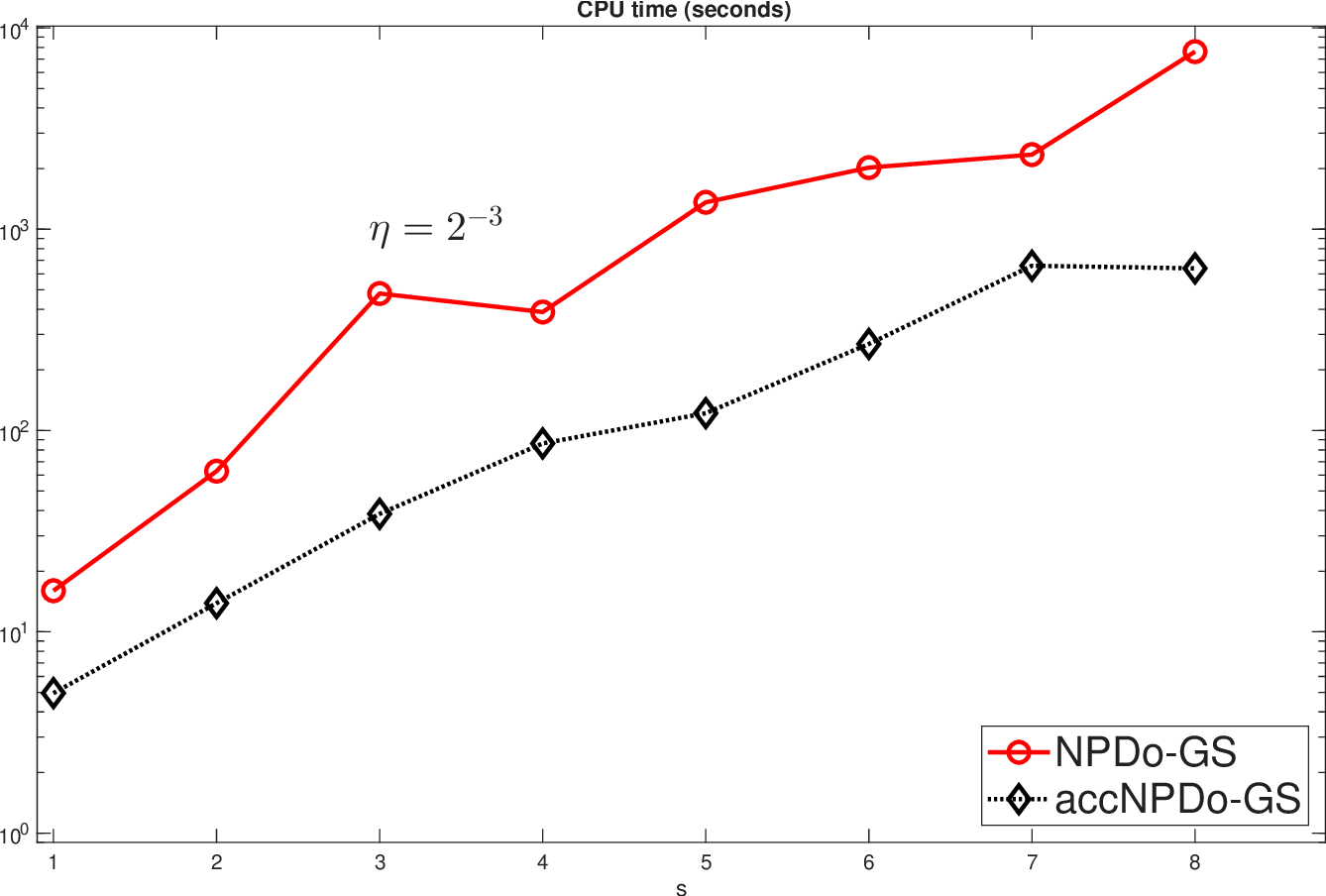}}
  &\resizebox*{0.31\textwidth}{0.17\textheight}{\includegraphics{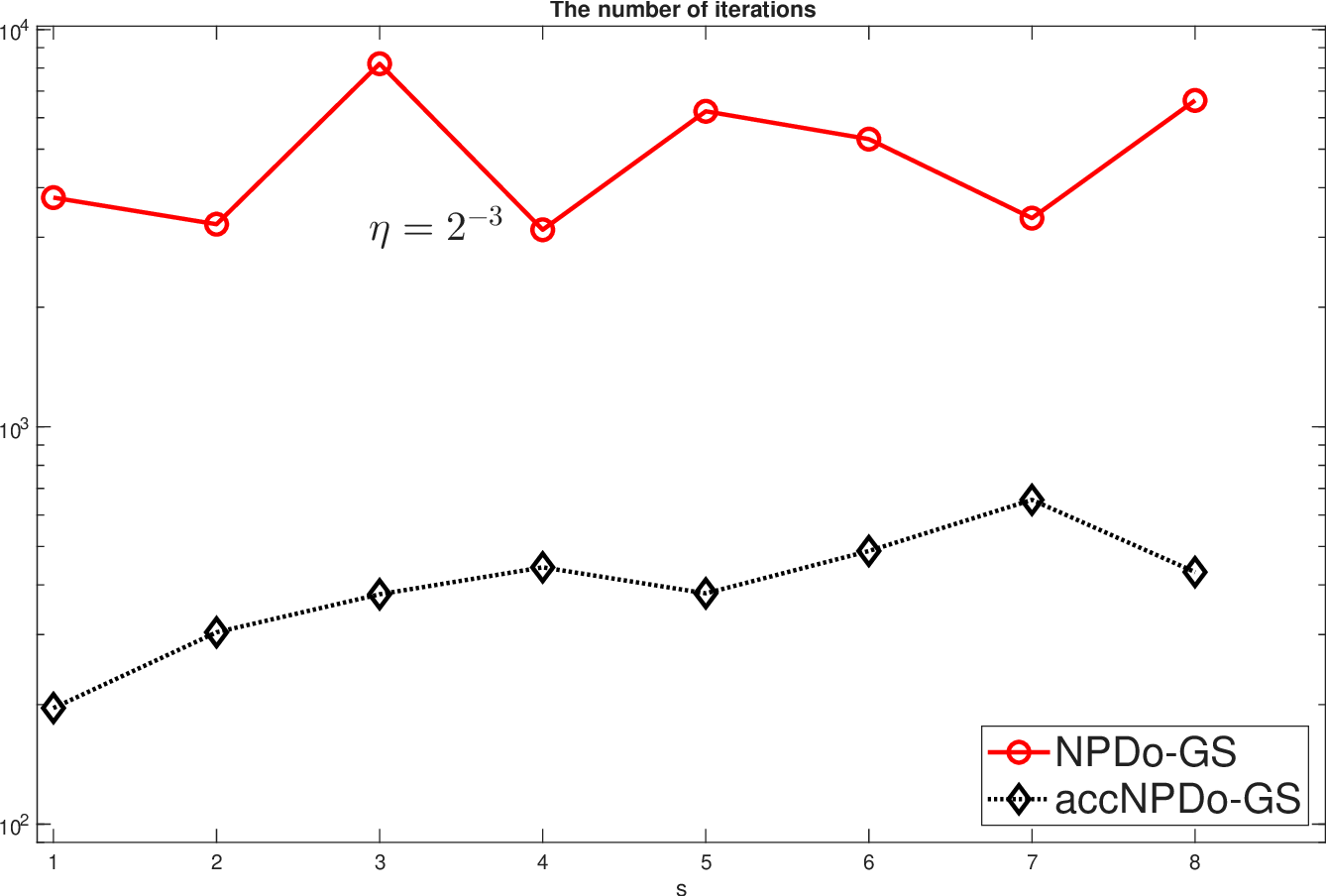}} \\
%
  \resizebox*{0.31\textwidth}{0.17\textheight}{\includegraphics{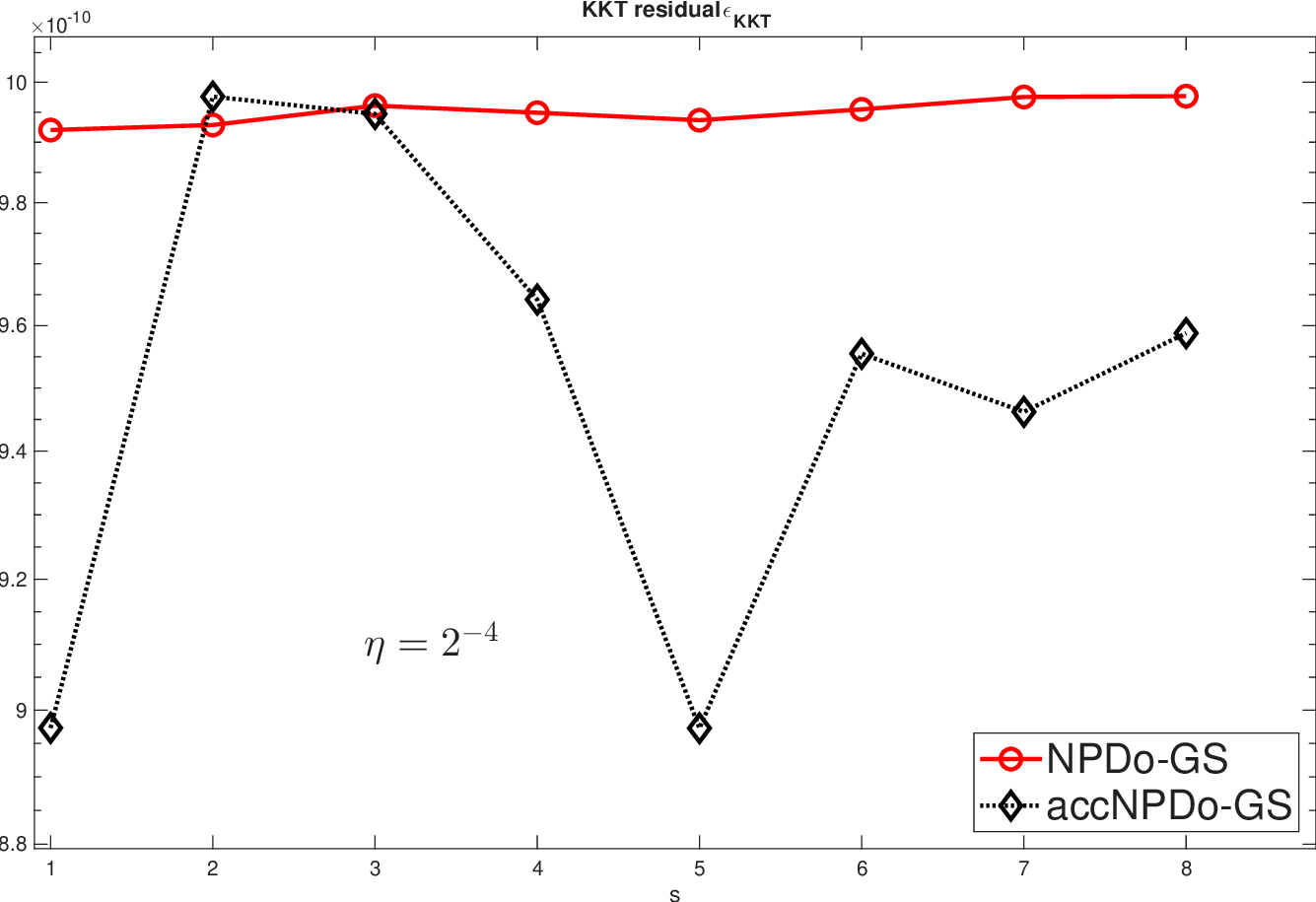}}
  &\resizebox*{0.31\textwidth}{0.17\textheight}{\includegraphics{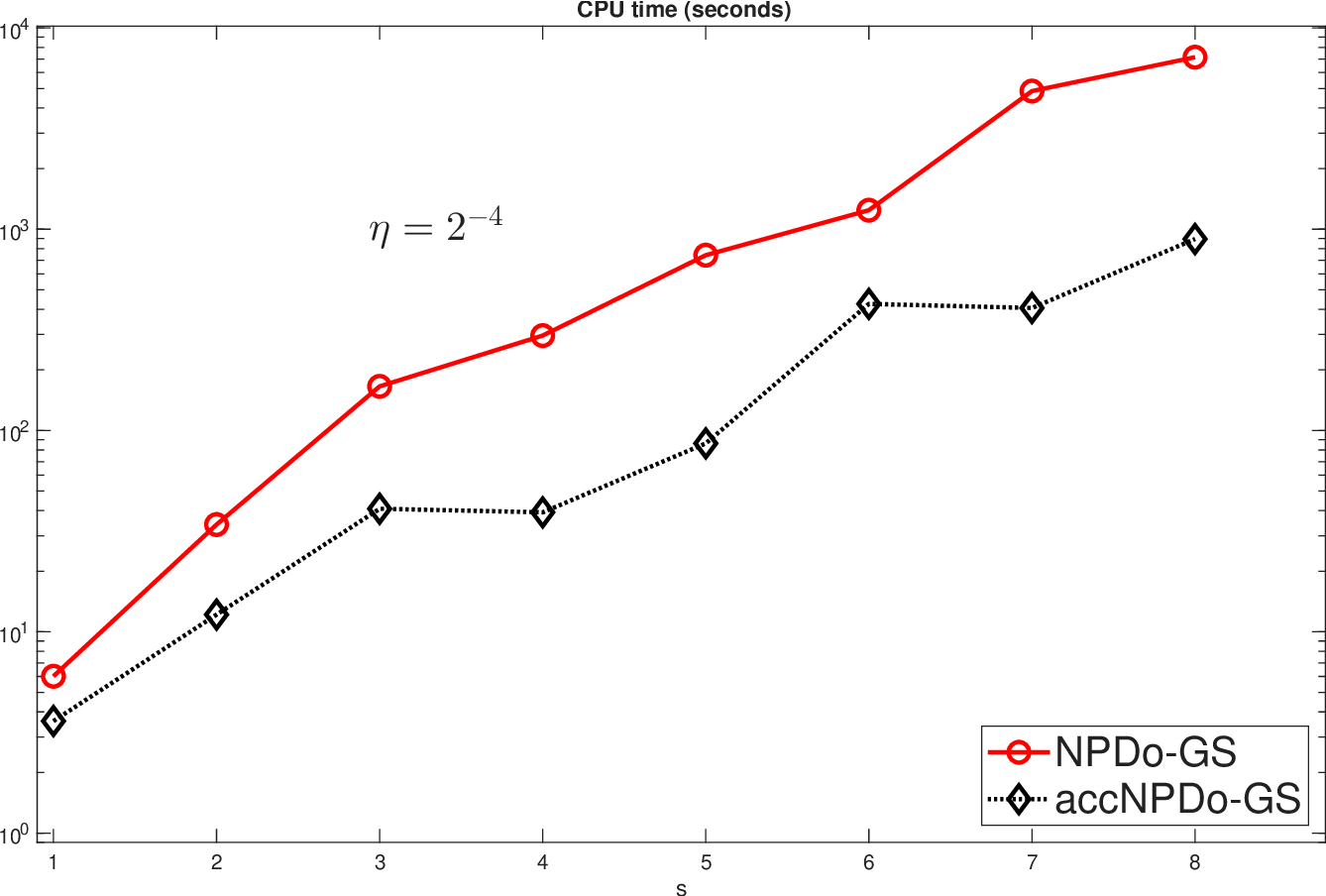}}
  &\resizebox*{0.31\textwidth}{0.17\textheight}{\includegraphics{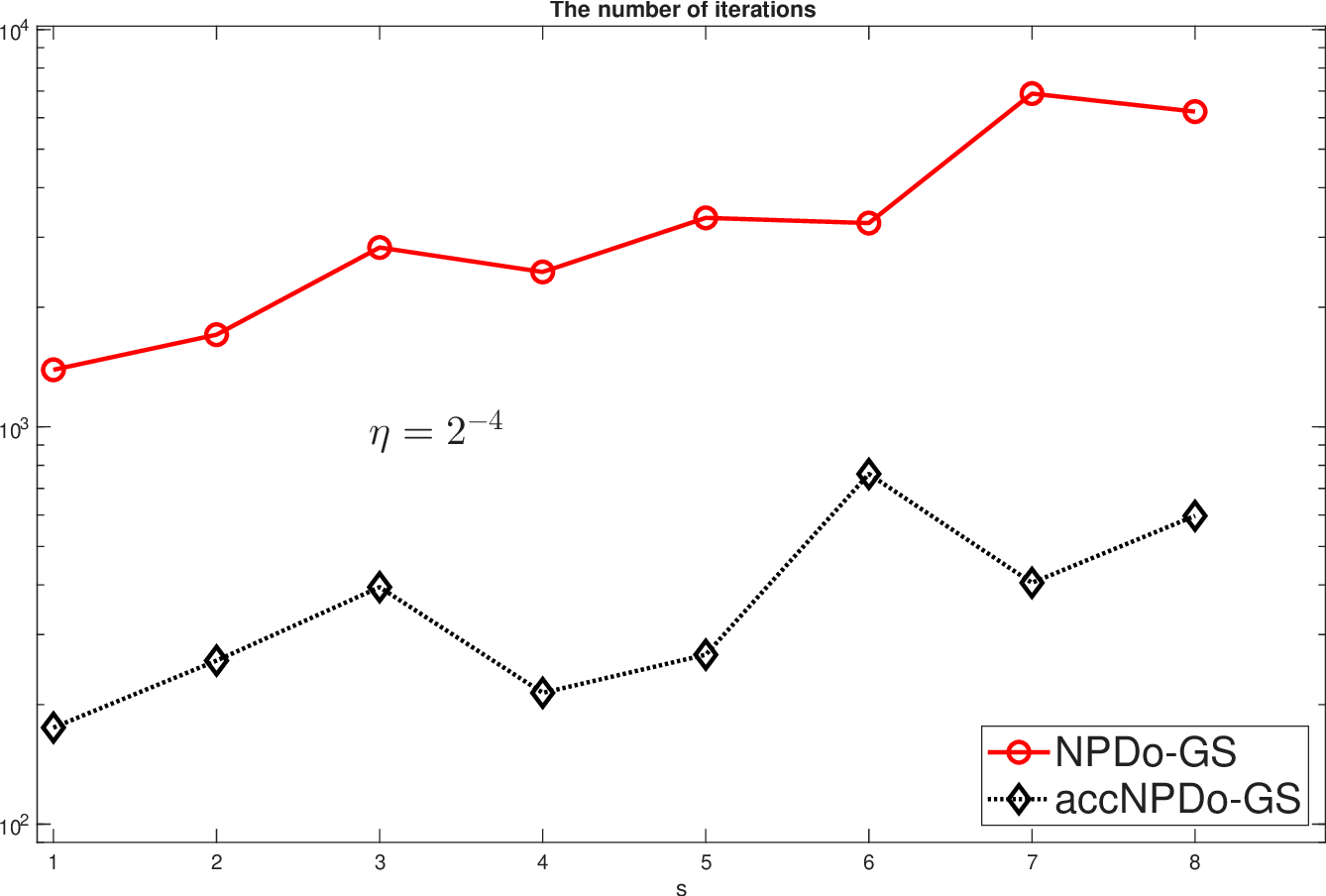}} \\
%
  \resizebox*{0.31\textwidth}{0.17\textheight}{\includegraphics{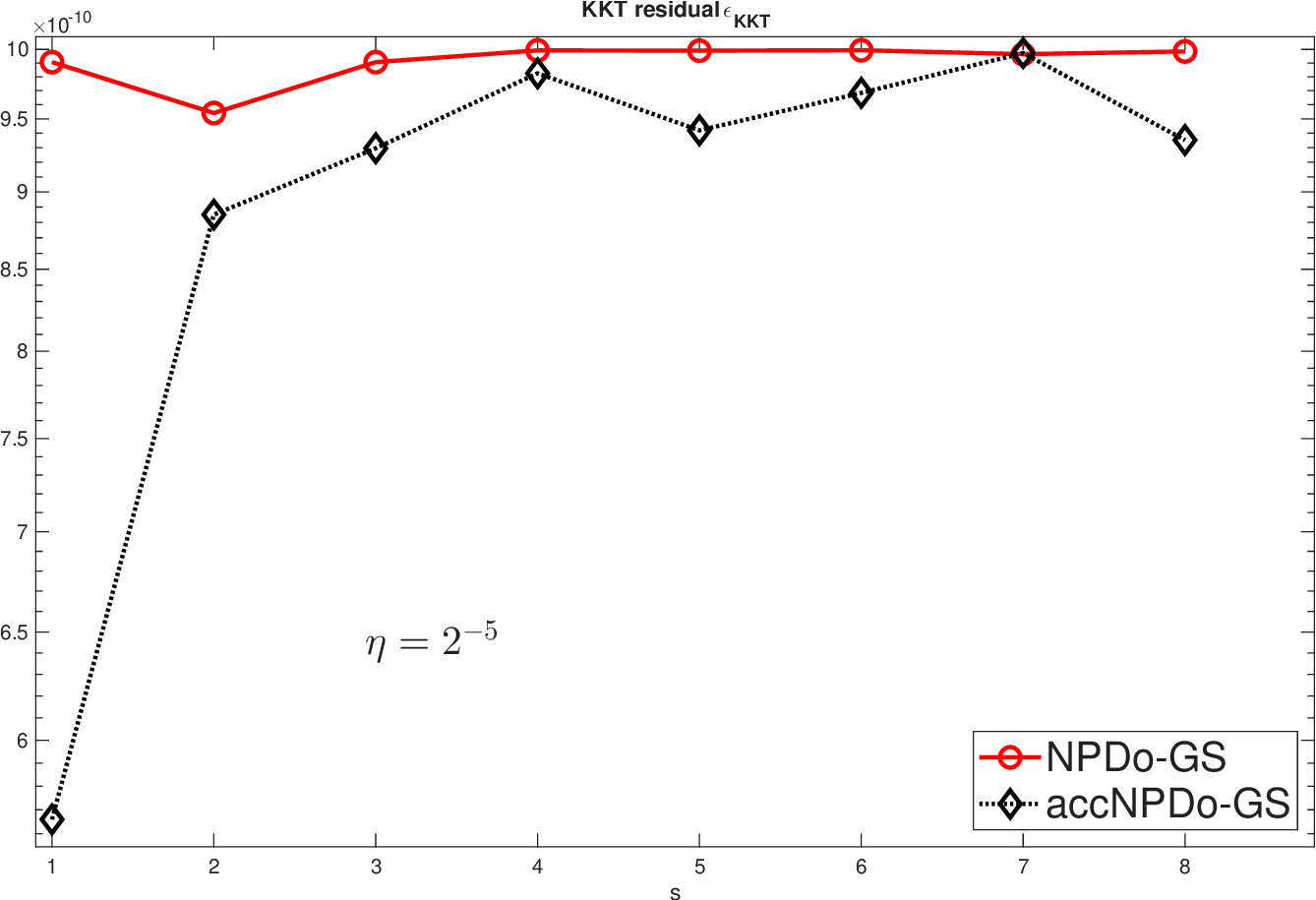}}
  &\resizebox*{0.31\textwidth}{0.17\textheight}{\includegraphics{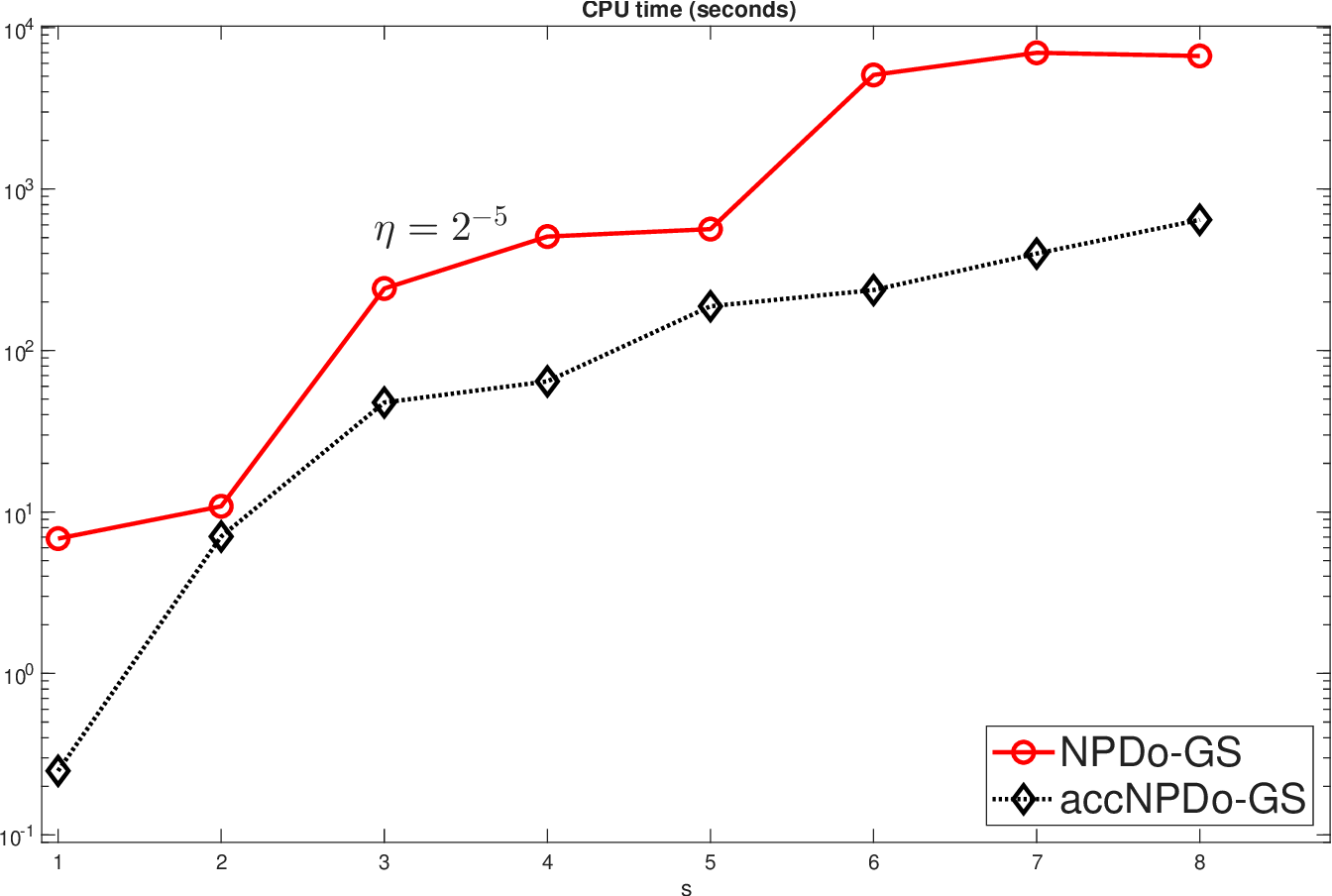}}
  &\resizebox*{0.31\textwidth}{0.17\textheight}{\includegraphics{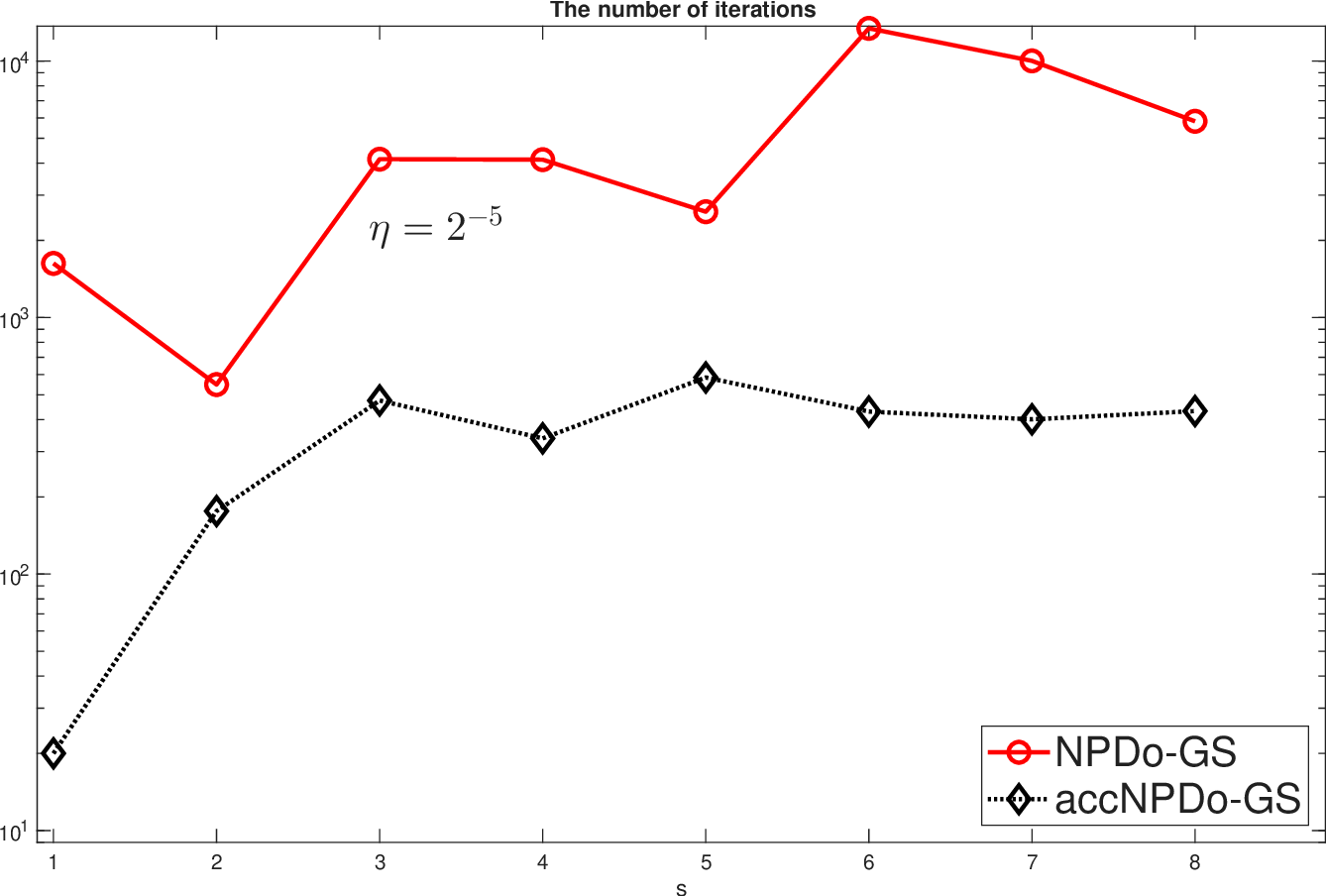}}
\end{tabular}\par
}
\vspace{-0.15 cm}
\caption{\small \ptbd: scalability of NPDo and accNPDo on real tensors with $[n_1,n_2,n_3]$ varies as in \eqref{eq:tensor-sizes}
    for $1\le s\le 8$, $[k_1,k_2,k_3]=[8,12,8]$, $t=4$, and each diagonal block $T_{jjj}$ is $2\times 3\times 2$. {\em Left panel:\/} KKT residual $\tilde\epsilon_{\KKT,j}$,
    {\em Middle panel:\/} CPU time, and {\em Right panel:\/} the number of iterations, while
    $\eta=10^{-3}, 10^{-4}, 10^{-5}$ for the first, second, and third row, respectively.
  }
\label{fig:scale-real-PTBD}
\end{figure}

\begin{figure}[t]
{\centering
\begin{tabular}{ccc}
  \resizebox*{0.31\textwidth}{0.17\textheight}{\includegraphics{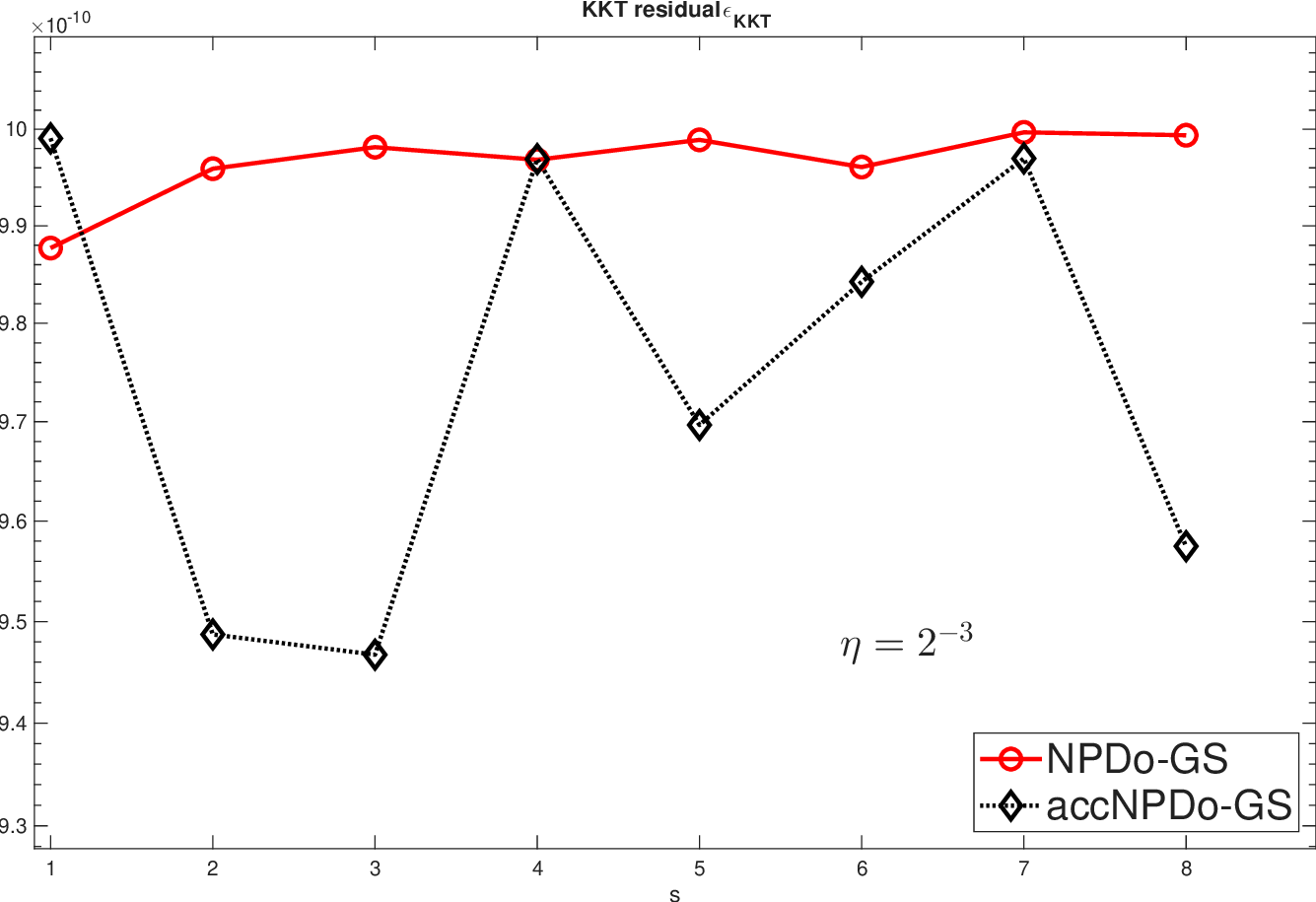}}
  &\resizebox*{0.31\textwidth}{0.17\textheight}{\includegraphics{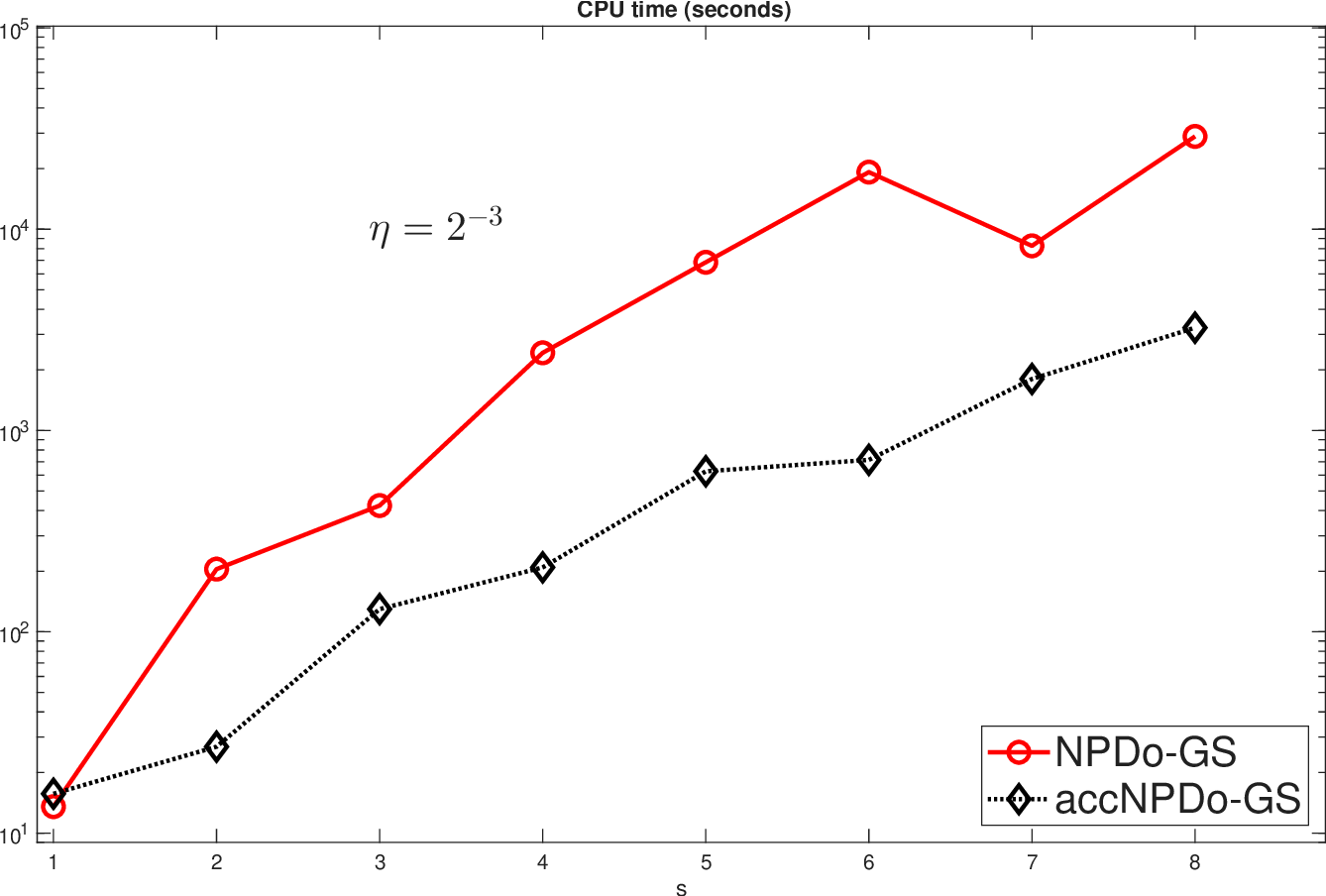}}
  &\resizebox*{0.31\textwidth}{0.17\textheight}{\includegraphics{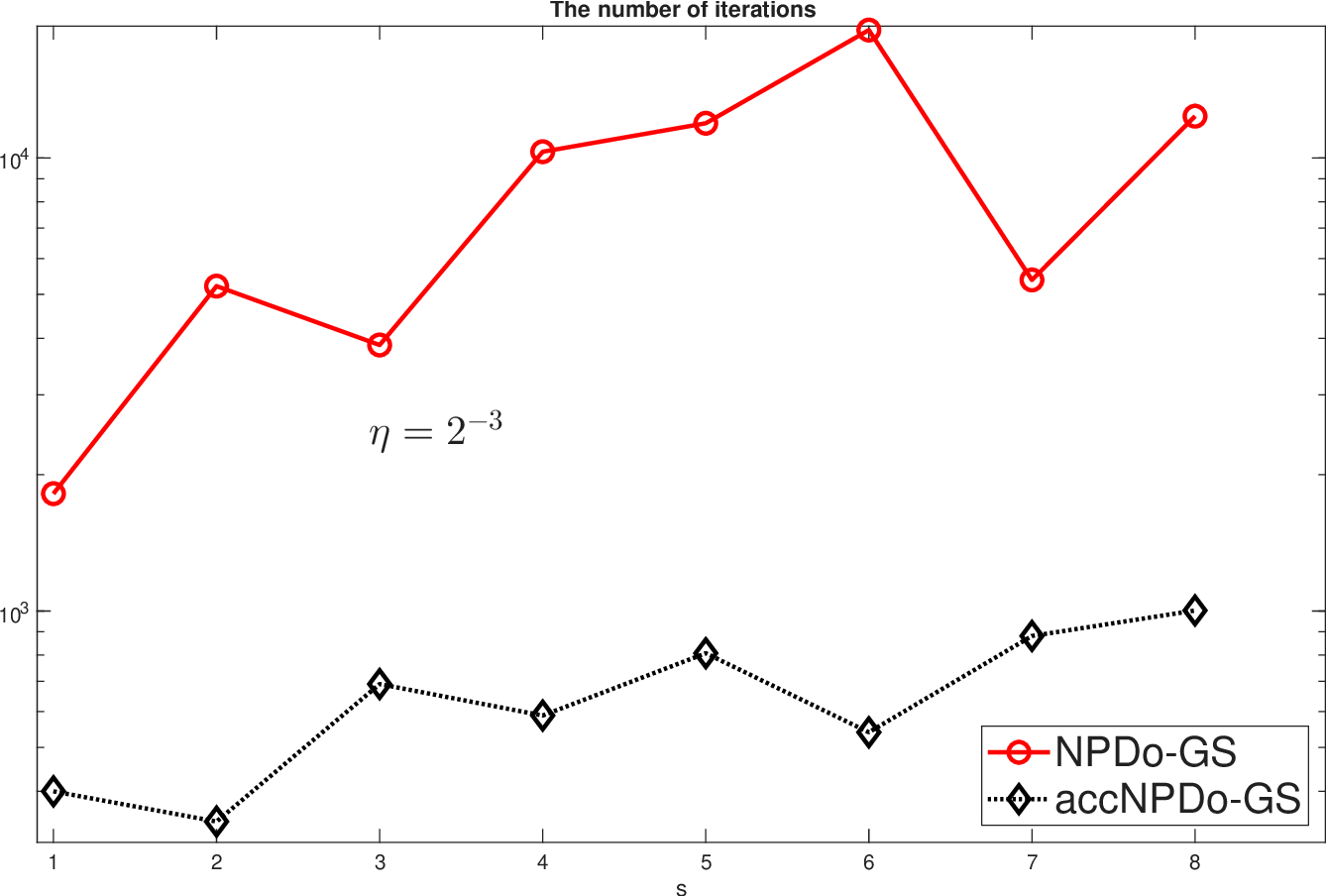}} \\
%
  \resizebox*{0.31\textwidth}{0.17\textheight}{\includegraphics{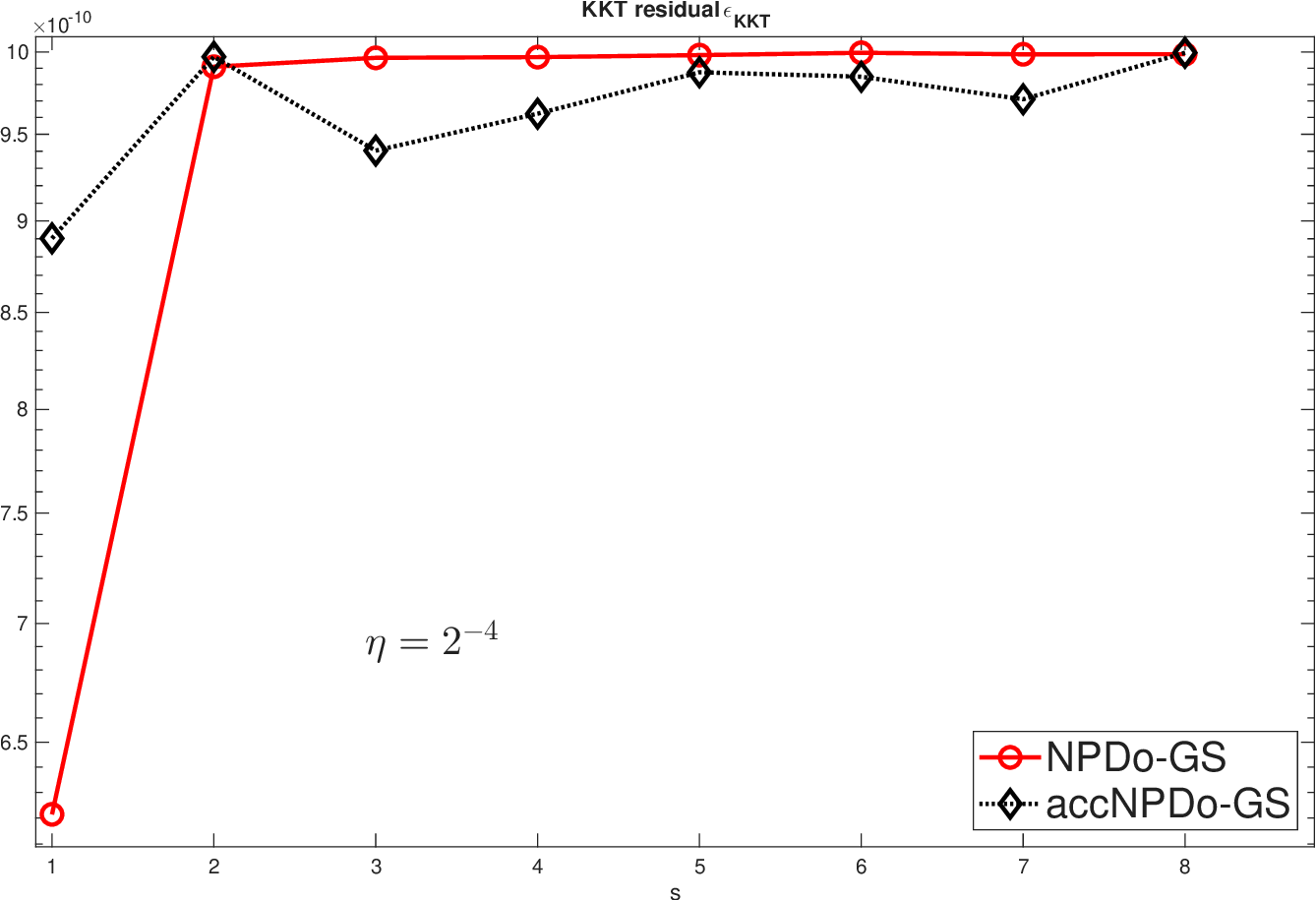}}
  &\resizebox*{0.31\textwidth}{0.17\textheight}{\includegraphics{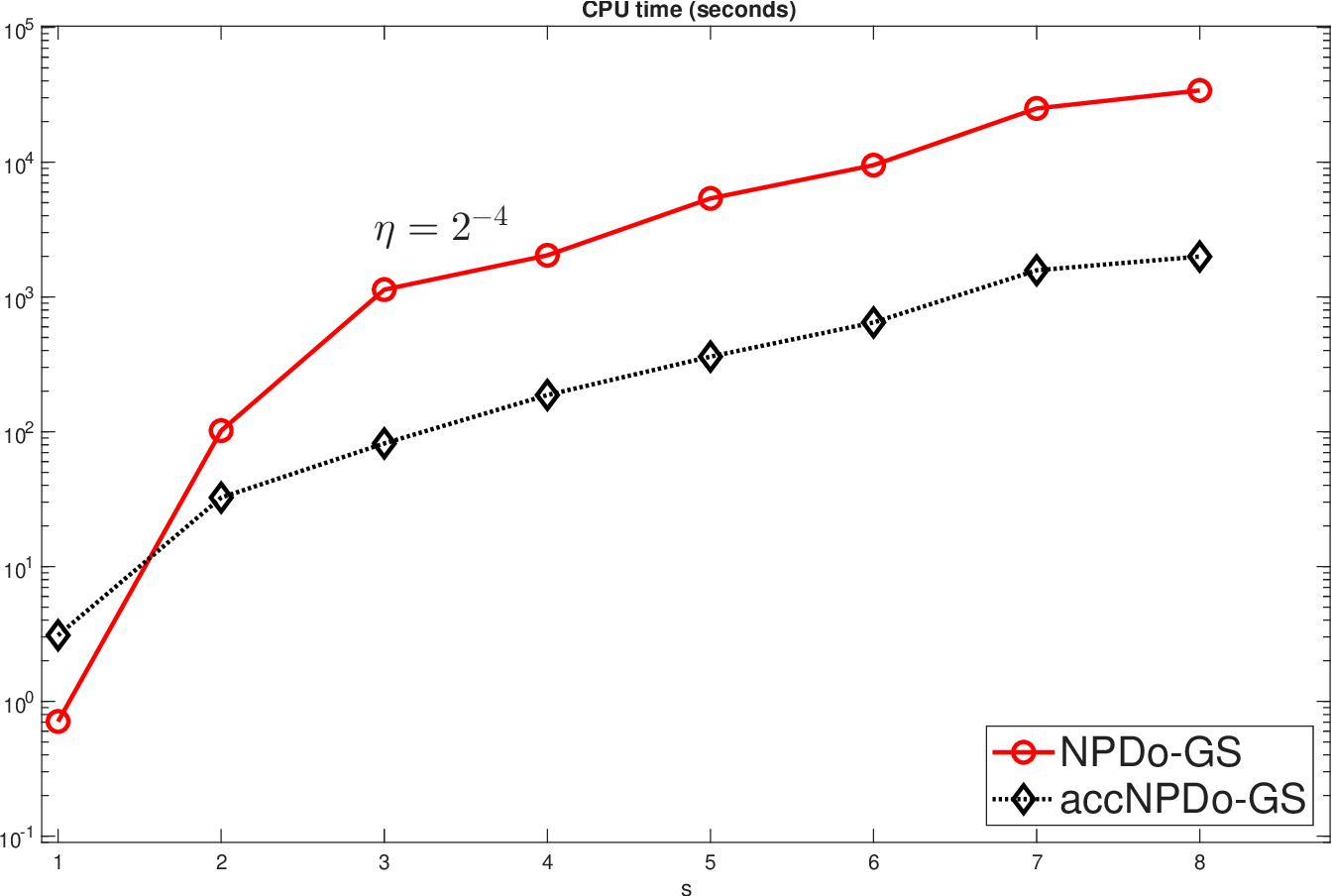}}
  &\resizebox*{0.31\textwidth}{0.17\textheight}{\includegraphics{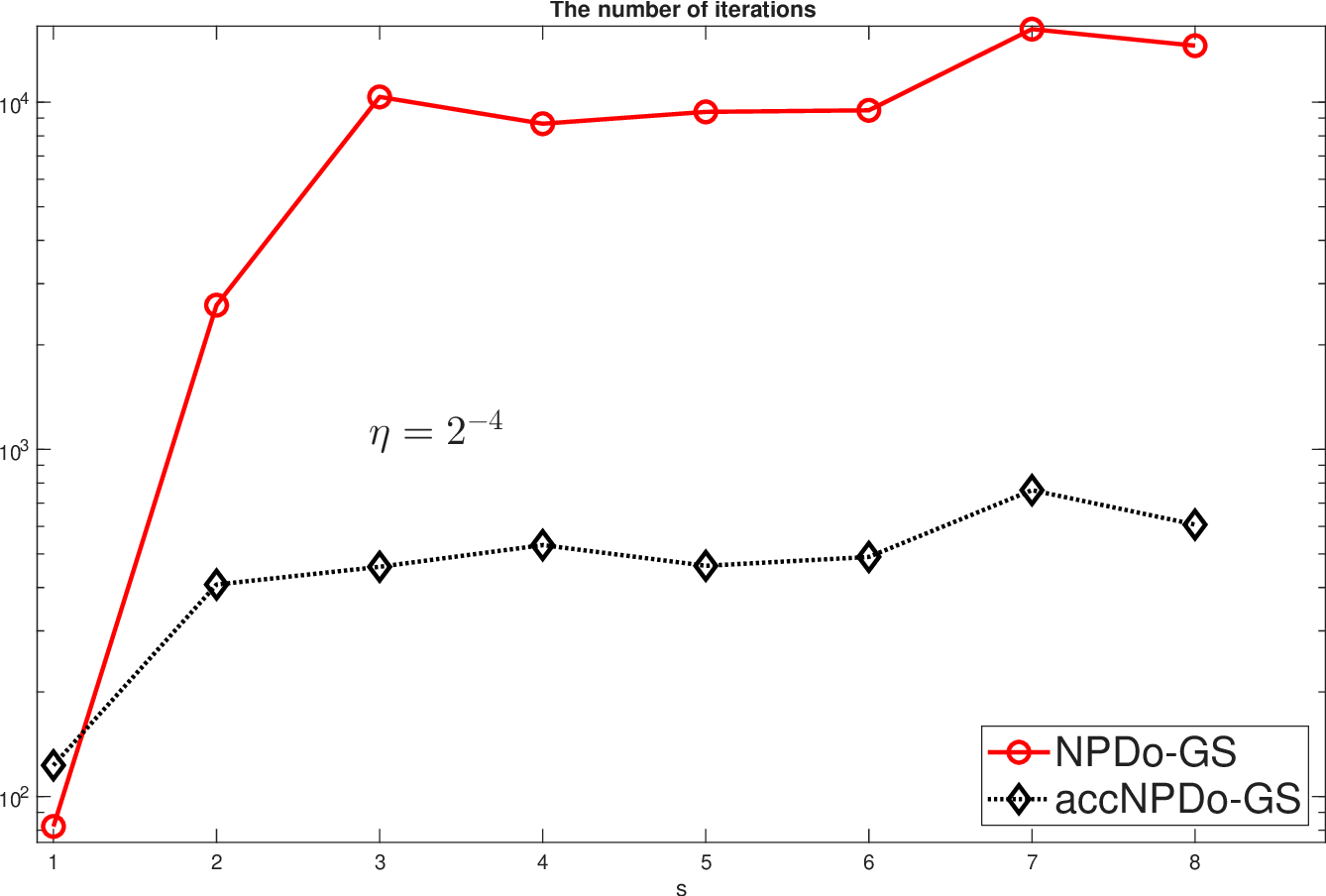}} \\
%
  \resizebox*{0.31\textwidth}{0.17\textheight}{\includegraphics{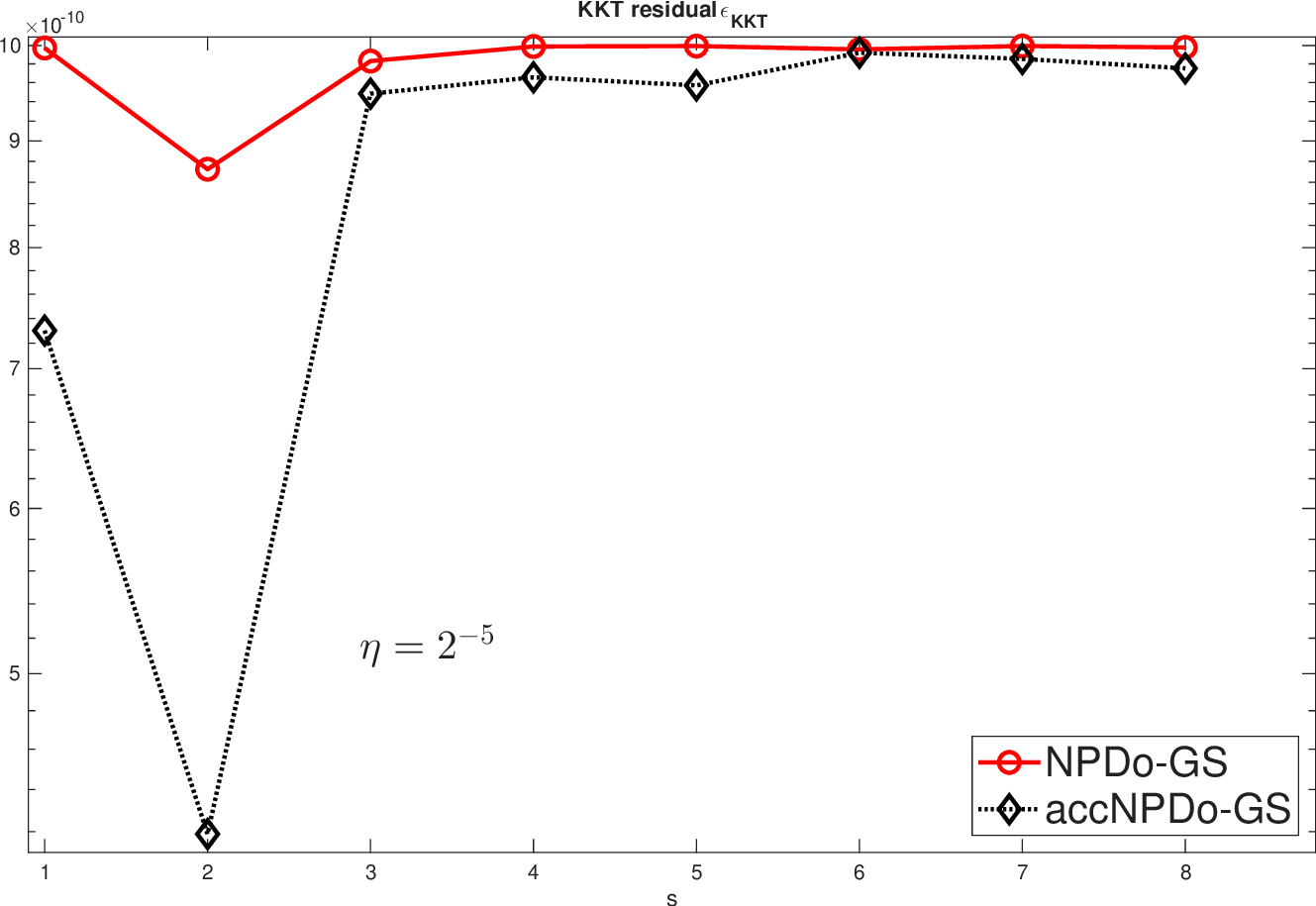}}
  &\resizebox*{0.31\textwidth}{0.17\textheight}{\includegraphics{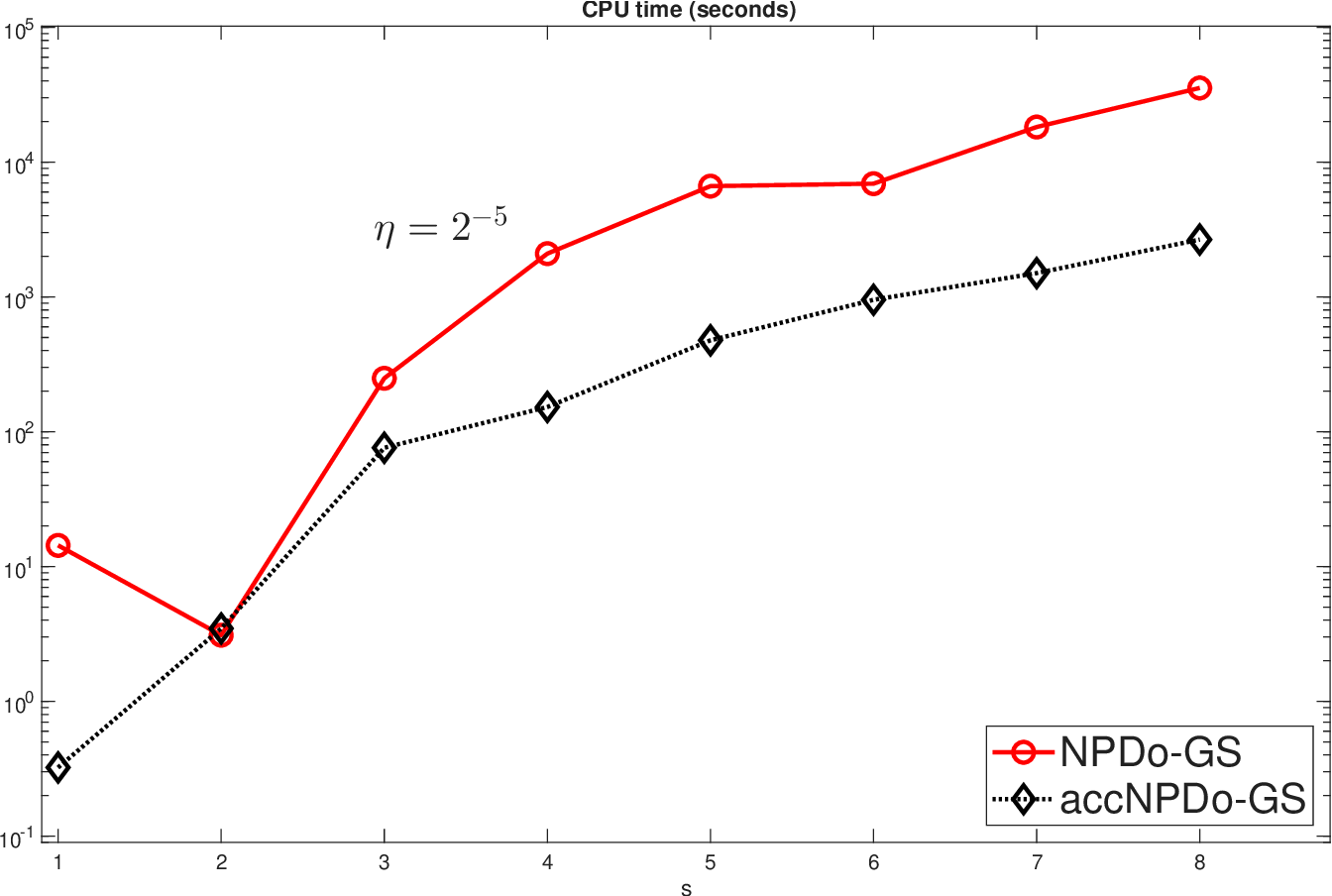}}
  &\resizebox*{0.31\textwidth}{0.17\textheight}{\includegraphics{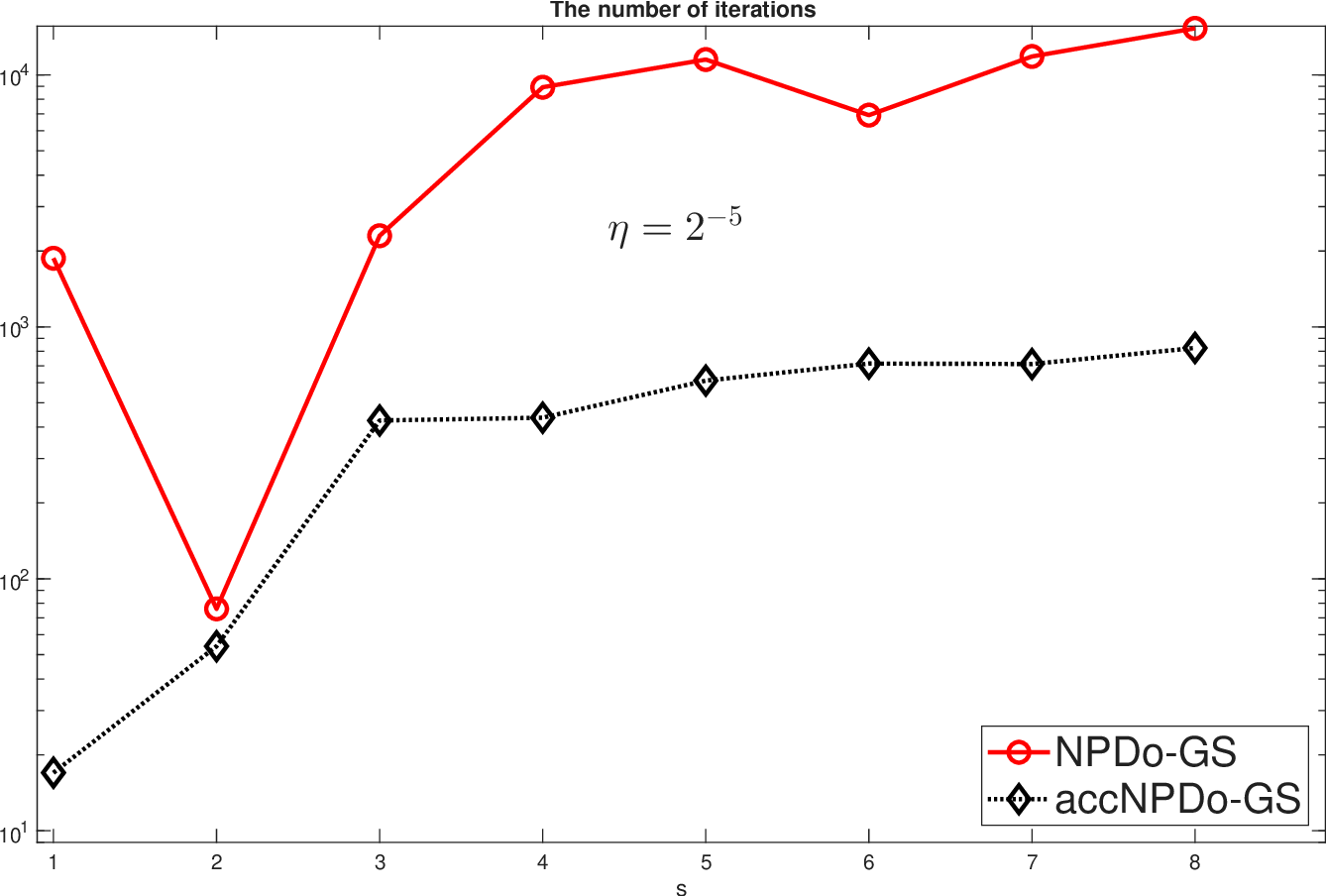}}
\end{tabular}\par
}
\vspace{-0.15 cm}
\caption{\small \ptbd: scalability of NPDo and accNPDo on complex tensors with $[n_1,n_2,n_3]$ varies as in \eqref{eq:tensor-sizes} for $1\le s\le 8$, $[k_1,k_2,k_3]=[8,12,8]$, $t=4$, and each diagonal block $T_{jjj}$ is $2\times 3\times 2$. {\em Left panel:\/} KKT residual $\tilde\epsilon_{\KKT,j}$,
    {\em Middle panel:\/} CPU time, and {\em Right panel:\/} the number of iterations.
  }
\label{fig:scale-cmpx-PTBD}
\end{figure}

\section{Conclusion}\label{sec:concl}
We are interested in the principal tensor block-diagonalization (\ptbd) of a multiway tensor.
Specifically, given $B\equiv[b_{i_1i_2\cdots i_m}]\in\bbC^{n_1\times n_2\times\cdots\times n_m}$, we seek $m$
orthonormal matrices $P_{\ell}\in\bbC^{n_{\ell}\times k_{\ell}}\,\,\forall\ell$ such that
$T=B\times_1 P_1^{\HH}\times_2 P_2^{\HH}\cdots\times_m P_m^{\HH}\in\bbC^{k_1\times k_2\times\cdots\times k_m}$ is optimally approximately block-diagonal with given block-diagonal structure in a dominant way.
It becomes the principal tensor SVD (\ptsvd) when each diagonal block of $T$ is scalar and
the principal Tucker decomposition when $T$ is considered as one block \Red{[????]}.
We propose an NPDo (nonlinear polar decomposition with orthonormal polar factor dependency)
approach for that purpose. It includes an alternating SCF (self-consistent-field) iteration to
numerically solve the KKT condition equations and a relevant theory.
It is shown the SCF iteration combined with Gauss-Seidel-type updating is globally convergent to a stationary point while the objective increases monotonically.
An accelerated version of the approach via the locally optimal conjugate gradient (LOCG) technique is also established.
Numerical experiments are presented to illustrate the efficiency of the NPDo approach.


\clearpage
{\small
\def\noopsort#1{}\def\l{\char32l}\def\v#1{{\accent20 #1}}
  \let\^^_=\v\def\hbk{hardback}\def\pbk{paperback}

}

\end{document}